\documentclass[english,a4paper,11pt]{scrartcl}
\usepackage{babel}
\usepackage{slantsc}
\usepackage{amsmath}
\setkomafont{subsection}{\usefont{T1}{fvm}{m}{n}}
\setkomafont{section}{\usefont{T1}{fvs}{b}{n}\Large}
\usepackage{baskervald}
\usepackage[T1]{fontenc}
\usepackage[usenames]{color}   
\definecolor{MyDarkBlue}{rgb}{0, 0.0, 0.45} 
\definecolor{MyDarkRed}{rgb}{0.45, 0.0, 0} 
\definecolor{MyDarkGreen}{rgb}{0, 0.45, 0} 
\definecolor{MyLightGray}{gray}{.90}
\definecolor{MyLightGreen}{rgb}{0.5, 0.99, 0.5}
\DefineNamedColor{named}{Peach}         {cmyk}{0,0.50,0.70,0}
\DefineNamedColor{named}{Cyan}          {cmyk}{1,0,0,0}
\usepackage{hyperref} 
\hypersetup{colorlinks, citecolor=MyDarkRed, filecolor=MyDarkBlue,
	linkcolor=MyDarkRed, urlcolor=MyDarkBlue}
\usepackage[in]{fullpage}

\usepackage{tcolorbox}

\usepackage{graphicx} % For including images
\usepackage[parfill]{parskip}    % Activate to begin paragraphs with an empty line rather than an indent
\usepackage{amsfonts, amscd, amssymb, amsthm, amsmath}
\usepackage{mathrsfs}
\usepackage{urwchancal}
\usepackage{xcolor}
\usepackage{MnSymbol}
\usepackage{pdfsync}%leaves makers for tex searching
\usepackage{enumerate}
\usepackage{enumitem} 
\usepackage[cal=boondox]{mathalfa}
\usepackage{dynkin-diagrams}
\usepackage{tikz-cd}
\usepackage{lipsum} % just for the example
\theoremstyle{plain}
\newtheorem{thm}{Theorem}[section]
\newtheorem*{theorem-non}{Theorem}
\newtheorem{lem}[thm]{Lemma}
\newtheorem{fact}[thm]{Fact}
\newtheorem{prop}[thm]{Proposition}
\newtheorem{cor}[thm]{Corollary}

\theoremstyle{definition}

\newtheorem{rem}{Remark}[section]

\def\fm{\mathfrak{m}}

\newcommand{\ov}[1]{\overline{#1}}
\def\u{$\bar{\mathrm{u}}$}

 \def\Q{\mathbb{Q}}\def\Z{\mathbb{Z}}
\def\le{\leqslant} \def\ge{\geqslant}

 \def\CC{\mathfrak{C}}

\def\CC{\mathbb{C}}
 \def\a{\mathfrak a}
\def\e{\varepsilon} \def\DD{\Delta} \def\om{\omega}

\def\a{\alpha}\def\b{\beta}
\def\+{\,+\,}  \def\={\;=\;}

\def\be{\begin{equation}}  \def\ee{\end{equation}}

\def\ss{\sigma}

\def\dd{\delta}
\def\dde{{\scalebox{1.1}{$\scriptscriptstyle \delta$}}}
\def\la{\langle} \def\ra{\rangle}

\def\xx{ \mathbf{ x} }

\def\db{ \mathbf{d} }

\def\qpoc#1#2{\left(#1; #2 \right)_\infty}

\def\lam{\lambda}\def\Lam{\Lambda}
\def\wW{\widetilde{W}}\def\wZ{\widetilde{Z}}

\def\ux{\underline{x}}
\def\ZW{Z_{\scriptscriptstyle W}}
\def\oo{\scalebox{1.2}{$\scriptscriptstyle\mathcal{o}, \mathcal{o}$}}
\def\eo{\scalebox{1.2}{$\scriptscriptstyle\mathcal{e}, \mathcal{o}$}}
\def\oe{\scalebox{1.2}{$\scriptscriptstyle\mathcal{o}, \mathcal{e}$}}
\def\eee{\scalebox{1.2}{$\scriptscriptstyle\mathcal{e}, \mathcal{e}$}}
\def\ZZ{\mathbf{Z}}
\def\tZZ{\mathbf{\tilde{Z}}}
\def\End{\operatorname{End}}
\def\tZW{\tilde{Z}_{\scriptscriptstyle W}}
\def\tZ{\tilde{Z}}
\def\tL{\tilde{\Lambda}}\def\tB{\tilde{B}}\def\tA{\tilde{A}}
\def\sss{\scriptscriptstyle}
\def\sZ{\mathscr{Z}}\def\sW{\mathscr{W}}
\def\ec{\scalebox{1.2}{$\scriptscriptstyle\mathcal{e}$}}
\def\oc{\scalebox{1.2}{$\scriptscriptstyle\mathcal{o}$}}
\def\eval{\vert_{x_{\scalebox{.75}{$\scriptscriptstyle 5$}}\,=\, 
  \zeta^{\scalebox{1.0}{$\scriptscriptstyle -1$}}(u
 \mathbf{x}^{\alpha\scalebox{.9}{$\scriptscriptstyle '$} })^{\scalebox{1.0}{$\scriptscriptstyle -1$}\slash n}}}
 \def\UU{\mathcal{U}}
 \def\vr{\varrho}
 \def\u1{\underline{1}}
  \def\vC{\mathbf{C}}
  \def\ZWg{Z_{\scriptscriptstyle W,g}}
\usepackage{listings}           % For source code fragments, many options
\pagecolor{white}		% For latex2html image backgrounds
\begin{document}

%\tableofcontents                % This can go after the abstract, too.
%\begin{titlepage}              % Uncomment for a separate title page
\title{Quadratic {W}eyl group multiple {D}irichlet series of Type $D_{\scriptscriptstyle 4}^{\scriptscriptstyle (1)}$}
\author{Adrian Diaconu, Vicen\c tiu Pa\c sol, and Alexandru A. Popa}

\newcommand{\Addresses}{{% additional braces for segregating \footnotesize
  \bigskip
  \footnotesize
  Adrian Diaconu \\
\vskip-18pt
\hskip-10.13pt\textsc{\hskip20pt University of Minnesota\\
 \vskip-18pt
 \hskip10pt School of Mathematics, 127 Vincent Hall\\
 \vskip-18pt
 \hskip10pt 206 Church St. \!SE, Minneapolis, MN 55455, USA;\\
 \vskip-18pt
 \hskip10pt Institute of Mathematics of the Romanian Academy\\
 \vskip-18pt
 \hskip10pt  P.O. Box 1-764, Bucharest RO-70700, Romania}\par\nopagebreak
\vskip-7.5pt
\hskip10pt Email: \!\textbf{cad@umn.edu}

\medskip

Vicen\c tiu Pa\c sol \\
\vskip-18pt
\hskip-10.13pt\textsc{\hskip20pt Institute of Mathematics of the Romanian Academy\\
 \vskip-18pt
 \hskip10pt  P.O. Box 1-764, Bucharest RO-70700, Romania}\par\nopagebreak
\vskip-7.5pt
\hskip10pt Email: \!\textbf{vpasol@gmail.com}

\medskip

Alexandru A. Popa \\
\vskip-18pt
\hskip-10.13pt\textsc{\hskip20pt Institute of Mathematics of the Romanian Academy\\
 \vskip-18pt
 \hskip10pt  P.O. Box 1-764, Bucharest RO-70700, Romania}\par\nopagebreak
\vskip-10.5pt
\hskip10pt Email: \!\textbf{aapopa@gmail.com}}}

\date{}
\maketitle
%\end{titlepage}

\begin{abstract}
  \noindent
  In this paper and its sequel \cite{DIPP}, we investigate the precise
relationship between the quadratic affine {W}eyl group multiple
{D}irichlet series in the sense of \cite{CG1, BD}, and those defined
axiomatically by {W}hitehead \cite{White2} and \cite{White1}. In
particular, we show that the axiomatic quadratic {W}eyl group multiple
{D}irichlet series of type
$
D_{\scalebox{1.1}{$\scriptscriptstyle 4$}}^{\scalebox{1.1}{$\scriptscriptstyle (1)$}}
$
\!over rational function fields of odd characteristic admits meromorphic
continuation to the interior of the corresponding complexified {T}its
cone. \!We shall also determine the polar divisor of this function,
and compute the residue at\linebreak each of its poles. As a
consequence, we obtain an {\normalfont\itshape exact} formula for a
weighted 4-th moment of\linebreak quadratic {D}irichlet $L$-functions
over rational function fields; we shall also derive an asymptotic
formula for this weighted moment that is expected to generalize to any
global field.
\end{abstract}

\tableofcontents

\section{Introduction} %\label{intro}
In this paper and its sequel \cite{DIPP}, we investigate the precise
relationship between the quadratic affine {W}eyl group multiple
{D}irichlet series in the sense of \cite{CG1, BD}, and those defined
axiomatically by {W}hitehead \cite{White2} and \cite{White1}. In
particular, we show that the axiomatic quadratic {W}eyl group multiple
{D}irichlet series of type
$
D_{\scalebox{1.1}{$\scriptscriptstyle 4$}}^{\scalebox{1.1}{$\scriptscriptstyle (1)$}}
$
over rational function fields of odd characteristic admits meromorphic
continuation to the interior of the corresponding complexified {T}its
cone. We shall also determine the polar divisor of this function,
and\linebreak compute the residue at each of its poles. As a
consequence, we obtain an {\normalfont\itshape exact} formula for a
weighted 4-th\linebreak moment of quadratic {D}irichlet $L$-functions
over rational function fields; we shall also derive an asymptotic
formula for this weighted moment that is expected to generalize to any
global field.

Before discussing our results in more detail, let us first recall some
basic facts about {W}eyl group multiple {D}irichlet series associated
to classical root systems.

A {\normalfont\itshape {W}eyl group multiple {D}irichlet series,}
WMDS for short, is a {D}irichlet series in several complex variables
attached to a root system satisfying a group of functional equations
isomorphic to the {W}eyl group of the root system; the number
of variables is precisely the rank of the root system. These objects
were initially\linebreak introduced on a case-by-case basis in the 1980's, when
the idea emerged that it could be useful to tie together a family of
related $L$-functions in one variable (e.g., the family of quadratic
{D}irichlet $L$-functions) to create a {\normalfont\itshape double
  {D}irichlet series,} which could be used to study the average
behavior of the original family\linebreak of $L$-functions; see
\cite{GH, BFH0, FHL, BFH2, Fi-Fr, Diaconu}. Motivated by the problem
of understanding the asymptotics of moments of $L$-functions,
double {D}irichlet series soon became multiple {D}irichlet series
(see \cite{BFH1} and \cite{DGH}),\linebreak and since approximately
2003, the idea of how to construct WMDS attached to classical root
systems took\linebreak shape; see
\cite{BBCFH, BBFH, BBF2, BBF3, CFG, CG1, CG2}. It became clear that
a {\normalfont\itshape quadratic} WMDS over a global field,
for exam-\linebreak ple, associated to a (finite) reduced irreducible
root system
$
\Phi_{\scalebox{1.0}{$\scriptscriptstyle 0$}}
$
of rank $r$ has the form
\[
\sum_{n_{\scalebox{.75}{$\scriptscriptstyle 1$}},
  \ldots, n_{\scalebox{.85}{$\scriptscriptstyle r$}}}  
\frac{H(n_{\scriptscriptstyle 1}, \ldots,
n_{\scalebox{1.2}{$\scriptscriptstyle r$}};
m_{\scriptscriptstyle 1}, \ldots,
m_{\scalebox{1.2}{$\scriptscriptstyle r$}})
\Psi(n_{\scriptscriptstyle 1}, \ldots,
n_{\scalebox{1.2}{$\scriptscriptstyle r$}})}
{\prod |n_{\scriptscriptstyle i}|^{s_{\scalebox{.75}{$\scriptscriptstyle i$}}}}
\]
the sum being over $r$-tuples of representatives of non-zero
$S$-integers modulo $S$-units in the number field setting
(or $r$-tuples of effective divisors prime to $S$ in the function
field setting), for $S$ a large enough finite\linebreak set of places;
see, e.g., \cite{CG2, Fi-Fr}. The $r$-tuple
$
(m_{\scriptscriptstyle 1}, \ldots,
m_{\scalebox{1.2}{$\scriptscriptstyle r$}})
$
of non-zero $S$-integers (resp. \!effective divisors\linebreak
prime to $S$) is a twisting parameter, and $|\cdot|$ denotes
the corresponding norm of an element. The function $H$ is the
most important part, giving the structure of the series, and
the function $\Psi$ is just a technical de-\linebreak vice, varying over a
certain finite-dimensional vector space of complex-valued functions,
that makes the product $H\Psi$ well-defined.

The coefficients $H$ satisfy a twisted multiplicativity that reduces
their specification to the determination of the $p$-parts
\[
\sum_{k_{\scalebox{.75}{$\scriptscriptstyle 1$}},
  \ldots, \, k_{\scalebox{.85}{$\scriptscriptstyle r$}}\, \ge \, 0}  
H\big(p^{k_{\scalebox{.75}{$\scriptscriptstyle 1$}}}\!, \ldots,
p^{k_{\scalebox{.85}{$\scriptscriptstyle r$}}}\!;
p^{l_{\scalebox{.75}{$\scriptscriptstyle 1$}}}\!, \ldots,
p^{l_{\scalebox{.85}{$\scriptscriptstyle r$}}}\big)\,
p^{- \, k_{\scalebox{.75}{$\scriptscriptstyle 1$}}
  s_{\scalebox{.75}{$\scriptscriptstyle 1$}} - \, \cdots \, - \,
  k_{\scalebox{.85}{$\scriptscriptstyle r$}}
  s_{\scalebox{.85}{$\scriptscriptstyle r$}}}
\]
for
%{\normalfont\itshape rational}
odd primes $p$ and tuples
$
(l_{\scalebox{1.1}{$\scriptscriptstyle 1$}},\ldots,
l_{\scalebox{1.25}{$\scriptscriptstyle r$}})\in \mathbb{N}^{r};
$ 
see \cite{CFG} and \cite{CG2}. The $p$-parts are constructed so that
 the resulting WMDS satisfies a group of functional equations isomorphic to
  the {W}eyl group of~$\Phi_{\sss 0}$;
the meromorphic continuation of this series to $\mathbb{C}^{r}$
is then essentially automatic. There are several equivalent
methods of representing the correct $p$-parts, namely,
\begin{itemize}
\item Definition by the ``averaging method''\!\!, also known as the
  {C}hinta-{G}unnells method \cite{CG1, CFG, CG2}.

\item Definition as spherical $p$-adic {W}hittaker functions \cite{BBF2, BBF3}.

\item Definition as sums over crystal bases \cite{BBF2, McN1}.

\item Definition as partition functions of statistical-mechanical
  lattice models \cite{BBB, BBBF}.
\end{itemize}
The {C}hinta-{G}unnells method has been extended to the particular affine
root system
$
D_{\scalebox{1.1}{$\scriptscriptstyle 4$}}^{\scalebox{1.1}{$\scriptscriptstyle (1)$}}
$
in \cite{BD}, and to any root system associated with a
symmetrizable {K}ac-{M}oody algebra in \cite{Lee-Zh}. Moreover, a
{C}asselman-{S}halika type formula for {W}hittaker functions on
metaplectic covers of {K}ac-{M}oody groups over non-archimedean local
fields has been recently established by {P}atnaik and {P}usk\'as
\cite{PP}. \!Despite these advances, we are quite far from a satisfactory theory
of multiple {D}irichlet series in the general setting of {K}ac-{M}oody {L}ie
algebras and their {W}eyl groups. In a way or another, the main difficulties occurring in
the infinite-dimensional case are caused by the existence of imaginary roots. To be more precise, 
let us consider, for example, the axiomatic multiple {D}irichlet series associated to the fourth 
moment of quadratic {D}irichlet {$L$}-\linebreak functions over rational function fields $\mathbb{F}_{\! q}(x)$ of odd 
characteristics, see \cite{DV, Saw, White2}; for simplicity, we shall assume throughout that $q\equiv 1 \!\pmod 4.$ For 
$
\Re(s_{\scriptscriptstyle i}) >  1,
$ 
$i = 1, \ldots, 5,$ this series can be expressed as 
\begin{equation} \label{eq: WMDS-affineD4-intro}
\sum_{\substack{d\; \mathrm{monic} \\ d = d_{\scalebox{.75}{$\scriptscriptstyle 0$}}^{} 
		d_{\scalebox{.75}{$\scriptscriptstyle 1$}}^{\scalebox{.8}{$\scriptscriptstyle 2$}}\\ 
		d_{\scalebox{.75}{$\scriptscriptstyle 0$}} \;  \mathrm{monic \, \& \, square-free}}} 
	\frac{\prod_{i = 1}^{4} L\!\left(s_{\scriptscriptstyle i} + \tfrac{1}{2}, 
		\chi_{d_{\scalebox{.8}{$\scriptscriptstyle 0$}}}\right) 
		\cdot P_{\scalebox{1.1}{$\scriptscriptstyle d$}}(\mathbf{s}\scalebox{1.}{$\scriptscriptstyle '$}; 
		\chi_{d_{\scalebox{.8}{$\scriptscriptstyle 0$}}})}{|d|^{s_{\scalebox{.75}{$\scriptscriptstyle 5$}}}}
\end{equation} 
where 
$
P_{\scalebox{1.1}{$\scriptscriptstyle d$}}(\mathbf{s}\scalebox{1.}{$\scriptscriptstyle '$}; 
\chi_{d_{\scalebox{.8}{$\scriptscriptstyle 0$}}}),
$ 
$\mathbf{s}\scalebox{1.}{$\scriptscriptstyle '$} = (s_{\scriptscriptstyle 1}, \ldots, s_{\scriptscriptstyle 4}),$ are certain correction polynomials. Substitute 
$
x_{\scalebox{1.1}{$\scriptscriptstyle i$}} \! = q^{- s_{\scalebox{.75}{$\scriptscriptstyle i$}} - \frac{1}{2}}
$ 
($i=1, \ldots,$ $5$), and denote the resulting function by $Z(\mathbf{x}; q),$ where 
$\mathbf{x} = (x_{\scalebox{1.1}{$\scriptscriptstyle 1$}}, \ldots, x_{\scalebox{1.1}{$\scriptscriptstyle 5$}}).$ 
Then the function 
$
Z(q^{\scalebox{.95}{$\scriptscriptstyle -1\slash 2$}}\mathbf{x}; q)
$\linebreak 
and the {C}hinta-{G}unnells average 
$
Z_{\scriptscriptstyle W}^{\scriptscriptstyle \mathrm{CG}}(\mathbf{x}; \sqrt{q})
$ 
for 
$
D_{\scalebox{1.1}{$\scriptscriptstyle 4$}}^{\scalebox{1.1}{$\scriptscriptstyle (1)$}}
$ 
satisfy the same group of functional equations.\linebreak 
Consequently 
\[
Z(q^{\scalebox{.95}{$\scriptscriptstyle -1\slash 2$}}\mathbf{x}; q)
= F(\mathbf{x}^{\scalebox{1.1}{$\scriptscriptstyle \delta$}})
Z_{\scriptscriptstyle W}^{\scriptscriptstyle \mathrm{CG}}(\mathbf{x}; \sqrt{q})
\] 
for some function $F$ of one complex variable. In other words, the two functions differ by a power series 
whose coefficients (apart from the constant term which is $1$) are supported on the positive affine imaginary\linebreak 
roots $n \delta,$ with $n \ge 1.$ The appearance of such correction factors is a common feature when dealing with 
extensions of classical formulas to various infinite-dimensional settings, and pinning down these factors is\linebreak 
usually quite a daunting task; see {M}acdonald's analogue \cite{Mac} of {W}eyl's denominator identity for affine root\linebreak systems, 
{K}ac's generalization \cite{Kac} of the same identity to symmetrizable {K}ac-{M}oody algebras, the affine 
{G}indikin-{K}arpelevich formula \cite{BFK, BGKP}, the affine {M}acdonald formula \cite{Mac2, Che-Ma, BKP}, and 
{W}hittaker functions on $p$-adic loop groups \cite{Pat, PP}.

One of our main results is the following (Theorem \ref{main-thm}): 
\vskip5pt
\begin{thm} \label{Th-I1}--- The factor $F(z)$ is given by
	\[
	F(z) = \prod_{n \, \ge \, 1} \!\left(1 - q  z^{2n - 1}\right)^{- 2}.
	\] 
	In particular, the function $Z(q^{\scalebox{.95}{$\scriptscriptstyle -1\slash 2$}}\mathbf{x}; q)$ has 
	meromorphic continuation to 
	$
	\Omega =\{\mathbf{x} \in \mathbb{C}^{5} : |\mathbf{x}^{\scalebox{1.1}{$\scriptscriptstyle \delta$}}| < 1 \},
	$ 
	and in this domain it satisfies a group of functional equations. 
\end{thm}
\vskip5pt
\begin{rem} \!The function $Z_{\scriptscriptstyle
    W}^{\scriptscriptstyle \mathrm{CG}}(\mathbf{x}; \sqrt{q})$ cannot
  be meromorphically continued beyond the region $\Omega,$
  and\linebreak the fact that $Z(q^{\scalebox{.95}{$\scriptscriptstyle -1\slash
            2$}}\mathbf{x}; q)$ satisfies a group of functional
        equations was already known by the work of {W}hitehead
        \cite{White2}. The coefficients of the power series expansion of 
	$
	Z_{\scriptscriptstyle W}^{\scriptscriptstyle \mathrm{CG}}(\mathbf{x}; \sqrt{q})
	$ 
	are polynomials in $u = \sqrt{q},$ and so 
	$
	Z(u^{\scalebox{1.1}{$\scriptscriptstyle -1$}}\mathbf{x}; u^{\scalebox{1.1}{$\scriptscriptstyle 2$}})$ 
	can be considered as a function of the additional complex variable $u.$ In particular, the function 
	$
	Z(q^{\scalebox{.95}{$\scriptscriptstyle 1\slash 2$}}\mathbf{x}; q^{\scalebox{1.1}{$\scriptscriptstyle -1$}}),
	$ 
	obtained for $u = q^{\scalebox{1.1}{$\scriptscriptstyle -1\slash 2$}},$ gives -- after substituting $q \to |p|$ -- 
	the $p$-part of the global WMDS.    
      \end{rem}

We have the following supplement to Theorem \ref{Th-I1}: 
\vskip5pt 
\begin{thm} \label{Th-I2}--- Let 
	$
	\Phi_{\scalebox{1.2}{$\scriptscriptstyle \mathrm{re}$}}^{\scalebox{1.2}{$\scriptscriptstyle +$}}
	$ 
	denote the set of positive affine real roots of a root system of type $D_{\scriptscriptstyle 4}^{\scriptscriptstyle (1)}\!.$ 
	Then the singu-\linebreak larities of $Z(q^{\scalebox{.95}{$\scriptscriptstyle -1\slash 2$}}\mathbf{x}; q)$ occur along the hypersurfaces $\mathbf{x}^{2\alpha} = q^{\scalebox{1.2}{$\scriptscriptstyle - 1$}},$ for 
	$
	\alpha  \in \Phi_{\scalebox{.95}{$\scriptscriptstyle \mathrm{re}$}}^{\scalebox{.95}{$\scriptscriptstyle +$}}.
	$ 
	In particular, this function has a simple pole at 
	$
	x_{\scalebox{1.1}{$\scriptscriptstyle 5$}} = q^{\scalebox{1.1}{$\scriptscriptstyle -1 \slash 2$}}
	$ 
	with residue 
	\[
	\underset{x_{\scalebox{.75}{$\scriptscriptstyle 5$}} \to q^{\scalebox{.75}{$\scriptscriptstyle -\frac{1}{2}$}}}{\mathrm{Res}}
	Z(q^{\scalebox{.95}{$\scriptscriptstyle -1\slash 2$}}\mathbf{x}; q)
	= -\, \frac{q^{\scalebox{.9}{$\scriptscriptstyle -\frac{1}{2}$}}}{\left(P_{\phantom i}^{\scalebox{1.1}{$\scriptscriptstyle 2$}}; 
		P_{\phantom i}^{\scalebox{1.1}{$\scriptscriptstyle 2$}}\right)_{\infty}
		\!\left(q P_{\phantom i}^{\scalebox{1.1}{$\scriptscriptstyle 2$}}; 
		P_{\phantom i}^{\scalebox{1.1}{$\scriptscriptstyle 2$}}\right)_{\infty}
		\prod_{\scalebox{1.1}{$\scriptscriptstyle i = 1$}}^{\scalebox{1.1}{$\scriptscriptstyle 4$}}  
		\left(x_{\scalebox{1.1}{$\scriptscriptstyle i$}}^{\scalebox{1.1}{$\scriptscriptstyle 2$}}; 
		P_{\phantom i}^{\scalebox{1.1}{$\scriptscriptstyle 2$}}\right)_{\infty} 
		\!\left(q x_{\scalebox{1.1}{$\scriptscriptstyle i$}}^{\scalebox{1.1}{$\scriptscriptstyle - 2$}}
		P_{\phantom i}^{\scalebox{1.1}{$\scriptscriptstyle 2$}}; 
		P_{\phantom i}^{\scalebox{1.1}{$\scriptscriptstyle 2$}}\right)_{\infty}
		\cdot \;\,  \prod_{\scalebox{1.1}{$\scriptscriptstyle 1\, \le \, i \, < \, j \, \le \, 4$}}  
		\left(x_{\scalebox{1.1}{$\scriptscriptstyle i$}}
		x_{\scalebox{1.1}{$\scriptscriptstyle j$}}; P\right)_{\infty}} 
		\] 
		where 
		$
			P = (x_{\scalebox{1.1}{$\scriptscriptstyle 1$}}
			x_{\scalebox{1.1}{$\scriptscriptstyle 2$}}
			x_{\scalebox{1.1}{$\scriptscriptstyle 3$}}
			x_{\scalebox{1.1}{$\scriptscriptstyle 4$}})\slash q,
			$ 
			and $(a; b)_{\infty} = \prod_{k\ge 0}\, (1 - ab^{\scalebox{1.1}{$\scriptscriptstyle k$}})$ is the $b$-Pochhammer symbol.
		\end{thm} 
\vskip5pt
\begin{rem} \!The residue at any other pole of this function can be obtained from the residue at 
	$
	x_{\scalebox{1.1}{$\scriptscriptstyle 5$}} = q^{\scalebox{1.1}{$\scriptscriptstyle -1 \slash 2$}}
	$ 
	by applying a specific functional equation. 
\end{rem}

As a first application, we obtain the precise relationship between the function 
$Z(q^{\scalebox{.95}{$\scriptscriptstyle -1\slash 2$}}\mathbf{x};q)$ and the {C}assel-\linebreak man-{S}halika 
formula \cite{PP} for unramified {W}hittaker functions on metaplectic covers of {K}ac-Moody 
groups over non-archimedean local fields. With the notations and terminology of {\normalfont\itshape loc. \!cit.,} the authors prove that, for each dominant coweight $\lambda^{\scalebox{1.1}{$\scriptscriptstyle \vee$}},$ the value 
$
\mathcal{W}(\varpi^{\lambda^{\scalebox{.9}{$\scriptscriptstyle \vee$}}})
$ 
of the metaplectic {W}hittaker function on the toral element $\varpi^{\lambda^{\scalebox{.9}{$\scriptscriptstyle \vee$}}},$ where $\varpi$ is a chosen uniformizer, is the $p$-adic specialization of the expression 
\begin{equation} \label{eq: CS-formula}
v^{\langle \lambda^{\scalebox{.9}{$\scriptscriptstyle \vee$}}, \, \rho \rangle}
\widetilde{\fm} 
\widetilde{\Delta}
\, \cdot \sum_{w \, \in \, W} 
(-1)^{\scalebox{1.1}{$\scriptscriptstyle \ell(w)$}} 
\Big(\prod_{\tilde{a}^{\scalebox{.8}{$\scriptscriptstyle \vee$}} 
\in \, \widetilde{\Phi}^{^{\scalebox{.8}{$\scriptscriptstyle \vee$}}}\!(w)} 
e^{-\tilde{a}^{\scalebox{.8}{$\scriptscriptstyle \vee$}}}\Big)w \, \mathbf{\star} \, 
e^{\lambda^{\scalebox{.9}{$\scriptscriptstyle \vee$}}}.
\end{equation}
Here $\widetilde{\Phi}^{^{\scalebox{.8}{$\scriptscriptstyle \vee$}}}$ is the dual root system 
constructed in \cite{PP} using a metaplectic structure $(\mathrm{Q}, n)$ on a root datum.\linebreak The factor $\widetilde{\Delta}$ is defined by the formal infinite product 
\[
	\widetilde{\Delta} : = 
	\prod_{\tilde{a} \, \in \, \widetilde{\Phi}^{^{\scalebox{.95}{$\scriptscriptstyle +$}}}}  
	\!\left(\frac{1 - v e^{-\tilde{a}^{\scalebox{.8}{$\scriptscriptstyle \vee$}}}}
	{1 - e^{-\tilde{a}^{\scalebox{.8}{$\scriptscriptstyle \vee$}}}} \right)^{\! m(\tilde{a})}
\] 
where $m(\tilde{a})$ is the root-multiplicity of $\tilde{a}.$ When the root datum is associated with 
an untwisted affine root system $\Phi$ of ADE type, the correction factor $\widetilde{\fm}$ is 
\[
\widetilde{\fm} : = \mathrm{ct}(\widetilde{\Delta}^{\scalebox{1.1}{$\scriptscriptstyle -1$}}) 
= \prod_{i = 1}^{r} \prod_{j = 1}^{\infty} 
\frac{1 - v^{\widetilde{m}_{\scalebox{.8}{$\scriptscriptstyle i$}}}e^{- j \tilde{\delta}}}
{1 - v^{\widetilde{m}_{\scalebox{.8}{$\scriptscriptstyle i$}} + 1}e^{- j \tilde{\delta}}}
\] 
where $\{\widetilde{m}_{\scalebox{1.1}{$\scriptscriptstyle i$}}\}_{\scalebox{1.}{$\scriptscriptstyle i = 1, \ldots, r$}}$ is the set of {\normalfont\itshape exponents} of the underlying {\normalfont\itshape finite} root system to $\widetilde{\Phi},$ and $\tilde{\delta}$ is the minimal positive imaginary root of $\widetilde{\Phi}.$ For example, if $\Phi$ is of 
$D_{\scalebox{1.1}{$\scriptscriptstyle 4$}}^{\scalebox{1.1}{$\scriptscriptstyle (1)$}}$ type, then $\widetilde{\Phi}$ is again of $D_{\scalebox{1.1}{$\scriptscriptstyle 4$}}^{\scalebox{1.1}{$\scriptscriptstyle (1)$}}$ type (see \cite[Table\linebreak 2.3.2]{PP}), and if 
$
\{\alpha_{\scriptscriptstyle i}\}_{\scriptscriptstyle i = 1, \ldots, 5}\subset 
\Phi_{\scalebox{1.2}{$\scriptscriptstyle \mathrm{re}$}}^{\scalebox{1.2}{$\scriptscriptstyle +$}}
$ 
is the set of simple roots, $x_{\scalebox{1.1}{$\scriptscriptstyle i$}} : = e^{-\alpha_{\scalebox{.75}{$\scriptscriptstyle i$}}}\!,$ then 
\[
\widetilde{\fm}(\mathbf{x}; v) = \prod_{i = 1}^{4} \prod_{j = 1}^{\infty} 
\frac{1 - v^{\widetilde{m}_{\scalebox{.8}{$\scriptscriptstyle i$}}}\mathbf{x}^{2j \delta}}
{1 - v^{\widetilde{m}_{\scalebox{.8}{$\scriptscriptstyle i$}} + 1} \mathbf{x}^{2j \delta}},
\; \text{the exponents being $1, 3, 3, 5$}.
\] 
If we specialize \eqref{eq: CS-formula} to the case when $\widetilde{G}$ is the metaplectic {\normalfont\itshape double} cover of a simply connected affine {K}ac-\linebreak {M}oody group $G$ of type 
$D_{\scalebox{1.1}{$\scriptscriptstyle 4$}}^{\scalebox{1.1}{$\scriptscriptstyle (1)$}}$ over a non-archimedean local field, we get: 
\vskip5pt
\begin{thm} \label{Th-I3}--- If we define 
	\[
	\tilde{\frak{c}}(\mathbf{x}; v) : = \prod_{n = 1}^{\infty} \!\left(1 - v\mathbf{x}^{n\delta}\right)^{\! 2}
	\!\left(1 - v\mathbf{x}^{2n\delta} \right)^{\! 2} \;\, \text{and} \;\;\, 
	D(\mathbf{x}; v) : = \prod_{\alpha  \in \Phi_{\scalebox{.95}{$\scriptscriptstyle \mathrm{re}$}}^{\scalebox{.95}{$\scriptscriptstyle +$}}} 
		\!\left(1 - v\mathbf{x}^{2\alpha} \right)
\] 
	then the value of the unramified {W}hittaker function on the group $\widetilde{G}$ 
	at the identity element is 
	\[
\mathcal{W}(1) = 
(\widetilde{\fm} \tilde{\frak{c}} D)(\mathbf{x}; q^{\scalebox{1.1}{$\scriptscriptstyle -1$}})
Z(q^{\scalebox{.95}{$\scriptscriptstyle 1\slash 2$}}\mathbf{x}; q^{\scalebox{1.1}{$\scriptscriptstyle -1$}})
\] 
	where $q\equiv 1 \!\pmod 4$ is the size of the residue field.
\end{thm}
\vskip5pt 
\begin{rem} \!A similar comparison result for metaplectic {W}hittaker
  functions on simply laced affine {K}ac-{M}oody groups will appear in
  \cite{DIPP}. In the finite-dimensional case, the equivalence between
  the {C}hinta-\linebreak {G}unnells method and the representation of the
  $p$-parts as spherical $p$-adic {W}hittaker functions was
  estab-\linebreak lished by {M}c{N}amara \cite{McN2}.
\end{rem}

Our main reason for singling out the study of WMDS of type
$
D_{\scalebox{1.1}{$\scriptscriptstyle 4$}}^{\scalebox{1.1}{$\scriptscriptstyle (1)$}}
$
concerns the fourth moment of quad-\linebreak ratic {D}irichlet
$L$-functions. \!Traditionally the moment problem for this family of
$L$-functions is asking for an asymptotic formula for
$ 
\sum_{\deg d \, = D} L\!\left(\tfrac{1}{2}, \chi_{d}\right)^{r}
$
($r \ge 1$) as $D \to \infty,$ the sum being over monic square-free
polynomials $d \in \mathbb{F}_{\! q}[x].$ A conjectural asymptotic
formula, for all $r,$ was first given by Andrade and Keating \cite{AK}
(see also \cite{RW}), and a refined version, exhibiting additional lower order
terms in the asymptotic formula, was recently proposed in
\cite{DT}. This conjecture is known only for $r \le 3,$ see
\cite{Florea1, Florea2, Dia}, and when\linebreak $r = 4,$ a weaker
form of the asymptotic formula is known by the work
of Florea \cite{Florea3}.

As an immediate consequence of Theorem \ref{Th-I1}, we have:
\vskip5pt
\begin{thm} \label{Th-I4}--- Assuming $q\equiv 1 \!\pmod 4$ and $D \ge 1,$
  we have the exact formula:
 \[ 
\sum_{\substack{\deg d \, = D\\d\; \mathrm{monic}}} L\!\left(\tfrac{1}{2}, \chi_{d_{\scalebox{.8}{$\scriptscriptstyle 0$}}}\right)^{4} 
\!P_{\scalebox{1.1}{$\scriptscriptstyle d$}}(\chi_{d_{\scalebox{.8}{$\scriptscriptstyle 0$}}})
= \mathrm{Coeff}_{\xi^{\scalebox{1.}{$\scriptscriptstyle D$}}}
\left[\, \prod_{n \, \ge \, 1} \!\left(1 - q  \xi^{4n - 2}\right)^{- 2}
\cdot Z_{\scriptscriptstyle W}^{\scriptscriptstyle
  \mathrm{CG}}(\underline{1}, \xi; \sqrt{q})
\right]
\]
where
$
P_{\scalebox{1.1}{$\scriptscriptstyle d$}}(\chi_{d_{\scalebox{.8}{$\scriptscriptstyle 0$}}})
= P_{\scalebox{1.1}{$\scriptscriptstyle d$}}(0, \ldots, 0; \chi_{d_{\scalebox{.8}{$\scriptscriptstyle 0$}}})$
and $\underline{1}: = (1, 1, 1, 1).$ Here, for a monic polynomial
$d \in \mathbb{F}_{\! q}[x],$ we write
$
d = d_{\scalebox{1.1}{$\scriptscriptstyle 0$}}^{} 
d_{\scalebox{1.1}{$\scriptscriptstyle 1$}}^{\scalebox{1.1}{$\scriptscriptstyle 2$}}
$
with $d_{\scalebox{1.1}{$\scriptscriptstyle 0$}}^{}$ monic and square-free.
\end{thm}

While this result is probably special to rational function fields,
the following asymptotic formula is ex-\linebreak pected to generalize to any
global field.
\vskip5pt
\begin{thm} \label{Th-I5} --- For $D, N \ge 1$ and
  $
  (N + 1)^{\scalebox{1.1}{$\scriptscriptstyle -1$}} < \Theta
  < N^{\scalebox{1.1}{$\scriptscriptstyle -1$}},
  $
  we have
  \begin{equation} \label{eq: asymptotic-intro}
    \sum_{\substack{\deg d \, = D\\d\; \mathrm{monic}}}
   L\!\left(\tfrac{1}{2}, \chi_{d_{\scalebox{.8}{$\scriptscriptstyle 0$}}}\right)^{4} 
\!P_{\scalebox{1.1}{$\scriptscriptstyle d$}}(\chi_{d_{\scalebox{.8}{$\scriptscriptstyle 0$}}})
\, = \sum_{n \, \le \, N} Q_{n}(D, q) q^{\scalebox{.8}{$\scriptscriptstyle \frac{D}{2n}$}}
\, + \, O_{\scalebox{1.1}{$\scriptscriptstyle \Theta$}, \, q}
\left(q^{\scalebox{.8}{$\scriptscriptstyle \frac{D\Theta}{2}$}}\right)
\end{equation}
where $Q_{n}(D, q)$ is a polynomial in $D$ of degree 10 if $n$ is odd
and of degree 7 if $n$ is even. Furthermore, if we let\linebreak
$\varrho = q^{-1\slash n} \in \mathbb{R},$ then the leading
coefficient of $Q_{n}(D, q)$ is given by an expression of the form
\[
  (-1)^{\lfloor n\slash 2 \rfloor} \varrho^{3  \lfloor n\slash 2 \rfloor \cdot  \lfloor (n+1)\slash 2 \rfloor }
  g_{\scalebox{1.1}{$\scriptscriptstyle n$}, \scalebox{1.}{$\scriptscriptstyle D$}}(\sqrt \varrho)
\]
where
$
g_{\scalebox{1.1}{$\scriptscriptstyle n$},
  \scalebox{1.}{$\scriptscriptstyle D$}}\ne 0
$
is a power series in $\varrho^{\scalebox{1.1}{$\scriptscriptstyle 1\slash 2$}}$
with non-negative coefficients.
\end{thm}

According to \cite{DT}, a similar asymptotic formula should hold for 
$
\sum_{\deg d \, = D} L\!\left(\tfrac{1}{2}, \chi_{d}\right)^{4}
$, 
summed only over square-free monics; it should be a consequence of
the analytic properties of
$
Z(q^{\scalebox{.95}{$\scriptscriptstyle -1\slash 2$}}\mathbf{x}; q)
$
(i.e., Theorem \ref{Th-I1} and Theorem \ref{Th-I2} above) and its twists, which shall be
addressed in a future work. However, from the ana-\linebreak lytic
point of view, the asymptotic formula \eqref{eq: asymptotic-intro}
should be sufficient. This is so since it has the correct order of
magnitude, and the coefficients
$
P_{\scalebox{1.1}{$\scriptscriptstyle d$}}(\chi_{d_{\scalebox{.8}{$\scriptscriptstyle 0$}}})
$
are, as we shall see, non-negative; when
$
d = d_{\scalebox{1.1}{$\scriptscriptstyle 0$}}
$
is square-free, we have
$
P_{\scalebox{1.1}{$\scriptscriptstyle d$}}(\chi_{d_{\scalebox{.8}{$\scriptscriptstyle 0$}}})
=  |d|^{\scalebox{1.1}{$\scriptscriptstyle - 1\slash 2$}},
$
which also explains the discrepancy in \eqref{eq: asymptotic-intro}
by a factor of
$
q^{\scalebox{1.}{$\scriptscriptstyle D\slash 2$}}
$
(cf. \cite[Conjecture\linebreak 1.2]{DT}).

\begin{rem} \!The study of the analytic properties of twisted versions of
		$
		Z(q^{\scalebox{.95}{$\scriptscriptstyle -1\slash 2$}}\mathbf{x}; q)
		$ 
		mentioned in the previous paragraph is also relevant to the problem of extending our results to 
		higher genus function fields. The main difficulty in dealing with both these general situations comes from the failure of the so-called {\normalfont\itshape local-to-global} property whose complexity grows as the degrees of the twisting parameters increase. Recently, {F}riedlander \cite{Fried} obtained an interesting explicit formula that can be interpreted as measuring the {\normalfont\itshape degree of failure} of the local-to-global property of a (finite-dimensional) twisted {W}eyl group multiple {D}irichlet series over rational function fields. We expect that an analogous formula will hold for affine {W}eyl group multiple {D}irichlet series.
	\end{rem}

Using the $p$-parts
$
Z(p^{\scalebox{.95}{$\scriptscriptstyle 1\slash 2$}}\mathbf{x};
p^{\scalebox{1.1}{$\scriptscriptstyle -1$}}),
$
with $p$ a rational odd prime, one constructs (cf. \cite[Remark 1]{DT})
the ana-\linebreak logue of \eqref{eq: WMDS-affineD4-intro} over the
rationals. This series has the form
\begin{equation} \label{eq: WMDS-rational-intro}
  Z(s_{\scriptscriptstyle 1}, \ldots, s_{\scriptscriptstyle 5};
  \chi_{a_{\scalebox{.8}{$\scriptscriptstyle 2$}}}, \chi_{a_{\scalebox{.8}{$\scriptscriptstyle 1$}}}) \;
  = \sum_{\substack{d  = d_{\scalebox{.75}{$\scriptscriptstyle 0$}}^{} 
d_{\scalebox{.75}{$\scriptscriptstyle 1$}}^{\scalebox{.8}{$\scriptscriptstyle 2$}} \, \ge \, 1\\ (d, 2) = 1}} 
	\frac{\prod_{i = 1}^{4} L\!\left(s_{\scriptscriptstyle i} + \tfrac{1}{2}, 
\chi_{a_{\scalebox{.8}{$\scriptscriptstyle 1$}} d_{\scalebox{.8}{$\scriptscriptstyle 0$}}}\right) 
\cdot \chi_{a_{\scalebox{.8}{$\scriptscriptstyle 2$}}}\!(d_{\scriptscriptstyle 0})
P_{\scalebox{1.1}{$\scriptscriptstyle
                    d$}}(\mathbf{s}\scalebox{1.}{$\scriptscriptstyle
                  '$}; \chi_{a_{\scalebox{.8}{$\scriptscriptstyle 1$}}
                  d_{\scalebox{.8}{$\scriptscriptstyle 0$}}})}{d^{s_{\scalebox{.75}{$\scriptscriptstyle 5$}}}}
            \end{equation}
            where
            $
            a_{\scriptscriptstyle 1}, a_{\scriptscriptstyle 2} \in \{\pm 1,
            \pm2 \}
            $
            and $\chi_{d},$ for $d\in \mathbb{Z}$ non-zero and
            square-free, is the usual quadratic character; as
            in\linebreak the function field case, this expression is
            certainly valid when
            $
            \Re(s_{\scriptscriptstyle i}) >  1
            $
            ($i = 1, \ldots, 5$). It has some initial continuation,
            and in that region it satisfies a group $W$ of functional
            equations isomorphic to a {W}eyl group\linebreak of type
            $
            D_{\scalebox{1.1}{$\scriptscriptstyle 4$}}^{\scalebox{1.1}{$\scriptscriptstyle (1)$}}
            $. 
If $a_{\scriptscriptstyle 2} = 1,$ this function has a simple pole
at $s_{\scriptscriptstyle 5} = \frac{1}{2},$ and its residue can now
be explicitly com-\linebreak puted as the {\normalfont\itshape infinite} product 
of zeta functions 
\[
\prod_{p \ne 2} R\big(
p^{\scalebox{.9}{$\scriptscriptstyle - \frac{1}{2}$} - s_{\scalebox{.75}{$\scriptscriptstyle 1$}}}\!,
p^{\scalebox{.9}{$\scriptscriptstyle - \frac{1}{2}$} - s_{\scalebox{.75}{$\scriptscriptstyle 2$}}}\!,
p^{\scalebox{.9}{$\scriptscriptstyle - \frac{1}{2}$} - s_{\scalebox{.75}{$\scriptscriptstyle 3$}}}\!,
p^{\scalebox{.9}{$\scriptscriptstyle - \frac{1}{2}$} - s_{\scalebox{.75}{$\scriptscriptstyle 4$}}};
p^{\scalebox{.9}{$\scriptscriptstyle - \frac{1}{2}$}}\big)
\] 
with $R(\underline{x};u)$
$
(\underline{x}
:= (x_{\scalebox{1.1}{$\scriptscriptstyle 1$}}, \ldots, x_{\scalebox{1.1}{$\scriptscriptstyle 4$}}))
$
given in Theorem \ref{Residue-tilde-Z-avg}; the residues at the other
simple poles (under the action of $W$) can be computed using the
functional equations.

Conjecturally
$
Z(\mathbf{s}; \chi_{a_{\scalebox{.8}{$\scriptscriptstyle 2$}}},
\chi_{a_{\scalebox{.8}{$\scriptscriptstyle 1$}}})
$
$
(\mathbf{s} : = (s_{\scriptscriptstyle 1}, \ldots, s_{\scriptscriptstyle 5}))
$
admits meromorphic continuation to $\Re(\delta(\mathbf{s})) > 0,$ with
all its singularities contained in the set
$
\left\{w\big(s_{\scriptscriptstyle 5} = \frac{1}{2}\big)\right\}_{w \,\in \, W}.
$

Let us say some words about the proofs of our main results; the
proofs of Theorems \ref{Th-I3}-\ref{Th-I5} are (essentially) straightforward applications
of Theorems \ref{Th-I1} and \ref{Th-I2}. The key result in establishing these two
theorems is a {\normalfont\itshape new} functional equation satisfied
by the {C}hinta-{G}unnells average for affine root systems. This functional equation has no finite-dimensional analogue. To state this in the
$
D_{\scalebox{1.1}{$\scriptscriptstyle 4$}}^{\scalebox{1.1}{$\scriptscriptstyle (1)$}}
$
case, consider the average
$
Z_{\scriptscriptstyle W}(\mathbf{x}; u) = \sum_{w \, \in \, W} 1 \vert w(\mathbf{x}; u),
$
where $f \vert w$ is the {C}hinta-{G}unnells action.
We will show in Proposition \ref{continuation-W-invarianceZ} that,
as a function of the complex variables
$
x_{\scalebox{1.1}{$\scriptscriptstyle 1$}}, \ldots, x_{\scalebox{1.1}{$\scriptscriptstyle 5$}}, u,$
$
Z_{\scriptscriptstyle W}(\mathbf{x}; u)
$
is holomorphic for
$
|\mathbf{x}^{\scalebox{1.1}{$\scriptscriptstyle \delta$}}| < 1, u \in \mathbb{C}
$
except for the set of points for which $u\mathbf{x}^{\alpha}  \pm 1 = 0$
for some
$ 
\alpha \in \Phi_{\scalebox{1.2}{$\scriptscriptstyle \mathrm{re}$}}^{\scalebox{1.2}{$\scriptscriptstyle +$}}. 
$
For\linebreak $a, b \in \{\mathcal{e}, \mathcal{o}\},$ where $\mathcal{e}$
(resp. \!$\mathcal{o}$) stands for even (resp. \!odd), let
$
Z_{\scriptscriptstyle W}^{\scalebox{1.2}{$\scriptscriptstyle a, b$}}
$
denote the part of $Z_{\scriptscriptstyle W}$ which has parity
$a$ with respect to the involution
$
\varepsilon_{\scalebox{1.1}{$\scriptscriptstyle 5$}}(\underline{x}, x_{\scalebox{1.1}{$\scriptscriptstyle 5$}}) = (-\underline{x}, x_{\scalebox{1.1}{$\scriptscriptstyle 5$}})
$
and parity $b$ with respect to the involution
$
\varepsilon_{\scalebox{1.1}{$\scriptscriptstyle 1$}}(\underline{x},
x_{\scalebox{1.1}{$\scriptscriptstyle 5$}})
= (\underline{x}, -x_{\scalebox{1.1}{$\scriptscriptstyle 5$}}).
$
If we let
$
\mathbf{Z}  = 
 \!^{t}(Z_{\scriptscriptstyle W}^{\scalebox{1.2}{$\scriptscriptstyle \mathcal{e}, \mathcal{e}$}}\!, 
Z_{\scriptscriptstyle W}^{\scalebox{1.2}{$\scriptscriptstyle \mathcal{e}, \mathcal{o}$}}\!, 
Z_{\scriptscriptstyle W}^{\scalebox{1.2}{$\scriptscriptstyle \mathcal{o}, \mathcal{e}$}}),
$
then this vector function satisfies the functional equation
\begin{equation} \label{eq: fe-intro}
\mathbf{Z}(\mathbf{x}; u) = B(\mathbf{x}; u)
\mathbf{Z}\!\left(\mathbf{x}; u \mathbf{x}^{\delta}\right)
\end{equation}
for a 3 by 3 matrix $B(\mathbf{x}; u)$ with {\normalfont\itshape
  rational} entries. In fact, the matrix
\[
\prod_{\substack{\alpha  \, \in \,    
\Phi_{\scalebox{.95}{$\scriptscriptstyle \mathrm{re}$}}^{\scalebox{.95}{$\scriptscriptstyle +$}}
\\ \alpha \, < \, \delta}} 
\!\left(1 - u^{2}\mathbf{x}^{2\alpha} \right)
\cdot B(\mathbf{x}; u)
\]
has polynomial entries in $\mathbf{x}$ and $u,$ and each entry of
$B(\mathbf{x}; u)$ is divisible by
$
\left(1 - u^{2}\mathbf{x}^{\delta}\right)^{\scalebox{1.1}{$\scriptscriptstyle 2$}}
$,
see Theorem \ref{Key-ingredient}.\linebreak (An explicit formula for
$B(\mathbf{x}; u)$ is obtained via the identity
$
B(\mathbf{x}; u) = A^{\scalebox{1.1}{$\scriptscriptstyle - 1$}}(\mathbf{x}; u\mathbf{x}^{\delta}),
$
with $A$ given in~\cite{eDPP}.)\linebreak

\vskip-12.4pt
The strategy to deduce Theorem \ref{Th-I1} and Theorem \ref{Th-I2}
from the functional equation \eqref{eq: fe-intro} goes as follows. The
divisibility of the matrix $B(\mathbf{x}; u)$ by
$
\left(1 - u^{2}\mathbf{x}^{\delta}\right)^{\scalebox{1.1}{$\scriptscriptstyle 2$}}
$
implies that the renormalization $\tilde{Z}_{\scriptscriptstyle W}$
of the Chinta-Gunnells average, defined by
\eqref{eq: average-zeta-normalized}, and the function
$Z_{\scriptscriptstyle W}$ have the same singularities in $\Omega.$ We
then compute the residue of the function
$\tilde{Z}_{\scriptscriptstyle W}(\mathbf{x}; u)$ at
$x_{\scalebox{1.1}{$\scriptscriptstyle 5$}} = u^{\scalebox{1.1}{$\scriptscriptstyle -1$}}$
by using \eqref{eq: fe-intro} and the invariance of the average
$Z_{\scriptscriptstyle W}(\mathbf{x}; u)$ under the {W}eyl group,
which reduce the calculation to a simple application of the
classical Macdonald's identity~\cite{Mac}. \!The structure of this
residue (given by the formula in Theorem \ref{Th-I2},
with $q$ replaced by $u^{\scalebox{1.1}{$\scriptscriptstyle 2$}}$) allows
us to construct a WMDS of type \eqref{eq: WMDS-affineD4-intro},
with
$
\tilde{Z}_{\scriptscriptstyle W}(\mathbf{x};
|p|^{\scalebox{1.1}{$\scriptscriptstyle - 1\slash 2$}})
$
as the $p$-parts, which after substituting
$
x_{\scalebox{1.1}{$\scriptscriptstyle i$}} = q^{ - s_{\scalebox{.75}{$\scriptscriptstyle i$}}},
$
turns out to have the same residue as
$
\tilde{Z}_{\scriptscriptstyle W}(\mathbf{x}; \sqrt{q})
$
at the simple pole
$
x_{\scalebox{1.1}{$\scriptscriptstyle 5$}} = 1\slash \sqrt{q}.
$ 
Thus, by letting $\mathscr{Z}\!(\mathbf{x})$ denote 
this WMDS in the variables 
$
x_{\scalebox{1.1}{$\scriptscriptstyle i$}}, 
$ 
we deduce that 
$
\mathscr{Z}\!(\mathbf{x}) = \tilde{Z}_{\scriptscriptstyle W}(\mathbf{x}; \sqrt{q}).
$ 
Finally, we note that the matrix 
$
B(\mathbf{x}; u)\slash
\left(1 -
  u^{2}\mathbf{x}^{\delta}\right)^{\scalebox{1.1}{$\scriptscriptstyle 2$}}
$ 
determines $\tilde{Z}_{\scriptscriptstyle W}(\mathbf{x}; u)$
recursively, see Lemma \ref{L5.3}; this implies that the\linebreak
coefficients of $\tilde{Z}_{\scriptscriptstyle W}(u\mathbf{x}; u)$
satisfy the dominance axiom \cite{DV, White2}, hence
$
\tilde{Z}_{\scriptscriptstyle W}(\mathbf{x}; \sqrt{q})
= Z(q^{\scalebox{.95}{$\scriptscriptstyle -1\slash 2$}}\mathbf{x}; q),
$
and that the coefficients
$
P_{\scalebox{1.1}{$\scriptscriptstyle d$}}(\chi_{d_{\scalebox{.8}{$\scriptscriptstyle 0$}}})
$
in Theorem \ref{Th-I4} are non-negative. 
\vskip5pt 
{\bfseries Concluding remark.}
\!By the analogy between number fields and function fields 
of curves over finite fields,\linebreak it is conceivable that the 
WMDS \eqref{eq: WMDS-rational-intro} satisfies some analogue 
of the functional equation \eqref{300}. It is also quite\linebreak probable that
the mechanism behind this additional symmetry will provide
a natural approach to prove the meromorphic continuation of
\eqref{eq: WMDS-rational-intro} to the half-space
$\Re(\delta(\mathbf{s})) > 0$.
\vskip5pt
\emph{Acknowledgements.} We would like to thank {M}anish {M.}
{P}atnaik for clarifying for us some aspects of his joint work with
Anna Pusk\'as \cite{PP}. We would also like to thank Bogdan Ion and the anonymous referee
for useful comments and suggestions.
The authors were partially supported by the
CNCS-UEFISCDI grant PN-III-P4-ID-PCE-2020-2498. The first author was 
also partially supported by a Simons Fellowship in Mathematics.

\section{Notation} %\label{notation}
We denote by $\mathbb{N}$ the set of non-negative integers,
by $\mathbb{Z}$ the ring of (rational) integers, and by
$\mathbb{Q},$ $\mathbb{R},$ $\mathbb{C}$ the fields
of rational numbers, real numbers and complex numbers,
respectively.

By $X \ll Y$ or $X = O(Y)$ we denote an inequality of
the form $|X| \le CY,$ for some constant $C.$ We refer to\linebreak
$C$ as the {\normalfont\itshape implied} constant,
and its dependence on parameters that we wish to keep track of
(e.g., $\epsilon, q$) will be indicated by appropriate subscripts (e.g.,
$X \ll_{\epsilon, q} Y$ or $X = O_{\epsilon, q}(Y)$).

\section{Preliminaries} %\label{prelim}
We begin by recalling some basic facts about affine root systems and
their {W}eyl groups, following closely the standard reference
\cite{Mac1}.

Let $V$ be a finite-dimensional real vector space, equipped with a
positive definite symmetric scalar prod-\linebreak uct $\langle x, y\rangle;$
we shall identify $V$ with its dual space $V^{\ast}$ via the scalar
product. Consider the vector space $\mathcal{F}$
of affine-linear functions on $V$ (i.e., the functions of the form
$f = f_{\scalebox{1.1}{$\scriptscriptstyle 0$}} + c\delta,$ where
$f_{\scalebox{1.1}{$\scriptscriptstyle 0$}}(x) = \langle Df, x\rangle$
is a linear functional on $V,$ $\delta$ is the constant function $1,$
and $c\in \mathbb{R}$). We define the positive semidefinite scalar product on
$\mathcal{F}$ by
\[
\langle f, g \rangle = \langle Df, Dg\rangle.
\]
The radical of this form is the one-dimensional subspace $\mathbb{R}\delta$ of
constant functions. For non-zero $x\in V,$ let\linebreak
$
x^{\scalebox{1.1}{$\scriptscriptstyle \vee$}}
\! = 2 x \slash \langle x, x\rangle,
$
and for each non-constant $f\in \mathcal{F},$ let
$
f^{\scalebox{1.1}{$\scriptscriptstyle \vee$}}
= 2 f \slash \langle f, f\rangle.
$

Let $f\in \mathcal{F}$ be non-constant. The orthogonal reflection
$
\sigma_{\!\scriptscriptstyle f}
$
in the affine hyperplane on which $f$ vanishes is given by
\[
  \sigma_{\!\scriptscriptstyle f}(x) =
  x - f^{\scalebox{1.1}{$\scriptscriptstyle \vee$}}\!(x)Df
  = x - f(x)D f^{\scalebox{1.1}{$\scriptscriptstyle \vee$}}.
\]
The reflection
$
\sigma_{\!\scriptscriptstyle f}
$
acts on $\mathcal{F}$ by
$
\sigma_{\!\scriptscriptstyle f}(g)
= g \circ  \sigma_{\!\scriptscriptstyle f}.
$ 
For notational convenience, we shall write from now on
$\sigma_{\!\scriptscriptstyle f}g$\linebreak
instead of $\sigma_{\!\scriptscriptstyle f}(g).$

For each $v \in V,$ the translation $x \to x + v$ will be denoted by
$\tau(v).$

\subsection{Affine root systems}
Let $\Phi_{\scalebox{1.0}{$\scriptscriptstyle 0$}}$ be a rank $r$
reduced irreducible root system, spanning a real vector space
$V,$ and let
$
W_{\scalebox{1.1}{$\scriptscriptstyle 0$}}
= W(\Phi_{\scalebox{1.0}{$\scriptscriptstyle 0$}})
$
denote the {W}eyl group of $\Phi_{\scalebox{1.0}{$\scriptscriptstyle 0$}}.$
The dual root system of
$
\Phi_{\scalebox{1.0}{$\scriptscriptstyle 0$}}
$
with respect to a $W_{\scalebox{1.1}{$\scriptscriptstyle 0$}}$-invariant
positive definite scalar product $\langle x, y\rangle$ on $V$ is
$
\Phi_{\scalebox{1.0}{$\scriptscriptstyle 0$}}^{\scalebox{1.1}{$\scriptscriptstyle \vee$}}
= \{\alpha^{\scalebox{1.1}{$\scriptscriptstyle \vee$}} :
\alpha \in \Phi_{\scalebox{1.0}{$\scriptscriptstyle 0$}}\}.
$
%where
%$\alpha^{\scalebox{1.1}{$\scriptscriptstyle \vee$}}
%= 2 \alpha \slash \langle \alpha, \alpha\rangle.
%$ 
Regarding each
$\alpha \in \Phi_{\scalebox{1.0}{$\scriptscriptstyle 0$}}$ as a linear
function on $V$ (i.e.,\linebreak $\alpha(x) = \langle \alpha, x\rangle$
for $x \in V$), 
%and letting $\delta$ denote the constant function $1$ on $V,$ 
one gets the {\normalfont\itshape affine root system}
\[
\Phi = \{\alpha + n\delta : \alpha \in
\Phi_{\scalebox{1.0}{$\scriptscriptstyle 0$}}, n\in \mathbb{Z}\}
\, \scalebox{1.85}{$\scriptscriptstyle \cup$}\, 
\{m\delta\}_ {\scalebox{1.1}{$\scriptscriptstyle m \in \mathbb{Z}\setminus\{0\}$}}
\]
associated with $\Phi_{\scalebox{1.0}{$\scriptscriptstyle 0$}}.$ The
elements of the subset
$
\Phi_{\scalebox{1.3}{$\scriptscriptstyle \mathrm{re}$}} =
\{\alpha + n\delta : \alpha \in
\Phi_{\scalebox{1.0}{$\scriptscriptstyle 0$}}, n\in \mathbb{Z}\}
$
are called {\normalfont\itshape affine real roots,} and the elements
$m\delta$ ($m \in \mathbb{Z}\setminus\{0\}$) are called
{\normalfont\itshape affine imaginary roots.}

For each $\beta \in  \Phi_{\scalebox{1.3}{$\scriptscriptstyle \mathrm{re}$}},$
let $\sigma_{\!\scriptscriptstyle \beta}$
denote the orthogonal reflection in the affine hyperplane on which
$\beta$ vanishes. Explicitly, if $\beta = \alpha + n\delta$ then
\[
H_{\!\scalebox{1.1}{$\scriptscriptstyle \beta$}} =
\beta^{\scalebox{1.1}{$\scriptscriptstyle -1$}}(0)
= \{x \in V: \langle \alpha, x\rangle = - \, n\}
\;\, \mathrm{and} \;\,
\sigma_{\!\scriptscriptstyle \beta}(x)
= x - (\langle \alpha, x\rangle + n)
\alpha^{\scalebox{1.1}{$\scriptscriptstyle \vee$}}.
\]
The reflection $\sigma_{\!\scriptscriptstyle \beta}$
acts on $\Phi_{\scalebox{1.3}{$\scriptscriptstyle \mathrm{re}$}}$ by
\[
\sigma_{\!\scriptscriptstyle \beta}\beta'
=\, \beta' \circ \sigma_{\!\scriptscriptstyle \beta}
=\beta' - \langle\alpha', \alpha^{\scalebox{1.1}{$\scriptscriptstyle
    \vee$}}\rangle\beta,
\quad \beta' = \alpha' + n'\delta \in \Phi_{\scalebox{1.3}{$\scriptscriptstyle \mathrm{re}$}}.
\]
The {\normalfont\itshape affine {W}eyl group} $W$ of $\Phi$ is defined
to be the group of affine isometries of $V$ generated by all
reflec-\linebreak tions $\sigma_{\!\scriptscriptstyle \beta},$
$\beta \in \Phi_{\scalebox{1.3}{$\scriptscriptstyle \mathrm{re}$}}.$ Note that, for each $\alpha \in
\Phi_{\scalebox{1.0}{$\scriptscriptstyle 0$}},$ the composition
$
\tau(\alpha^{\scalebox{1.1}{$\scriptscriptstyle \vee$}}): =
\sigma_{\!\scriptscriptstyle \alpha} \circ
\sigma_{\!\scriptscriptstyle \alpha + \delta}
$
sends an element $x \in V$ to $x +
\alpha^{\scalebox{1.1}{$\scriptscriptstyle \vee$}},$ that is,
$\tau(\alpha^{\scalebox{1.1}{$\scriptscriptstyle \vee$}})$ is
translation by $\alpha^{\scalebox{1.1}{$\scriptscriptstyle \vee$}}.$
Thus $W$ contains a subgroup of translations
$\tau(Q^{\scalebox{1.1}{$\scriptscriptstyle \vee$}})$ isomorphic to the
root lattice $Q^{\scalebox{1.1}{$\scriptscriptstyle \vee$}}$ of
$
\Phi_{\scalebox{1.0}{$\scriptscriptstyle 0$}}^{\scalebox{1.1}{$\scriptscriptstyle \vee$}},
$
and we have
$
W = W_{\scalebox{1.1}{$\scriptscriptstyle 0$}} \ltimes
\tau(Q^{\scalebox{1.1}{$\scriptscriptstyle \vee$}}).
$

Let $\alpha_{\scriptscriptstyle 1}, \ldots, \alpha_{\scriptscriptstyle
  r}$ be a set of simple roots of
$\Phi_{\scalebox{1.0}{$\scriptscriptstyle 0$}},$ and let
$\Phi_{\scalebox{1.0}{$\scriptscriptstyle
    0$}}^{\scalebox{1.1}{$\scriptscriptstyle +$}}$
(resp. $\Phi_{\scalebox{1.0}{$\scriptscriptstyle
    0$}}^{\scalebox{1.1}{$\scriptscriptstyle -$}}$) be the set of
positive (resp. negative) roots determined by
$\alpha_{\scriptscriptstyle 1}, \ldots, \alpha_{\scriptscriptstyle
  r}.$ If we let $\theta \in \Phi_{\scalebox{1.0}{$\scriptscriptstyle
    0$}}^{\scalebox{1.1}{$\scriptscriptstyle +$}}$ denote the highest
root, then the affine roots $\alpha_{\scriptscriptstyle i}$ ($1
\le i \le r$) together with $\alpha_{\scriptscriptstyle 0}
= -\theta + \delta$ form a set of simple
roots for $\Phi.$

An affine real root $\beta$ is positive (resp. negative) relative
to the open $r$-simplex
\[
C  = \{x \in V : \text{$\alpha_{\scriptscriptstyle i}(x) >
  0$ ($1\le i \le r$) and $(-\theta + \delta)(x) > 0$}\}
\]
if $\beta(x) > 0$ (resp. $\beta(x) < 0$) for all $x \in C.$ If we let
$ 
\Phi_{\scalebox{1.2}{$\scriptscriptstyle \mathrm{re}$}}^{\scalebox{1.2}{$\scriptscriptstyle +$}} 
$
(resp.
$ 
\Phi_{\scalebox{1.2}{$\scriptscriptstyle \mathrm{re}$}}^{\scalebox{1.2}{$\scriptscriptstyle -$}} 
$)
denote the set of positive (resp. negative)\linebreak affine real roots, then
$ 
\Phi_{\scalebox{1.2}{$\scriptscriptstyle
    \mathrm{re}$}}^{\scalebox{1.2}{$\scriptscriptstyle -$}}
= - \, \Phi_{\scalebox{1.2}{$\scriptscriptstyle
    \mathrm{re}$}}^{\scalebox{1.2}{$\scriptscriptstyle +$}},
$
and
$
\Phi_{\scalebox{1.3}{$\scriptscriptstyle \mathrm{re}$}} =
\Phi_{\scalebox{1.2}{$\scriptscriptstyle
    \mathrm{re}$}}^{\scalebox{1.2}{$\scriptscriptstyle +$}}
\, \scalebox{1.8}{$\scriptscriptstyle \cup$}\, 
\Phi_{\scalebox{1.2}{$\scriptscriptstyle
    \mathrm{re}$}}^{\scalebox{1.2}{$\scriptscriptstyle -$}}.
$
The set of positive affine real roots is
\[
\Phi_{\scalebox{1.2}{$\scriptscriptstyle \mathrm{re}$}}^{\scalebox{1.2}{$\scriptscriptstyle +$}}
= \{\alpha + (n + \chi(\alpha))\delta : \alpha \in \Phi_{\scalebox{1.0}{$\scriptscriptstyle 0$}}, 
n \in \mathbb{N}\}
\]
where $\chi$ is the characteristic function of 
$\Phi_{\scalebox{1.0}{$\scriptscriptstyle 0$}}^{\scalebox{1.}{$\scriptscriptstyle  \minus$}}.$
An affine imaginary root $m\delta$ is positive or negative
according as $m > 0$ or $m < 0.$ If $\alpha$ and $\beta$ are distinct
affine roots then $\alpha < \beta$ will signify that 
$
\beta - \alpha \in \sum_{i} \mathbb{N}\alpha_{\scriptscriptstyle i}.
$

The affine {W}eyl group $W$ is also a {C}oxeter group on the generators
$
\sigma_{\scriptscriptstyle i} = \sigma_{\!\scriptscriptstyle \alpha_{\scalebox{.55}{$\scriptscriptstyle i$}}}
$
($1 \le i \le r$), and
$
\sigma_{\scriptscriptstyle 0} 
= \sigma_{\!\scriptscriptstyle -\theta + \delta},
$
subject to the relations $\sigma_{\scriptscriptstyle
  i}^{\scalebox{1.0}{$\scriptscriptstyle 2$}} = 1$ for all $0\le i \le r ,$ 
  and 
$
(\sigma_{\scriptscriptstyle i}
\sigma_{\!\scriptscriptstyle j})^{m_{\scalebox{.75}{$\scriptscriptstyle ij$}}}
\! = 1 
$ 
if $i \ne j$ and
$
m_{\scalebox{1.1}{$\scriptscriptstyle ij$}} < \infty.
$

Note that
\begin{equation} \label{eq: fixing-delta}
\text{$w\delta = \delta \circ w^{\scalebox{1.1}{$\scriptscriptstyle -1$}}
\! = \delta$ for all $w \in W.$} 
\end{equation}

\subsection{Extended Affine {W}eyl Group}
With notations as before, let $P^{\scalebox{1.1}{$\scriptscriptstyle \vee$}}$
denote the {\normalfont\itshape coweight} lattice of
$
\Phi_{\scalebox{1.0}{$\scriptscriptstyle 0$}},
$
that is,
\[
P^{\scalebox{1.1}{$\scriptscriptstyle \vee$}} =
\{\text{$\lambda \in V:
  \langle \lambda, \alpha \rangle \in \mathbb{Z}$ for all 
$\alpha \in \Phi_{\scalebox{1.0}{$\scriptscriptstyle 0$}}$}\}.
\]
The {\normalfont \itshape extended affine {W}eyl} group is defined to be
$
\widetilde{W}: = W_{\scalebox{1.1}{$\scriptscriptstyle 0$}} \ltimes
\tau(P^{\scalebox{1.1}{$\scriptscriptstyle \vee$}}),
$
where $\tau(P^{\scalebox{1.1}{$\scriptscriptstyle \vee$}})$ is the
group of translations by elements of
$P^{\scalebox{1.1}{$\scriptscriptstyle \vee$}}.$
The extended affine {W}eyl group contains
$W$ as a normal subgroup, and the quotient
$
\widetilde{W} \slash W
\cong
P^{\scalebox{1.1}{$\scriptscriptstyle \vee$}} \slash
Q^{\scalebox{1.1}{$\scriptscriptstyle \vee$}}
$
is a finite (abelian) group.

The extended affine {W}eyl group acts on the set of affine roots. To
see this, take an element $w$ of $\widetilde{W},$ and express it as
$
w = w_{\scalebox{1.1}{$\scriptscriptstyle 0$}}
\tau(\lambda),
$
with
$
w_{\scalebox{1.1}{$\scriptscriptstyle 0$}} \in
W_{\scalebox{1.1}{$\scriptscriptstyle 0$}}
$
and
$
\lambda \in P^{\scalebox{1.1}{$\scriptscriptstyle \vee$}}.
$
If $\beta = \alpha + n\delta \in \Phi,$ then
\[
 (w\beta)(x) = \beta(w^{\scalebox{1.1}{$\scriptscriptstyle -1$}}x)
= \langle w_{\scalebox{1.1}{$\scriptscriptstyle 0$}}\alpha, x \rangle + n
- \langle \lambda, \alpha \rangle \quad \;\, \text{(for $x \in V$)}.
\]
Thus
$
w\beta = w_{\scalebox{1.1}{$\scriptscriptstyle 0$}}\beta
- \langle \lambda, \alpha \rangle\delta \in \Phi
$
because
$
\langle \lambda, \alpha \rangle \in \mathbb{Z}.
$
It follows that $\widetilde{W}$ permutes the affine real roots
and fixes\linebreak the imaginary roots.

For an element $w \in W,$ we shall denote its {\normalfont \itshape length}
(with respect to the generators $\sigma_{\scriptscriptstyle i}, 0
\le i \le r$) by $\ell(w).$ If\linebreak we let
$
\Phi(w) = \{\beta \in \Phi^{+} : w\beta \in \Phi^{-}\},
$
then $\ell(w) = |\Phi(w)|.$ The length function is extended to
$\widetilde{W}$ by the same formula. If
$
w_{\scalebox{1.1}{$\scriptscriptstyle 0$}} \in
W_{\scalebox{1.1}{$\scriptscriptstyle 0$}}
$
and
$
\lambda \in P^{\scalebox{1.1}{$\scriptscriptstyle \vee$}},
$
then
\be\label{eq2}
\ell(w_{\scalebox{1.1}{$\scriptscriptstyle 0$}}
\tau(\lambda))
\; = \sum_{\alpha \in \Phi_{\scalebox{1.0}{$\scriptscriptstyle 0$}}^{\scalebox{1.1}{$\scriptscriptstyle +$}}}
|\langle \lambda, \alpha \rangle + \chi(w_{\scalebox{1.1}{$\scriptscriptstyle 0$}}\alpha)|
\ee
where, as before, $\chi$ is the characteristic function of
$\Phi_{\scalebox{1.0}{$\scriptscriptstyle
    0$}}^{\scalebox{1.}{$\scriptscriptstyle  \minus$}}.$

Let $O= \{w \in \widetilde{W} : \ell(w) = 0\}.$ The elements of $O$ permute the simple affine roots, and we have 
\[
\widetilde{W} = W \rtimes O
\]
so
$
O \cong
\widetilde{W} \slash W
\cong
P^{\scalebox{1.1}{$\scriptscriptstyle \vee$}} \slash
Q^{\scalebox{1.1}{$\scriptscriptstyle \vee$}}.
$
In particular, $O$ is a finite abelian group.

\subsection{The Chinta-Gunnells action} %\label{s3.3}
From now on we will assume that the affine root system $\Phi$ is
simply laced. Let $\alpha_{\scriptscriptstyle 0}, \ldots,
\alpha_{\scriptscriptstyle r}$ be the set of affine simple roots, where
$
\alpha_{\scriptscriptstyle 0}
= -\theta + \delta.
$
If two $\alpha_{\scriptscriptstyle i}$ and
$\alpha_{\scriptscriptstyle j}$ are connected in the Dynkin diagram of
$\Phi,$ we shall write $i \sim j.$ Define an action of $W$ on monomials 
$
\mathbf{x}^{\beta} : =
\prod_{i = 0}^{r} x_{\scalebox{1.1}{$\scriptscriptstyle i$}}^{k_{\scalebox{.75}{$\scriptscriptstyle i$}}}
$,
for $\beta = \sum_{i = 0}^{r } k_{\scriptscriptstyle i}\alpha_{\scriptscriptstyle i}$
in the root lattice of $\Phi$, by $w\xx^\b=\xx^{w^{-1}\b}$, which
corresponds to the contragredient action on roots.
%thus $\mathbf{x}^{\delta} = x_{\sss 0}\mathbf{x}^{\theta}.$
We also have an action of $W$ on variables
$
\mathbf{x}
= (x_{\scalebox{1.1}{$\scriptscriptstyle 0$}},
\ldots, x_{\scalebox{1.1}{$\scriptscriptstyle r $}})
$
by
$
(w \mathbf{x}) _{\scriptscriptstyle j} = \mathbf{x}^{w^{\scalebox{.85}{$\scriptscriptstyle -1$}}\alpha_{\scalebox{.7}{$\scriptscriptstyle j$}}}\!,
$
that is (using that $\Phi$ is simply laced):
\[
 (\sigma_{\scriptscriptstyle i}\mathbf{x})_{\scriptscriptstyle j} =
\begin{cases} 
1\slash x_{\!\scalebox{1.1}{$\scriptscriptstyle j$}}
&\mbox{if $j = i$} \\
x_{\scalebox{1.1}{$\scriptscriptstyle i$}}
x_{\!\scalebox{1.1}{$\scriptscriptstyle j$}}
&\mbox{if $j \sim i$}  \\
x_{\!\scalebox{1.1}{$\scriptscriptstyle j$}} & \mbox{otherwise}.
\end{cases}  
\]
Let $\varepsilon_{\scalebox{1.1}{$\scriptscriptstyle
    i$}}\mathbf{x}$ be the involution defined by
\[
(\varepsilon_{\scalebox{1.1}{$\scriptscriptstyle i$}}\mathbf{x})_{\scalebox{1.1}{$\scriptscriptstyle j$}} =
\begin{cases} 
- x_{\!\scalebox{1.1}{$\scriptscriptstyle j$}}
&\mbox{if $j \sim i$}  \\
x_{\!\scalebox{1.1}{$\scriptscriptstyle j$}} & \mbox{otherwise.}
\end{cases}  
\]
Let $\mathbb{C}(\mathbf{x}, u)$ be the field of rational functions
in $x_{\sss 1}, \ldots, x_{\sss r}, u$, for an additional variable $u$.
With this notation, we define the action of a simple reflection
$\sigma_{\scriptscriptstyle i}$ on $f \in \mathbb{C}(\mathbf{x}, u)$ by
\[
 f\vert\sigma_{\scriptscriptstyle i}(\mathbf{x}; u) =
 f(\sigma_{\scriptscriptstyle i}\mathbf{x}; u)
 J(x_{\scalebox{1.1}{$\scriptscriptstyle i$}}, 0)
 + f(\varepsilon_{\scriptscriptstyle i}\sigma_{\scriptscriptstyle i}\mathbf{x}; u)
 J(x_{\scalebox{1.1}{$\scriptscriptstyle i$}}, 1),
%\quad f \in \mathbb{C}(\mathbf{x}, u)
\]
where, for $\varepsilon \in \{0, 1\},$
\begin{equation} \label{eq: function-J}
J(x, \varepsilon) = J(x, u, \varepsilon) =
\frac{x}{2}\left(\frac{u - x}{1 - ux} - (-1)^{\varepsilon}\right).
\end{equation}
One verifies as in \cite[Lemma~3.2]{CG1} that this action extends to a
well-defined $W$-action on $\mathbb{C}(\mathbf{x}, u).$

To construct the analogue of the Chinta-Gunnells (average) function in
our context, let $\Delta(\mathbf{x})$ be defined by
\begin{equation} \label{eq: McD-denominator}
  \Delta(\mathbf{x}) =
  \prod_{n \, \ge \, 1}
\!\left(1 - \mathbf{x}^{2 n \delta}\right)^{r}
\, \cdot \prod_{\beta \, \in \,  
\Phi_{\scalebox{.95}{$\scriptscriptstyle \mathrm{re}$}}^{\scalebox{.95}{$\scriptscriptstyle +$}}} 
\!\left(1 - \mathbf{x}^{2\beta} \right).
\end{equation}
The product \eqref{eq: McD-denominator} is absolutely convergent in
the region $|\mathbf{x}^{\scalebox{1.1}{$\scriptscriptstyle \delta$}}|
< 1$ of $\mathbb{C}^{r \scalebox{1.15}{$\scriptscriptstyle + 1$}},$
and by \eqref{eq: fixing-delta}, it satisfies the
transfor-\linebreak mation formulas
\begin{equation} \label{eq: trans-form-McD-denominator}
\Delta(\mathbf{x}) = - x_{\scalebox{1.1}{$\scriptscriptstyle
    i$}}^{\scalebox{1.1}{$\scriptscriptstyle 2$}}
\Delta(\sigma_{\scriptscriptstyle i}\mathbf{x})
\qquad (\text{for $i = 0, \ldots, r $}).
\end{equation}
The Chinta-Gunnells function can now be defined by
$
Z_{\scriptscriptstyle W}^{\scriptscriptstyle \mathrm{CG}}(\mathbf{x}; u)
= Z_{\scriptscriptstyle W}(\mathbf{x}; u)\slash \Delta(\mathbf{x}),
$
where
\begin{equation} \label{eq: average-zeta}
Z_{\scriptscriptstyle W}(\mathbf{x}; u) \,  : = \sum_{w \, \in \, W} 1 \vert w(\mathbf{x}; u).
\end{equation}
For affine root systems, however, there is a highly non-trivial
correction of
$
Z_{\scriptscriptstyle W}^{\scriptscriptstyle \mathrm{CG}}(\mathbf{x};
u)
$  
(corresponding to the affine imaginary roots) that one should take
into account. This issue will be addressed in Section
\ref{renorm}\linebreak
when the root system is $D_{\scriptscriptstyle 4}^{\scriptscriptstyle
  (1)}\!,$ but for now let us concentrate on the function
$Z_{\scriptscriptstyle W}(\mathbf{x}; u)$ that we just
defined.\linebreak

\vskip-12pt
The following proposition (cf. \cite[Proposition~3.1.2]{White2}) gives
the largest possible region of convergence for the series \eqref{eq: average-zeta},
as a function of the complex variables
$
x_{\scalebox{1.1}{$\scriptscriptstyle 0$}}, \ldots,
x_{\scalebox{1.1}{$\scriptscriptstyle r $}}, u.
$ 

\vskip5pt
\begin{prop} \label{continuation-W-invarianceZ} --- For 
$
x_{\scalebox{1.1}{$\scriptscriptstyle 0$}}, \ldots, x_{\scalebox{1.1}{$\scriptscriptstyle r $}}, u \in \mathbb{C}
$ 
such that $|\mathbf{x}^{\scalebox{1.1}{$\scriptscriptstyle \delta$}}| < 1,$ 
the series defining $Z_{\scriptscriptstyle W}(\mathbf{x}; u)$ converges absolutely and uniformly on every compact
subset away from the points for which $u\mathbf{x}^{\beta}  \pm 1 = 0$ for some 
$ 
\beta \in \Phi_{\scalebox{1.2}{$\scriptscriptstyle \mathrm{re}$}}^{\scalebox{1.2}{$\scriptscriptstyle +$}}. 
$ 
In addition, the function $Z_{\scriptscriptstyle W}(\mathbf{x}; u)$ is $W$\!-- \!invariant under the Chinta-Gunnells action.  
\end{prop}

\begin{proof} We shall follow closely \cite{White2}. Let $w$ be a fixed {W}eyl group element, and write 
$w = \sigma_{\scriptscriptstyle i_{\scalebox{0.7}{$\scriptscriptstyle \ell$}}} \cdots \, 
\sigma_{\scriptscriptstyle i_{\scalebox{0.7}{$\scriptscriptstyle 1$}}}$ 
in reduced form. It is well-known that the set
$\Phi(w) = \{\beta \in \Phi^{+} : w\beta \in \Phi^{-}\}$ is explicitly
given by the $\ell$ distinct positive roots 
%\be \label{eq7}
\[
\Phi(w) = \{\beta_{1} = \alpha_{\scriptscriptstyle i_{\scalebox{0.7}{$\scriptscriptstyle 1$}}},\,  
\beta_{2} =  \sigma_{\scriptscriptstyle i_{\scalebox{0.7}{$\scriptscriptstyle 1$}}}\alpha_{\scriptscriptstyle i_{\scalebox{0.7}{$\scriptscriptstyle 2$}}}, \,
\beta_{3} =  \sigma_{\scriptscriptstyle i_{\scalebox{0.7}{$\scriptscriptstyle 1$}}}
\sigma_{\scriptscriptstyle i_{\scalebox{0.7}{$\scriptscriptstyle 2$}}}
\alpha_{\scriptscriptstyle i_{\scalebox{0.7}{$\scriptscriptstyle 3$}}}, \ldots, \,
\beta_{\ell} =  \sigma_{\scriptscriptstyle i_{\scalebox{0.7}{$\scriptscriptstyle 1$}}}
\sigma_{\scriptscriptstyle i_{\scalebox{0.7}{$\scriptscriptstyle 2$}}} \cdots\, 
\sigma_{\scriptscriptstyle i_{\scalebox{0.7}{$\scriptscriptstyle \ell - $} 
\scalebox{0.77}{$\scriptscriptstyle 1$}}}
\alpha_{\scriptscriptstyle i_{\scalebox{0.7}{$\scriptscriptstyle \ell$}}}
\}.\]
%\ee 
Using the fact that
$
f \vert w_{\scalebox{1.1}{$\scriptscriptstyle 1$}}\vert w_{\scalebox{1.1}{$\scriptscriptstyle 2$}}
= f \vert w_{\scalebox{1.1}{$\scriptscriptstyle 1$}}w_{\scalebox{1.1}{$\scriptscriptstyle 2$}},
$
one can verify by induction on $\ell$ the formula: 
\begin{equation} \label{eq: 1-acted-by-w}
f \vert w(\mathbf{x}; u) 
\; =  \sum_{\delta_{\scalebox{0.7}{$\scriptscriptstyle 1$}}, \ldots,  
\delta_{\scalebox{0.7}{$\scriptscriptstyle \ell$}} \, \in \, \{0, 1\}}
 f\!\left(w \e^{\sss \sum_{i=1}^\ell \dd_{\scalebox{0.7}{$\scriptscriptstyle i$}}
 \beta_{\scalebox{0.7}{$\scriptscriptstyle i$}}} \xx;u \right) 
\,\prod_{k = 1}^{\ell} 
J\!\left((-1)^{\scriptscriptstyle
    \langle \beta_{\scalebox{0.7}{$\scriptscriptstyle k$}},\, \sum_{i \, < \, k} 
\delta_{\scalebox{0.7}{$\scriptscriptstyle i$}} 
\beta_{\scalebox{0.7}{$\scriptscriptstyle i$}}\rangle}
\mathbf{x}^{\scriptscriptstyle \beta_{\scalebox{0.7}{$\scriptscriptstyle k$}}}\!, 
\delta_{\scalebox{1.1}{$\scriptscriptstyle k$}} \right)
\end{equation}
where
$
\varepsilon^{\alpha} : =
\prod_{i} \varepsilon_{\scalebox{1.1}{$\scriptscriptstyle i$}}^{k_{\scalebox{.75}{$\scriptscriptstyle i$}}},
$
for
$
\alpha = \sum_{i} k_{\scriptscriptstyle i}\alpha_{\scriptscriptstyle i},
$
and $J$ is defined by \eqref{eq: function-J}; we are interpreting the factor corresponding to $k = 1$ to be 
$
J(\mathbf{x}^{\scriptscriptstyle \beta_{\scalebox{0.7}{$\scriptscriptstyle 1$}}}\!, 
\delta_{\scalebox{1.1}{$\scriptscriptstyle 1$}}).
$ 

To estimate $1 \vert w(\mathbf{x}; u),$ express each root $\beta \in \Phi(w)$ as 
$\beta = \alpha + (n + \chi(\alpha))\delta$ with 
$\alpha \in \Phi_{\scalebox{1.0}{$\scriptscriptstyle 0$}},$ 
$n \in \mathbb{N},$ and $\chi$ the characteristic function of 
$\Phi_{\scalebox{1.0}{$\scriptscriptstyle 0$}}^{\scalebox{1.}{$\scriptscriptstyle  \minus$}}.$ 
By assuming that these roots are as small as possible, one finds that 
\[
{\displaystyle 
\sum_{\alpha + (n + \chi(\alpha))\delta \, \in \,  \Phi(w)}
n + \chi(\alpha) 
\, \ge  \, \frac{|\Phi_{\scalebox{1.0}{$\scriptscriptstyle 0$}}|
\left\lfloor \frac{\ell}{|\Phi_{\scalebox{1.0}{$\scriptscriptstyle 0$}}|}\right\rfloor
\left(\left\lfloor \frac{\ell}{|\Phi_{\scalebox{1.0}{$\scriptscriptstyle 0$}}|}\right\rfloor  - 1\right)}{2}}
\] 
where $\lfloor x \rfloor$ denotes the integer part of $x,$ and 
$|\Phi_{\scalebox{1.0}{$\scriptscriptstyle 0$}}|$ is the cardinality of 
$\Phi_{\scalebox{1.0}{$\scriptscriptstyle 0$}}.$ One can also see that, for a positive affine real root 
$
\beta = \alpha + (n + \chi(\alpha))\delta,
$ 
\[
\left|J\!\left(\mathbf{x}^{\beta}\!, \e \right)\right|  \le 
|\mathbf{x}^{\scalebox{1.1}{$\scriptscriptstyle (n + \chi(\alpha))\delta$}}|
\, K(\mathbf{x}, u)\slash 2
\] 
where we set 
\[
K(\mathbf{x}, u) := 
\left(1 + |u|\right)
\max_{\alpha \, \in \, \Phi_{\scalebox{0.75}{$\scriptscriptstyle 0$}}} 
\left\{|\mathbf{x}^{\alpha}|\left(1 + |\mathbf{x}^{\alpha}|\right)\right\} 
\, \cdot \sup_{\beta \, \in \, \Phi_{\scalebox{.95}{$\scriptscriptstyle \mathrm{re}$}}^{\scalebox{.95}{$\scriptscriptstyle +$}}} 
\left|1 - u\mathbf{x}^{\beta}\right|^{\sss -1};
\] 
the supremum is finite since $|\mathbf{x}^{m\delta}| \to 0$ as $m \to \infty,$ 
and $u\mathbf{x}^{\beta} \ne \pm 1$ for all 
$ 
\beta \in \Phi_{\scalebox{1.2}{$\scriptscriptstyle \mathrm{re}$}}^{\scalebox{1.2}{$\scriptscriptstyle +$}}. 
$ 
It follows that 
\be\label{eq8}
|1 \vert w(\mathbf{x}; u)| \le |\mathbf{x}^{\scalebox{1.1}{$\scriptscriptstyle \delta$}}|^{\scalebox{1.05}
{$\scriptscriptstyle \frac{\ell^{\scalebox{.8}{$\scriptscriptstyle 2$}}}{2 |\Phi_{\scalebox{.65}{$\scriptscriptstyle 0$}}|} 
- O(\ell)$}}K(\mathbf{x}, u)^{\ell}
\ee
and thus 
\begin{equation*}
\begin{split}
\sum_{w \, \in \, W} |1 \vert w(\mathbf{x}; u)| \, & \le \,  
\sum_{\ell \, \ge \,  0}\sum_{\substack{w \, \in \, W\\ \ell(w) \, = \, \ell}} 
|\mathbf{x}^{\scalebox{1.1}{$\scriptscriptstyle \delta$}}|^{\scalebox{1.05}
{$\scriptscriptstyle \frac{\ell^{\scalebox{.8}{$\scriptscriptstyle 2$}}}{2 |\Phi_{\scalebox{.65}{$\scriptscriptstyle 0$}}|} 
- O(\ell)$}}K(\mathbf{x}, u)^{\ell}\\
& < \, \sum_{\ell \, \ge \,  0} 
(r + 1)^{\ell} \, |\mathbf{x}^{\scalebox{1.1}{$\scriptscriptstyle \delta$}}|^{\scalebox{1.05}
{$\scriptscriptstyle \frac{\ell^{\scalebox{.8}{$\scriptscriptstyle 2$}}}{2 |\Phi_{\scalebox{.65}{$\scriptscriptstyle 0$}}|} 
- O(\ell)$}}K(\mathbf{x}, u)^{\ell}
\end{split}
\end{equation*} 
where, for the last inequality, we applied the trivial bound 
$
\# \{w \in W : \ell(w)  = \ell \} \le r^{\ell - 1}(r + 1) < (r + 1)^{\ell}.
$
\linebreak 
The last series converges since $|\mathbf{x}^{\scalebox{1.1}{$\scriptscriptstyle \delta$}}| < 1$ and 
$u\mathbf{x}^{\beta} \ne \pm 1$ for all 
$ 
\beta \in \Phi_{\scalebox{1.2}{$\scriptscriptstyle \mathrm{re}$}}^{\scalebox{1.2}{$\scriptscriptstyle +$}}. 
$ 
This implies our first assertion. 

Finally, for 
$
x_{\scalebox{1.1}{$\scriptscriptstyle 0$}}, \ldots, x_{\scalebox{1.1}{$\scriptscriptstyle r $}}, u
$ 
in the region of absolute convergence, we have 
\[
Z_{\scriptscriptstyle W} \vert w'(\mathbf{x}; u) 
\, = \sum_{w \, \in \, W} 1 \vert ww'(\mathbf{x}; u) 
= Z_{\scriptscriptstyle W}(\mathbf{x}; u)
\] 
which completes the proof. 
\end{proof}

\begin{rem} \!The singularities $u\mathbf{x}^{\beta}  \pm 1 = 0$ ($
	\beta \in \Phi_{\scalebox{1.2}{$\scriptscriptstyle \mathrm{re}$}}^{\scalebox{1.2}{$\scriptscriptstyle +$}} 
	$) of the function $Z_{\scriptscriptstyle W}(\mathbf{x}; u)$ are at most simple poles, and the residues at these poles of a closely related function will be evaluated in \cite{DIPP} by generalizing the method
	introduced in the present paper to arbitrary simply laced affine root systems.
\end{rem}

We conclude this subsection by extending the Chinta-Gunnells action to
the group $\widetilde{W},$ which is done as follows.
An element $w \in \widetilde{W}$ acts on the multivariable
$\mathbf{x}$ by
$
(w \mathbf{x}) _{\scriptscriptstyle j}
= \mathbf{x}^{w^{\scalebox{.85}{$\scriptscriptstyle -1$}}\alpha_{\scalebox{.7}{$\scriptscriptstyle j$}}}\!,
$
and for $\eta \in O$ and a function $f(\mathbf{x}; u),$ we define
\[
 f\vert \eta (\mathbf{x}; u) = f(\eta\mathbf{x}; u).
\]
This defines an action of $\widetilde{W} = W \rtimes O,$
and to check that it is well-defined, it suffices to check
its compati-\linebreak bility with the commutation relations between the elements
of $O$ and the simple reflections $\sigma_{\scriptscriptstyle i}.$
Since $W$ is a normal subgroup of $\widetilde{W},$
it follows at once from \cite[(2.2.5) and (2.2.6)]{Mac1} that, for
$\eta \in O,$ we have:
$
\eta \alpha_{\scriptscriptstyle i} = \alpha_{\scriptscriptstyle j}
$
if and only if
$
\eta \sigma_{\scriptscriptstyle i}
= \sigma_{\scriptscriptstyle j} \eta.
$
Moreover, one checks that
\[
\text{$\varepsilon^{\alpha} w
  = w \varepsilon^{w^{\scalebox{.85}{$\scriptscriptstyle -1$}}\alpha}$
  for all $\alpha$ in the root lattice and $w \in \widetilde{W}$}
\]
with $\varepsilon_{\scalebox{1.1}{$\scriptscriptstyle i$}}$ the sign
action defined at the beginning of this subsection. It is
easy to check that
$
f\vert \eta \vert \sigma_{\scriptscriptstyle i} 
= f\vert \sigma_{\scriptscriptstyle j} \vert \eta,
$
for $\eta \in O$ and $i, j$ such that
$
\eta \alpha_{\scriptscriptstyle i} = \alpha_{\scriptscriptstyle j},
$
so the action is, indeed, well-defined.
\vskip5pt
\begin{rem} One could also define the function
$Z_{\sss \wW}=\sum_{w\in \wW}1|w$, but it follows from the 
definition of
the extended action that $Z_{\sss \wW}=n\cdot Z_{\sss W}$, 
with $n$ the cardinality of $O$. For this reason, the extension of
the {C}hinta-{G}unnells action to $\wW$ will in fact not be needed.
However, with no additional effort required, some
of the relations used in the proofs of our results will be stated
(for completeness) in the extended {W}eyl group.
\end{rem}

\subsection{Chinta-Gunnells averages of polynomials}
We will also need the convergence and properties of the
Chinta-Gunnells average
\[
\ZWg:=\sum_{w \, \in \, W} g\vert w
\]
for any Laurent polynomial $g(\xx)$ with coefficients in
$\CC(u)$. Note that $Z_{\sss W, 1}=\ZW$.

\vskip5pt
\begin{prop} --- For any Laurent polynomial $g(\xx),$ there is
  an integer $N = N_g\ge 0$ such that the series defining
  $\xx^{\scalebox{1.1}{$\scriptscriptstyle N\dd$}} \ZWg(\xx)$
  converges absolutely and uniformly on compacta in the
  same region as $\ZW(\xx),$ and it is $W$-invariant under
  the Chinta-Gunnells action. 
\end{prop}

\begin{proof} By linearity, it is enough to consider the case
  $g(\xx)=\xx^{\a}$ for
  $
  \a=a_{\sss 0} \a_{\sss 0}+\cdots + a_{\sss r} \a_{\sss r}
  $
  an element in the\linebreak root lattice of $\Phi.$
  Using~\eqref{eq: 1-acted-by-w} and~\eqref{eq8},
  it also suffices to show that
  $
  |\xx^{w^{\scalebox{.85}{$\scriptscriptstyle -1$}} \a }| \ll |\xx^{\dde }|^{\sss -O(\ell)},
  $
  for the length $\ell$ of $w$ sufficiently large, say
  $
  \ell \ge \ell(w_{\sss 0})
  $
  for all $w_{\sss 0} \in W_{\sss 0}$; here we are taking
  $
  \xx = (x_{\sss 0}, \ldots, x_{\sss r})
  $
  in a compact set with $x_{\sss i} \ne 0$ for all $i.$

  We decompose $w=w_{\sss 0} t,$ with $w_{\sss 0}\in W_{\sss 0}$ and
  $t=\tau(\lam)$ a translation with $\lam\in P^\vee.$
  Letting $\mu_{\sss i}\in P^\vee$ be the fun-\linebreak
  damental coweights ($i=1,\ldots,r$), we can write
  $
  t^{\sss -1}=\prod_{i=1}^{r} 
\tau(\mu_{\sss i})^{n_{\scalebox{.75}{$\scriptscriptstyle i$}}}
$ 
for some $n_{\sss i}\in \Z;$ we have 
\[
\tau(\mu_{\sss i})\a_{\sss i}=\a_{\sss i}-\dd, \;\; 
\tau(\mu_{\sss i})\a_{\sss 0}=\a_{\sss 0}+m_{\sss i} \dd \;\;\;\, \text{and} \;\;\;\,
\tau(\mu_{\sss i})\a_{\sss j}=\a_{\sss j} \;\text{$(j\ne 0, i)$}
\] 
where $\theta=\sum_{i=1}^r m_{\sss i} \a_{\sss i}$ is the highest root 
in $\Phi_{\sss 0}.$ It follows from~\eqref{eq2} that 
\[
\sum_{i=1}^{r} |n_{\sss i}|  \le 
\ell(t)\; = 
\sum_{\alpha \in \Phi_{\scalebox{1.0}{$\scriptscriptstyle
      0$}}^{\scalebox{1.1}{$\scriptscriptstyle +$}}} |\langle \lambda, \alpha \rangle|
\le \ell(w_{\sss 0}) + \ell \le 2\ell
\] 
where the lower bound of $\ell(t)$ was obtained by retaining only the 
terms corresponding to 
$ 
\a = \a_{\sss i}\in \Phi_{\sss 0}^{\scalebox{1.1}{$\scriptscriptstyle +$}}.
$

On the other hand, if we write 
$
w_{\sss 0}^{\sss -1}\a=\sum_{j=0}^{r} k_{\sss j} \a_{\sss j},
$ 
then 
$
w^{\sss -1}\a=t^{\sss -1} w_{\sss 0}^{\sss -1} \a 
= w_{\sss 0}^{\sss -1}\a + n\dd,
$ 
with $n$ given by 
\[ 
n = \langle \lambda, w_{\sss 0}^{\sss -1}\a\rangle
= \sum_{i=1}^{r} (k_{\sss 0}m_{\sss i} - k_{\sss i})n_{\sss i}.
\] 
Thus 
\[ 
|n| \le \sum_{i=1}^{r} |k_{\sss 0}m_{\sss i} - k_{\sss i}||n_{\sss i}|
\le C\sum_{i=1}^{r} |n_{\sss i}| \le 2C\ell
\] 
for a positive constant $C$ depending only upon $\Phi_{\sss 0}$ and
$\a.$ It follows that
\[
|\xx^{w^{\sss -1} \a}| 
\le |\xx^{w_{\scalebox{.8}{$\scriptscriptstyle 0$}}^{\sss -1}\a}| 
|\mathbf{x}^{\scalebox{1.1}{$\scriptscriptstyle \delta$}}|^{\scalebox{1.1}{$\scriptscriptstyle - |n|$}} 
\ll |\mathbf{x}^{\scalebox{1.1}{$\scriptscriptstyle
    \delta$}}|^{\scalebox{1.1}{$\scriptscriptstyle - 2C\ell$}} 
\]
where the implied constant can be taken to be the maximum over
$\xx$ in the compact set and $w_{\sss 0}\in W_{\sss 0}$ of
$
|\xx^{w_{\scalebox{.8}{$\scriptscriptstyle 0$}}^{\sss -1}\a}|.
$
Accordingly $|g\vert w|$ satisfies an estimate of
type~\eqref{eq8}. It is also clear that there are only finitely many
$w\in W$ such that $g\vert w$ has poles when $x_{i}= 0,$
which completes the proof.
\end{proof}

For brevity, we state the following result only for the case 
$D^{\sss (1)}_{\sss 4}$, the general case being considered in \cite{DIPP}.\linebreak 
\begin{prop}\label{p3.3} --- Assume the root system $\Phi$ is affine 
of type $D^{\sss (1)}_{\sss 4}.$ Then for every monomial~$g,$ 
the function $\ZWg$ satisfies
\[ 
\ZWg(\xx) = C_{g}(\xx^{\scalebox{1.1}{$\scriptscriptstyle \delta$}})\ZW(\xx)
\] 
where $C_{g}(x)$ is a Laurent polynomial with coefficients in
$\Z[u],$ which can be determined recursively in terms of $g.$ 
\end{prop}

\begin{proof} We use the labelling of simple roots for 
$D^{\sss (1)}_{\sss 4}$ given in the beginning of the next section, 
so that 
$
\delta =
\alpha_{\scriptscriptstyle 1}
+ \alpha_{\scriptscriptstyle 2}
+ \alpha_{\scriptscriptstyle 3}
+ \alpha_{\scriptscriptstyle 4}
+ 2\alpha_{\scriptscriptstyle 5}. 
$
We abbreviate $Z_g=\ZWg$ in this proof.

Let $g(\xx)=\xx^{\a}$ be a monomial with
$
\a = a_{\sss 1} \a_{\sss 1}+\cdots + a_{\sss 5} \a_{\sss 5}
$,
with $a_{\sss i}\in\Z$. Then $g(\xx)$ is even (resp. odd) for the\linebreak
involution $\e_{\sss i}$ if and only if
$v_{\sss i}(g):=\sum_{j\sim i} a_{\sss j}$
is even (resp. odd). Note that in this case we have only two sign\linebreak
functions, $\e_{\sss 1}=\e_{\sss i}$ for $i=1,\ldots, 4$, and
$\e_{\sss 5}$.

We use the following properties of the Chinta-Gunnells action.
If $g(\xx)$ is odd for the involution $\e_{\sss i}$, then
\[
  g \vert \ss_{\sss i}=-x_{\sss i} \ss_{\sss i} g
\]
and if $g(\xx)$ is even for $\e_{\sss i}$, then
\[
[(x_{\sss i} - u) g]\vert \ss_{\sss i} = (u  - x_{\sss i})  \ss_{\sss i} g.
\]
By the previous lemma, we can act with $w\in W$ and sum
these relations over $W$ to obtain (replacing first $g$
by $g\slash x_{\sss i}$ in the even case):
\be\label{eq10}
Z_{g} = \begin{cases} - Z_{x_{\scalebox{.75}{$\scriptscriptstyle i$}}
    \ss_{\scalebox{.75}{$\scriptscriptstyle i$}} g}
  & \text{for $v_{\sss i}(g)$ odd} \\
  u Z_{x_{\scalebox{.75}{$\scriptscriptstyle i$}}
    \ss_{\scalebox{.75}{$\scriptscriptstyle i$}} g} +
  u Z_{g\slash x_{\scalebox{.75}{$\scriptscriptstyle i$}}}
  - Z_{x_{\scalebox{.75}{$\scriptscriptstyle
        i$}}^{\scalebox{.75}{$\scriptscriptstyle 2$}}
    \ss_{\scalebox{.75}{$\scriptscriptstyle i$}} g} &
  \text{for $v_{\sss i}(g)$ even}.
    \end{cases}
    \ee
    We now introduce an order on the monomials $g$.
    Let $\db (g)$ be the $8$-tuple consisting of the differences: 
\be\label{eq9} \pm (a_{\sss 1}-a_{\sss 2}),\ \pm(a_{\sss 5}-a_{\sss 1}-a_{\sss 2}), \ \pm (a_{\sss 3}+a_{\sss 4}-a_{\sss 5}), \ \pm(a_{\sss 3}-a_{\sss 4})
\ee
ordered decreasingly. We order the monomials $g$
by the lexicographic order of the tuples $\db (g).$
Note that\linebreak
$g = \xx^{n\dd}$ if and only if  $\db(g) = (0,\ldots, 0),$
and that these are the smallest elements for the order
just defined.\linebreak
Using the relations above, we show that 
$Z_{g}$ with $g\ne \xx^{n\dd}$ can be expressed 
in terms of $Z_{g'}$ with $\db(g')<\db(g),$ which shows
recursively that $Z_{g}$ can be expressed in terms of
$Z_{\xx^{n\dd}} = \xx^{n\dd}Z_{1},$ with coefficients in $\Z[u].$

The differences in~\eqref{eq9} containing $a_{\sss i}$ are of type
$\pm (a_{\sss i}-b_{\sss i})$, $\pm(a_{\sss i}-c_{\sss i})$, 
with $b_{\sss i}+c_{\sss i}=v_{\sss i}(g)$. For each index $i,$ we define 
two differences:
\[
  d_{\sss i}^{\sss (1)}=a_{\sss i}-b_{\sss i} \;\;\; \text{and} \;\;\;
  d_{\sss i}^{\sss (2)}=c_{\sss i}-a_{\sss i}
\]
(fixing throughout a choice of $b_{\sss i}, c_{\sss i},$ e.g.,
\!$b_{\sss 1}=a_{\sss 2},$ $c_{\sss 1}=a_{\sss 5}-a_{\sss 2}$, etc.).
Since $\ss_{\sss i} g=g x_{\sss i}^{v_{\scalebox{.75}{$\scriptscriptstyle i$}}(g)-2a_{\scalebox{.75}{$\scriptscriptstyle i$}}}\!,$
and $b_{\sss i}+c_{\sss i}\linebreak =v_{\sss i}(g),$
the differences $\pm d_{\sss i}^{\sss (1)}$ and
$\pm d_{\sss i}^{\sss (2)}$ are switched for $\ss_{\sss i} g,$ and 
the other four differences are the same for $g$ and $\ss_{\sss i} g,$ so $\db (\ss_{\sss i} g)=\db(g).$

If $g\ne \xx^{n\dd},$ then there exists an index~$i$
with $d_{\sss i}^{\sss (1)} > d_{\sss i}^{\sss (2)};$ indeed,
if $d_{\sss i}^{\sss (1)}\le d_{\sss i}^{\sss (2)},$ that is
$2a_{\sss i}\le v_{\sss i}(g)$ for all $i,$ then
one must have equality for all $i,$ and so
$g = \xx^{n\dd}$ for some $n$, a contradiction.
We claim that the relations~\eqref{eq10} for this choice
of $i,$ express $Z_{g}$ in terms of $Z_{g'},$
with $\db (g') <\db (g),$ hence finishing the proof. We\linebreak
have $-d_{\sss i}^{\sss (2)}>-d_{\sss i}^{\sss (1)}$ as well, and 
the pairs $(d_{\sss i}^{\sss (1)}, d_{\sss i}^{\sss (2)})$,   
$(-d_{\sss i}^{\sss (2)}, -d_{\sss i}^{\sss (1)})$
for $x_{\sss i}\ss_{\sss i} g$ (resp. $x_{\sss i}^{\sss 2}\ss_{\sss i} g$) are
\[
 (d_{\sss i}^{\sss (2)}+1,d_{i}^{\sss (1)}-1),\ 
( -d_{\sss i}^{\sss (1)}+1,-d_{i}^{\sss (2)}-1) 
\;\;\;\; (\text{resp. } (d_{\sss i}^{\sss (2)}+2, d_{\sss i}^{\sss (1)}-2),\
(-d_{\sss i}^{\sss (1)}+2, -d_{\sss i}^{\sss (2)}-2)).
\]
It follows that
$
\db(g\slash x_{\sss i})=\db(x_{\sss i} \ss_{\sss i} g )<\db(g),$
unless $d_{\sss i}^{\sss (1)}=d_{\sss i}^{\sss (2)}+1,$
in which case $g=x_{\sss i} \ss_{\sss i} g$ and $Z_{g}=0.$
Simi-\linebreak larly, if $v_{\sss i}(g)$ is even, we have
$
\db(x_{\sss i}^{\sss 2} \ss_{\sss i} g )<\db(g),$
unless $d_{\sss i}^{\sss (1)}=d_{\sss i}^{\sss (2)}+2,$ in which case  
$
g=x_{\sss i}^{\sss 2} \ss_{\sss i} g$ and
$Z_{g}=u Z_{g\slash x_{\sss i}}$ with
$\db(g\slash x_{\sss i})<\db(g).$

The proof clearly gives an algorithm for computing recursively
the polynomial $C_{g}(x)$ in terms of $g.$
\end{proof}
\vskip5pt
\begin{rem}   
For $D_{\sss 4}=\la\ss_{\sss 2}, \ldots, \ss_{\sss 5} \ra$, 
one can define the same 8-tuple $\db(g)$ for a monomial $g$ in
$x_{\sss 2},\ldots, x_{\sss 5}$, by setting $a_{\sss 1}=0$
in~\eqref{eq9}. The above proof then shows that for all monomials
$g$, we have $Z_{\sss W, g}=p(u)Z_{\sss W, h},$ for a polynomial $p$
and a monomial
$
h=\prod x_{\sss i}^{a_{\scalebox{.75}{$\scriptscriptstyle i$}}}
$
such that $2a_{\sss i}\le v_{\sss i}(h)$ for $i=2, \ldots, 5$.
However, there are infinitely many such~$h$,
for example $h = x_{\sss 2}^b x_{\sss 3}^a x_{\sss 4}^a x_{\sss 5}^{2a}$ for $2a\le b \le a$. 
Therefore the previous proposition fails for finite {W}eyl groups.
\end{rem}

\begin{rem}

When specializing $u=-1$, the quotient of the averages appearing in Proposition~\ref{p3.3},
for an arbitrary simply laced affine root system $\Phi$, becomes
%the action becomes $f|w(\xx)=f(w\xx) \cdot (-1)^{\ell(w)}\prod_{\a\in \Phi(w)}\xx^\a$,
\[\frac{\ZWg}{\ZW}(\xx;-1)=\chi_\lam(\xx):=
\frac{\sum_{w\in W} (-1)^{\ell(w)} \xx^{w(\lam-\rho)+\rho}}
{\sum_{w\in W} (-1)^{\ell(w)} \xx^{\rho-w\rho}}
%=C_g(\xx^\dd)|_{u=-1}
\]
where $g=\xx^\lam$, and $\rho=\om_{\sss 0}+\cdots+\om_{\sss r}$, with
$\om_{\sss i}$ the affine fundamental weights. If $\lam$ were an anti-dominant affine weight,
then $\chi_\lam$ would be the character of the infinite-dimensional representation of
the associated affine Kac-Moody Lie algebra with lowest weight~$\lam$.
However, in our situation $\lam$ is an element of the affine root lattice, and letting $w_0\in W$ be such that
$w_0(\lam-\rho)=\mu$ is the unique anti-dominant affine weight in the $W$-orbit of $\lam-\rho$,
there are two possibilities. Either $\mu$ is singular, i.e., it is fixed by a simple reflection, and then $\chi_\lam=0$; or $\mu$ is regular, in which case one checks easily that it must be of the form $\mu=n\dd-\rho$ for some $n\in \Z$, and $\chi_\lam(\xx)=(-1)^{\ell(w_0)} \xx^{n\dd}$. We conclude that the specialization at $u=-1$
of the function $C_g(x)$ in Proposition~\ref{p3.3}  is either $0$ or $\pm x^n$, for some $n\in\Z$. For example,
when $\Phi$ is of type $D_{\sss 4}^{\sss (1)}$ and $g=x_{\sss 1}^2 x_{\sss 2}^2 x_{\sss 3}^2$, we have
\[C_g(x)=-u^4 x^{-1}+(u^2-1)^3 ,\; \text{and} \;\, C_g(x)|_{u=-1}=-x^{-1}.
\]
\end{rem}

\section{An extra functional equation}\label{s4}
From now on we take the affine root system $\Phi$ to be of type
$D_{\scriptscriptstyle 4}^{\scriptscriptstyle (1)}\!,$
see \cite[p. 54, TABLE Aff 1]{Kac}. The Dynkin\linebreak diagram of
the finite root system $\Phi_{\scalebox{1.0}{$\scriptscriptstyle 0$}}
= D_{\scriptscriptstyle 4}$ is
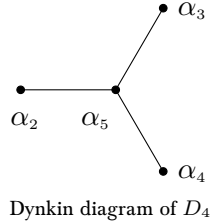
\[
\underset{\text{Dynkin diagram of
$D_{\scriptscriptstyle 4}$}}{\begin{dynkinDiagram}[text style/.style={scale=1.,black},
  edge length=1.25cm,
  labels={2,5,3,4},label
  macro/.code={\alpha_{\scalebox{.9}{$\scriptscriptstyle
        \drlap{#1}$}}}]D{4}
\end{dynkinDiagram}}
\]
and the Dynkin diagram of $D_{\scriptscriptstyle
  4}^{\scriptscriptstyle (1)}$ is \dynkin [extended]D{4} with the
additional simple root denoted by $\alpha_{\scriptscriptstyle 1};$
note the shift in\linebreak the labeling of simple roots compared to
the previous section. The set of positive roots
$
\Phi_{\scalebox{1.0}{$\scriptscriptstyle 0$}}^{\scalebox{1.1}{$\scriptscriptstyle +$}}
$
is given explic-\linebreak itly by
\[
\Phi_{\scalebox{1.0}{$\scriptscriptstyle 0$}}^{\scalebox{1.1}{$\scriptscriptstyle +$}}
 = \{\alpha_{\scriptscriptstyle i},\;
  \alpha_{\scriptscriptstyle 5},\;
  \alpha_{\scriptscriptstyle i} +
  \alpha_{\scriptscriptstyle 5},\;
\alpha_{\scriptscriptstyle i} +
  \alpha_{\scriptscriptstyle j} +
  \alpha_{\scriptscriptstyle 5},\;
  \alpha_{\scriptscriptstyle 2} +
  \alpha_{\scriptscriptstyle 3} +
  \alpha_{\scriptscriptstyle 4} +
  \alpha_{\scriptscriptstyle 5},\;
 \alpha_{\scriptscriptstyle 2}
+ \alpha_{\scriptscriptstyle 3}
+ \alpha_{\scriptscriptstyle 4}
+ 2\alpha_{\scriptscriptstyle 5}
\}_{\scalebox{1.1}{$\scriptscriptstyle 2\, \le \, i\, \ne \, j\, \le\, 4$}};
\]
thus the highest root is
$
\theta =
\alpha_{\scriptscriptstyle 2}
+ \alpha_{\scriptscriptstyle 3}
+ \alpha_{\scriptscriptstyle 4}
+ 2\alpha_{\scriptscriptstyle 5},
$
and
$
\delta =
\alpha_{\scriptscriptstyle 1}
+ \alpha_{\scriptscriptstyle 2}
+ \alpha_{\scriptscriptstyle 3}
+ \alpha_{\scriptscriptstyle 4}
+ 2\alpha_{\scriptscriptstyle 5}.
$ 
The {W}eyl group of $\Phi$ is 
$
W = \langle \sigma_{\scriptscriptstyle
  i}\rangle_{\scalebox{.95}{$\scriptscriptstyle i = 1, \ldots, 5$}},
$ 
where 
$
\sigma_{\scriptscriptstyle i} = \sigma_{\!\scriptscriptstyle
  \alpha_{\scalebox{.55}{$\scriptscriptstyle i$}}}\!.
$ 

%In this context, one can represent the Chinta-Gunnells action as follows. 
\subsection{A cocycle associated with the Chinta-Gunnells action}\label{s4.1}
For any function $f(\mathbf{x}),$ and $a, b \in \{\mathcal{e},
\mathcal{o}\},$ where $\mathcal{e}$ (resp. \!$\mathcal{o}$) stands for
even (resp. \!odd), let $f^{\scalebox{1.2}{$\scriptscriptstyle a, b$}}$ 
denote the part of $f$ which has parity $a$ with respect to the
involution $\e_{\sss 5}(\underline{x}, x_{\sss 5})
= (-\underline{x}, x_{\sss 5})$ and parity $b$ with respect to\linebreak 
the involution $\e_{\sss 1}(\underline{x}, x_{\sss 5})
= (\underline{x}, -x_{\sss 5}),$ where 
$\underline{x}=(x_{\sss 1},\ldots, x_{\sss 4})$.
For the following lemma, it is convenient to write the
Chinta-Gunnells action as
\be\label{eq12}
f\vert \ss_{\sss i}(\xx)=
x_{\sss i}\frac{u- x_{\sss i}}{1-u x_{\sss i}}
f_{\sss i}^+(\ss_{\sss i}\xx)-x_{\sss i} f_{\sss i}^-(\ss_{\sss i}\xx) 
\ee
where
$
f_{\sss i}^{\pm}(\xx)=(f(\xx)\pm f(\e_{\sss i}\xx))\slash 2
$
are the even and odd components of $f(\xx)=f(\xx;u)$ with 
respect to $\e_{\sss i}.$\linebreak
\begin{lem}\label{L4.1} --- The Chinta-Gunnells action
  of $w\in\wW$ preserves the subspace 
  $$
  \CC(\xx,u)_{\sss 0}:=\{f\in\CC(\xx,u) \mid f^{\oo}=0 \}.
$$
\end{lem}
\begin{proof} Assume $f^{\oo} \!=0$. Then for all $i$
  and $j\sim i$, we have
  \[
    (f\vert\ss_{\sss i})^{\oo}(\xx) =
    [[-x_{\sss i} f_{\sss i}^{-}(\ss_{\sss i}\xx)]^-_{\sss i}]^-_{\sss j}
    =-x_{\sss i} f^{\oo}(\ss_{\sss i}\xx)=0.
\]  
Similarly, $(f\vert \eta)^{\oo} \!=0$ for all $\eta\in O.$ 
\end{proof}

In particular, $(1\vert w)^{\oo} \! =0$ for all
$w\in \wW$, so $Z_{\scriptscriptstyle W}^{\oo} \equiv 0.$
We restrict henceforth the Chinta-Gunnells action
to the invariant subspace $\CC(\xx,u)_{\sss 0}.$

Letting $\bar{f}$ be the column vector
%\be\label{5}
\[\bar{f}: = \!\, ^{t}\!\left(f^{\eee}, f^{\eo}, f^{\oe} \right)\]
%\ee
we define a $3\times 3$ matrix $\Lambda_{w}(\xx)$ such that 
\be\label{6}
\overline{f \vert w}(\mathbf{x})
= \Lambda_{w}(\mathbf{x})
\bar{f}(w\mathbf{x}) \qquad  \text{(for $w \in \wW$)}
\ee 
with $\Lambda_{w}$ satisfying the $1$-cocycle relation 
$
\Lambda_{w w\scalebox{.75}{$\scriptscriptstyle '$}}(\mathbf{x})
= \Lambda_{w\scalebox{.75}{$\scriptscriptstyle '$}}(\mathbf{x})
\Lambda_{w}(w{\scriptscriptstyle '} \mathbf{x}) 
$ 
for all $w, w{\scriptscriptstyle '} \in \wW.$
On the generators $\sigma_{\scriptscriptstyle i},$ this cocycle 
is given by
\[
\Lambda_{\sigma_{\scalebox{0.75}{$\scriptscriptstyle i$}}}\!(\mathbf{x})  
= \Lambda_{\sss 1}(x_{\scalebox{1.1}{$\scriptscriptstyle i$}}, u)  
= - \, x_{\scalebox{1.1}{$\scriptscriptstyle i$}}\begin{pmatrix}
\frac{(1 - u^{\scalebox{.8}{$\scriptscriptstyle 2$}}) x_{\scalebox{.8}{$\scriptscriptstyle i$}}}
{1- u^{\scalebox{.8}{$\scriptscriptstyle 2$}}x_{\scalebox{.8}{$\scriptscriptstyle i$}}^{\scalebox{.8}{$\scriptscriptstyle 2$}}} 
& 0 & - \frac{u (1 - x_{\scalebox{.8}{$\scriptscriptstyle i$}}^{\scalebox{.8}{$\scriptscriptstyle 2$}})}
{1 - u^{\scalebox{.8}{$\scriptscriptstyle 2$}}
x_{\scalebox{.8}{$\scriptscriptstyle i$}}^{\scalebox{.8}{$\scriptscriptstyle 2$}}} \\
0 & 1 & 0 \\
- \frac{u (1 - x_{\scalebox{.8}{$\scriptscriptstyle i$}}^{\scalebox{.8}{$\scriptscriptstyle 2$}})}
{1 - u^{\scalebox{.8}{$\scriptscriptstyle 2$}}
x_{\scalebox{.8}{$\scriptscriptstyle i$}}^{\scalebox{.8}{$\scriptscriptstyle 2$}}} & 0 
& \frac{(1 - u^{\scalebox{.8}{$\scriptscriptstyle 2$}}) x_{\scalebox{.8}{$\scriptscriptstyle i$}}}
{1- u^{\scalebox{.8}{$\scriptscriptstyle 2$}}x_{\scalebox{.8}{$\scriptscriptstyle i$}}^{\scalebox{.8}{$\scriptscriptstyle 2$}}}  \\
\end{pmatrix} 
\] 
for $i = 1, \ldots, 4,$ and 
\[
\Lambda_{\sigma_{\scalebox{0.75}{$\scriptscriptstyle 5$}}}\!(\mathbf{x})  
= \Lambda_{\sss 2}(x_{\scalebox{1.1}{$\scriptscriptstyle 5$}}, u)  
= - \, x_{\scalebox{1.1}{$\scriptscriptstyle 5$}}
\begin{pmatrix}
\frac{(1 - u^{\scalebox{.8}{$\scriptscriptstyle 2$}}) x_{\scalebox{.8}{$\scriptscriptstyle 5$}}}
{1- u^{\scalebox{.8}{$\scriptscriptstyle 2$}}x_{\scalebox{.8}{$\scriptscriptstyle 5$}}^{\scalebox{.8}{$\scriptscriptstyle 2$}}} 
& - \frac{u (1 - x_{\scalebox{.8}{$\scriptscriptstyle 5$}}^{\scalebox{.8}{$\scriptscriptstyle 2$}})}
{1 - u^{\scalebox{.8}{$\scriptscriptstyle 2$}}
x_{\scalebox{.8}{$\scriptscriptstyle 5$}}^{\scalebox{.8}{$\scriptscriptstyle 2$}}}  & 0 \\
\hskip3pt
- \frac{u (1 - x_{\scalebox{.8}{$\scriptscriptstyle 5$}}^{\scalebox{.8}{$\scriptscriptstyle 2$}})}
{1 - u^{\scalebox{.8}{$\scriptscriptstyle 2$}}
x_{\scalebox{.8}{$\scriptscriptstyle 5$}}^{\scalebox{.8}{$\scriptscriptstyle 2$}}}  
& \frac{(1 - u^{\scalebox{.8}{$\scriptscriptstyle 2$}}) x_{\scalebox{.8}{$\scriptscriptstyle 5$}}}
{1- u^{\scalebox{.8}{$\scriptscriptstyle 2$}}x_{\scalebox{.8}{$\scriptscriptstyle 5$}}^{\scalebox{.8}{$\scriptscriptstyle 2$}}} & 0 \\
0 & 0 & 1 \\
\end{pmatrix}.
\]
The group $O\subset \wW$ has order 4, and it is generated
by elements $\eta$ of order 2 with 
$\eta\a_{\sss i}=\a_{\sss j}$ and $\eta \a_{\sss k}=\a_{\sss l}$
for $\{i, j, k, l\}$ a permutation of 
$\{1, 2, 3, 4\}$. From~\eqref{6} we have that
$\Lambda_{\eta}$ is the identity matrix for $\eta\in O.$

We denote by~$\ZZ$ the vector $\ov{\ZW}$.
The functional equation $\ZW=\ZW\vert w$ becomes:
\be\label{7}
\ZZ(\xx;u)=\Lam_{w}(\xx) \ZZ(w\xx;u) 
\ee
for all $w\in \wW$, which follows from the cocycle relation of $\Lam_{w}$.

\subsection{An extension of the Chinta-Gunnells action}\label{s4.2}
Define the transformation $\tau$ acting on
$f\in\CC(\xx, u)$ by
$$
\tau f(\xx; u) = f\!\left(\xx; u\slash\xx^{\dd}\right).
$$
This transformation extends the natural action of
$w\in \wW$ given by 
$
wf(\xx; u)=f(w^{\sss -1}\xx; u)
$,
and we denote by $\wW\oplus \Z$ the group 
generated by $\wW$ and $\tau^{a}$, $a\in\Z$,
which can be seen as a subgroup of $\End(\CC(\xx,u))$.
We show that there is an extension of the
Chinta-Gunnells action to this larger group
$\wW\oplus \Z$, from which we will derive an extra
functional equation for the vector function $\ZZ$.
As in the previous subsection, we will\linebreak restrict the action to 
the space $\CC(\xx,u)_{\sss 0}$ of functions $f$ with $f^{\oo} \!=0.$

Similarly to the way the Chinta-Gunnells action is defined
in~\eqref{eq12}, we look for an action of the form 
\be\label{11}
f\vert \tau(\xx;u)=A_{\eee}(\xx;u)
\tau f^{\eee}(\xx;u)+A_{\eo}(\xx;u) \tau f^{\eo}(\xx;u) + 
A_{\oe}(\xx;u) \tau f^{\oe}(\xx;u)
\ee
defined for $f\in\CC(\xx,u)_{\sss 0}$, for three
unknown functions $A_{\eee}$, $A_{\eo}$, $A_{\oe}\in\CC(\xx,u)$.
In order for this action to preserve the space
$\CC(\xx,u)_{\sss 0}$, we require that
\[
%\be\label{100}
A_{\eee}^{\oo} \! =0,\quad A_{\eo}^{\oe} \! =0,\quad
A_{\oe}^{\eo} \! =0.\]
%\ee
We also require this action to be compatible with the
Chinta-Gunnells action, namely  
\be\label{12}
f\vert\tau\vert w=f\vert w\vert \tau\qquad
\text{(for $w\in \wW$)}.
\ee
%One can define the action of $\tau^{\sss -1}$ in a similar way, but  be easier to describe it in terms of 
%the associated extension of the cocycle $\Lam_w$.
Formula~\eqref{11} can be written: 
\be\label{10}
\ov{f\vert \tau}(\xx;u)=A(\xx;u) \bar{f}\!\left(\xx;u\slash \xx^\dd\right) 
\ee
for the $3\times 3$ matrix 
\be\label{13}
A(\xx;u)=
\begin{pmatrix}
  \ov{A_{\eee}}& \frac{1}{x_{\scalebox{.75}{$\scriptscriptstyle 5$}}}
  \ov{x_{\sss 5} A_{\eo}}& \frac{1}{x_{\scalebox{.75}{$\scriptscriptstyle 1$}}}
  \ov{x_{\sss 1} A_{\oe}} \\
\end{pmatrix}. 
\ee
Defining $\Lambda_{\tau}:=A$,
condition~\eqref{12} is then equivalent to the fact that the
cocycle~$\Lambda$ has a well-defined exten-\linebreak sion to the monoid 
$\wW\times \{ \tau^{n} \}_{n\ge 0} $ by 
the $1$-cocycle relation, namely it satisfies
$\Lambda_{\tau w}=\Lambda_{w\tau}$ for all
$w\in \wW,$ that is,
\be\label{14}
\Lam_{w}(\xx;u)A(w\xx;u)=A(\xx;u) \Lam_{w}\!\left(\xx;u\slash\xx^\dd\right).
\ee
We also require that $A$ is invertible,
so that $\Lambda_{\tau^{\sss -1}}$ is well-defined by the
cocycle relation. Therefore $f\vert \tau^{\sss -1}$ is\linebreak
well-defined as well, by the analogue of~\eqref{10} written
for $\tau^{\sss -1}$, with $A$ replaced by $\Lambda_{\tau^{\sss -1}}$.

The main theorem of this section shows that there exists an extension
of the Chinta-Gunnells action to\linebreak $\wW\oplus \Z$ as above, and
that~$\ZZ$ satisfies an additional functional equation under the
transformation~$\tau$.

\vskip5pt
\begin{thm}\label{T1} ---
  There exists an invertible matrix $A(\xx;u)$ of the form~\eqref{13}
  satisfying the following conditions:
  \begin{itemize}
 \item The cocycle relation~\eqref{14} is satisfied. 
 \item There exists an element $a\in \CC\!\left(u,\xx^\dd\right)$ such that
   the matrix $a\cdot A$ has polynomial entries in $\xx$ and $u$.
 \item The vector function $\ZZ(\xx;u)$ satisfies the functional
   equation
   \be\label{15}
\ZZ(\xx;u)=A(\xx;u) \ZZ\!\left(\xx; u\slash\xx^\dd\right).
\ee
\end{itemize}
\end{thm}

In \cite{DIPP}, we will use results in \cite{DIPP-fin} to show that a similar
result holds for arbitrary affine irreducible reduced root systems.

\begin{rem} The matrix $A$ in the theorem is essentially unique;
  more precisely, any invertible matrix $A$ of\linebreak the form~\eqref{13}
  that satisfies the cocycle relation is unique, up to multiplication by 
  elements in $\CC\!\left(u,\xx^\dd\right)$. We\linebreak omit the proof of this fact
  since it is quite technical, and we will not need it in the sequel.
\end{rem}

\begin{proof} The relations~\eqref{14}, for $w$ running through the generators $\ss_{\sss i}$ of $W$, reduce to a linear system of equations over the field $\Q(u)$,
with unknowns being the coefficients of the entries of $A$ (assuming that these entries are \emph{polynomials} of bounded degree in $\xx$). Using MAGMA~\cite{mag}, we solved this sparse system assuming that the degree of each entry is at most 16 in the variables $\xx$ (in which case there are about 100,000 unknowns and about four times as many equations).
We found an explicit non-singular matrix
  $A_{\sss 0}(\xx;u)$ of the form~\eqref{13} with polynomial
  entries in $\xx,u$ satisfying~\eqref{14}, and the interested
  reader can find it in~\cite{eDPP}. Thus we have the corresponding
extensions of the cocycle $\Lambda$, and of the
Chinta-Gunnells action, to all of $\wW\oplus\Z.$
By~\eqref{10} it follows that
\[
  A_{\sss 0}(\xx;u)\ZZ\!\left(\xx;u\slash \xx^\dd\right)
  =A_{\sss 0}(\xx;u) \cdot \sum_{w \, \in \, W}
  \ov{1 \vert w}\!\left(\xx; u\slash \xx^\dd\right)
  \; =\sum_{w\, \in \, W} \ov{1\vert \tau w}(\xx;u)
=\ov{Z_{\sss W, F}}(\xx;u)
\]
where $F = 1\vert \tau,$ and
$
Z_{\sss W, F}=\sum_{w\, \in \, W} F\vert w
$.
The function $F$ is the sum of the
  entries in the first column of~$A_{\sss 0}$, so it is a
  polynomial in $\xx$. By Proposition~\ref{p3.3}, it follows
that $Z_{\sss W, F}=a Z_{\sss W}$, with $a\in \CC\!\left(u,\xx^\dd\right)$,
and thus the matrix $A=A_{\sss 0}\slash a$ satisfies all
the conditions in the theorem. Given~$A_{\sss 0}$, the constant $a$
  can be determined explicitly using the algorithm in the
  proof of Proposition~\ref{p3.3}. More explicitly,
  the leading term among all the entries of $A_{\sss 0}$ is
  $\xx^{16\dd}$ (it occurs in the $(2, 2)$ entry), and the
  value of $a$ is given by
  \[
  a=u^{\sss 8} \xx^{\sss 5\dd} (\xx^{\sss \dd}-u^{\sss 2})^{\sss 3}
  (\xx^{\sss \dd}-u^{\sss 4})(\xx^{\sss 3\dd}-u^{\sss 4}). \qedhere
\]
\end{proof}
We shall also need the following information about $\Lam_{\tau^{\sss -1}}$.

\vskip5pt
\begin{thm} \label{Key-ingredient} --- The vector function 
$\mathbf{Z}(\mathbf{x}; u)$ satisfies the functional equation
\[
\mathbf{Z}(\mathbf{x}; u) = B(\mathbf{x}; u)
\mathbf{Z}\!\left(\mathbf{x}; u \mathbf{x}^{\delta}\right) 
\]
for a 3 by 3 matrix $B(\mathbf{x}; u)$ satisfying, in addition,
the following two conditions:
\begin{itemize}

\item The matrix
  \[
\prod_{\substack{\alpha  \, \in \,    
\Phi_{\scalebox{.95}{$\scriptscriptstyle \mathrm{re}$}}^{\scalebox{.95}{$\scriptscriptstyle +$}}
\\ \alpha \, < \, \delta}} 
\!\left(1 - u^{2}\mathbf{x}^{2\alpha} \right)
\cdot B(\mathbf{x}; u)
\]
has polynomial entries in $\mathbf{x}$ and $u.$ 
\item  Each entry of $B(\mathbf{x}; u)$ is divisible by
$
\left(1 - u^{2}\mathbf{x}^{\delta}\right)^{\scalebox{1.1}{$\scriptscriptstyle 2$}}\!.
$
\end{itemize}
\end{thm} 
\begin{proof}
  Take
  $
  B(\xx;u)=\Lambda_{\tau^{\sss -1}}(\xx;u)
  = A^{\sss -1}\!\left(\xx; u\xx^{\dd}\right)
  $,
  with $A$ from Theorem \ref{T1}. The functional equation
  fol-\linebreak lows at once from the previous theorem, and the
  other conditions follow from the explicit formula of~$A,$
  see~\cite{eDPP}. 
\end{proof}
We shall also need:

\vskip5pt
\begin{lem} \label{Z-when-u-is-zero} --- The specialization
  of the vector function $\mathbf{Z}$ to $u = 0$ is given by 
$
\mathbf{Z}(\mathbf{x}; 0) = \, \!^{t}(\Delta(\mathbf{x}), 0, 0).
$ 
\end{lem}

\begin{proof} By induction on length, for $w \in W,$ we have
  $
1 \vert w(\mathbf{x}; 0) 
= (- 1)^{\scalebox{1.1}{$\scriptscriptstyle \ell(w)$}}
\prod_{\beta \, \in \, \Phi(w)} \mathbf{x}^{2\beta},
$
where we recall that\linebreak 
$
\Phi(w) = \Phi^{+} \cap w^{\scalebox{1.1}{$\scriptscriptstyle -1$}}(\Phi^{-}).
$ 
Thus our assertion is just Macdonald's identity \cite{Mac} in type 
$
D_{\scalebox{1.1}{$\scriptscriptstyle 4$}}^{\scalebox{1.1}{$\scriptscriptstyle (1)$}}.$
\end{proof}

Note that, by combining the functional equations~\eqref{7}
and~\eqref{15}, we get
\be\label{200}
\ZZ(\xx; u) =\Lambda_{w\tau^{k}}(\xx;u)
\ZZ\!\left(w\xx; u \slash \xx^{k\dd}\right)
\ee
for all $w\in\wW$ and $k\in\Z$.

\section{Renormalization} \label{renorm}
The correction of the Chinta-Gunnells average for the affine root
system $D_{\scriptscriptstyle 4}^{\scriptscriptstyle (1)}$ is given
by
\begin{equation} \label{eq: average-zeta-normalized}
\begin{split}
  \tilde{Z}_{\scriptscriptstyle W}(\mathbf{x}; u) 
& = \prod_{n \, \ge \, 1}
\!\left(1 - u^{2} \mathbf{x}^{(2n -
    1)\delta}\right)^{- 2}
\cdot \, Z_{\scriptscriptstyle W}^{\scriptscriptstyle \mathrm{CG}}(\mathbf{x}; u)\\
& = \frac{1}{\Delta(\mathbf{x})\prod_{n \, \ge \, 1}
\!\left(1 - u^{2} \mathbf{x}^{(2n - 1)\delta}\right)^{\!\scalebox{1.1}{$\scriptscriptstyle 2$}}}
\, \cdot \sum_{w \, \in \, W} 1 \vert w(\mathbf{x}; u).
\end{split}
\end{equation}
As we shall see in Section \ref{comparison}, this function is directly
connected to a {W}eyl group multiple {D}irichlet series associated with
the $4$-th moment of quadratic {D}irichlet $L$-functions.

Letting $D(\mathbf{x}; u),$ the denominator of
$\tilde{Z}_{\scriptscriptstyle W}(\mathbf{x}; u),$ be defined by
\begin{equation} \label{eq: denominator}
D(\mathbf{x}; u) 
: = \prod_{\alpha  \in   
\Phi_{\scalebox{.95}{$\scriptscriptstyle \mathrm{re}$}}^{\scalebox{.95}{$\scriptscriptstyle +$}}} 
\!\left(1 - u^{2}\mathbf{x}^{2\alpha} \right)
\end{equation}
we can now show the following:

\vskip5pt
\begin{thm} \label{Divisibility-and-analytic-continuation} --- 
%The function $Z_{\scriptscriptstyle W}(\mathbf{x}; u)$ is divisible by 
%$
%\Delta(\mathbf{x})\prod_{n \, \ge \, 1}
%\left(1 - u^{2}\mathbf{x}^{(2n - 1)\delta}\right)^{\scalebox{1.1}{$\scriptscriptstyle 2$}}\!.
%$ 
The function $D\tilde{Z}_{\scriptscriptstyle W}(\mathbf{x}; u)$ is holomorphic in the region 
$|\mathbf{x}^{\scalebox{1.1}{$\scriptscriptstyle \delta$}}| < 1.$ 
\end{thm}

\begin{proof} For $|\mathbf{x}^{\scalebox{1.1}{$\scriptscriptstyle \delta$}}| < 1,$ the absolutely convergent product 
$
\Delta^{\scalebox{1.2}{$\scriptscriptstyle \mathrm{im}$}}(\mathbf{x}) : = 
\prod_{n \, \ge \, 1}
\!\left(1 - \mathbf{x}^{2 n \delta}\right)^{\scalebox{1.1}{$\scriptscriptstyle 4$}}
$ 
is non-vanishing, and the divisibility (in the obvious sense) of $Z_{\scriptscriptstyle W}$ by 
$\Delta^{\scalebox{1.2}{$\scriptscriptstyle \mathrm{re}$}} 
=  \Delta \slash  
\Delta^{\scalebox{1.2}{$\scriptscriptstyle \mathrm{im}$}}
$ 
in this region has already been discussed in \cite[Section~4]{BD}.

On the other hand, by \eqref{eq: 1-acted-by-w} and \eqref{eq: function-J}, the function 
$Z_{\scriptscriptstyle W}(\mathbf{x}; u)$ is a sum of rational functions whose denominators are 
products of {\normalfont\itshape distinct} factors of the form $1 - u^{2}\mathbf{x}^{2\alpha}$ 
($
\alpha \in \Phi_{\scalebox{1.2}{$\scriptscriptstyle \mathrm{re}$}}^{\scalebox{1.2}{$\scriptscriptstyle +$}}
$). 
It follows that $DZ_{\scriptscriptstyle W}\slash \Delta$ is holomorphic\linebreak when 
$|\mathbf{x}^{\scalebox{1.1}{$\scriptscriptstyle \delta$}}| < 1,$ and by Theorem \ref{Key-ingredient}, that 
$(D\slash \Delta)\mathbf{Z}(\mathbf{x}; u)$ is divisible by 
$\left(1 - u^{2}\mathbf{x}^{\delta}\right)^{\scalebox{1.1}{$\scriptscriptstyle 2$}}\!.$ Accordingly, the vector function 
$(D\slash \Delta)\mathbf{Z}\!\left(\mathbf{x}; u \mathbf{x}^{\delta}\right)$ is divisible by 
$\left(1 - u^{2}\mathbf{x}^{3\delta}\right)^{\scalebox{1.1}{$\scriptscriptstyle 2$}}\!,$ and by the functional equation 
\[
\left(\frac{D}{\Delta}\mathbf{Z}\right)\!(\mathbf{x}; u) \; = \prod_{\substack{\alpha  \, \in \,    
\Phi_{\scalebox{.95}{$\scriptscriptstyle \mathrm{re}$}}^{\scalebox{.95}{$\scriptscriptstyle +$}}
\\ \alpha \, < \, \delta}} 
\!\left(1 - u^{2}\mathbf{x}^{2\alpha} \right)
\cdot B(\mathbf{x}; u)
\left(\frac{D}{\Delta}\mathbf{Z}\right)\!\left(\mathbf{x}; u \mathbf{x}^{\delta}\right) 
\] 
so does $(D\slash \Delta)\mathbf{Z}(\mathbf{x}; u).$ Proceeding by induction on $n,$ 
we see at once that $(D\slash \Delta)\mathbf{Z}(\mathbf{x}; u)$ is divisible by\linebreak 
$
\left(1 - u^{2}\mathbf{x}^{(2n - 1)\delta}\right)^{\scalebox{1.1}{$\scriptscriptstyle 2$}}
$ 
for all $n \ge 1.$ Thus $DZ_{\scriptscriptstyle W}\slash \Delta$ is also divisible by the product 
$
\prod_{n \, \ge \, 1}
\left(1 - u^{2}\mathbf{x}^{(2n - 1)\delta}\right)^{\scalebox{1.1}{$\scriptscriptstyle 2$}}\!,
$ 
which completes the proof.
\end{proof}

Since
$
Z_{\scriptscriptstyle W}
$
is $W$\!--invariant under the Chinta-Gunnells action, it follows from
\eqref{eq: trans-form-McD-denominator} and
the $W$\!--invariance of\linebreak the
product over imaginary roots that $\tilde{Z}_{\scriptscriptstyle W}$
itself satisfies a functional equation with respect to each
$w \in W$. More precisely, we have:
\be\label{fun-eq-tZ}
\tZW=\tZW\| w 
\ee
where we write $\|$ for the action defined on generators by
\[
f\|\ss_{\sss i}(\xx):=-\frac{1}{x_{\sss i}^{\sss 2}}
  f\vert \ss_{\sss i}(\xx)
  =\frac{1}{x_{\sss i}} f_{\sss i}^{-}(\ss_{\sss i}\xx) +
\frac{1-u\slash x_{\sss i}}{1-u x_{\sss i}} f_{\sss i}^{+}(\ss_{\sss i}\xx). 
\]
It is clear that the subspace $\CC(\xx,u)_{\sss 0}$ defined in
Lemma~\ref{L4.1} is invariant under this action.
The functional equations satisfied by the vector $\tZZ:=\ov{\tZW}$ are:
\be\label{eq: f-e-loc}
\tZZ(\xx;u)=\tL_{w}(\xx) \tZZ(w\xx;u) 
\ee
where $\tL_{w}(\xx)=\tL_{w}(\xx;u)$ is the 3 by 3 matrix-cocycle such that 
\be\label{eq:cocycle}
\overline{f \| w}(\mathbf{x};u) =
\tL_{w}(\mathbf{x})\bar{f}(w\mathbf{x};u) 
\ee
for $w \in W$ and $f\in\CC(\xx,u)_{\sss 0}$.
On generators, $\tL_{w}(\xx)$ is given by
$
\tL_{\ss_{\scalebox{.75}{$\scriptscriptstyle i$}}}\!(\xx)
=-\frac{1}{x_{\sss i}^{\sss 2}} \Lam_{\ss_{\scalebox{.75}{$\scriptscriptstyle i$}}}\!(\xx)
$.

The following lemma provides some structural properties of the
function $\tilde{Z}_{\scriptscriptstyle W}(\mathbf{x}; u)$ that will
be used to get some analytic information about the {W}eyl 
group multiple {D}irichlet series we will introduce in 
Section \ref{MDS}.\linebreak 
\begin{lem} \label{initial-properties} --- Set 
$
\underline{x} = (x_{\scalebox{1.1}{$\scriptscriptstyle 1$}}, \ldots, x_{\scalebox{1.1}{$\scriptscriptstyle 4$}}),
$ 
$
\underline{k} = (k_{\scalebox{1.1}{$\scriptscriptstyle 1$}}, \ldots, k_{\scalebox{1.1}{$\scriptscriptstyle 4$}})
$ 
and $l = k_{\scalebox{1.1}{$\scriptscriptstyle 5$}}.$ Then we have:  
\begin{enumerate} 
	
	\item The function $\tilde{Z}_{\scriptscriptstyle W}(\mathbf{x}; u)$ can be written as 
	\begin{equation*}
	\begin{split}
	\tilde{Z}_{\scriptscriptstyle W}(\mathbf{x}; u) \, & = \,  
	\frac{\sum_{l - \mathrm{even}} 
	P_{\scalebox{.95}{$\scriptscriptstyle l$}}(\underline{x}; u)
	x_{\scalebox{1.1}{$\scriptscriptstyle 5$}}^{\scalebox{1.1}{$\scriptscriptstyle l$}}}
	{\prod_{j = 1}^{4} (1 - ux_{\!\scalebox{1.1}{$\scriptscriptstyle j$}})} \; + 
	\sum_{l - \mathrm{odd}} P_{\scalebox{.95}{$\scriptscriptstyle l$}}(\underline{x}; u)
	x_{\scalebox{1.1}{$\scriptscriptstyle 5$}}^{\scalebox{1.1}{$\scriptscriptstyle l$}}\\
	& = \, \frac{\sum_{|\underline{k}| - \mathrm{even}} 
	Q_{\scalebox{1.1}{$\scriptscriptstyle \underline{k}$}}(x_{\scalebox{1.1}{$\scriptscriptstyle 5$}}; u)
	\underline{x}^{\scalebox{1.1}{$\scriptscriptstyle \underline{k}$}}}
	{1 - u x_{\scalebox{1.1}{$\scriptscriptstyle 5$}}}  \;\;  + 
	\sum_{|\underline{k}| - \mathrm{odd}} 
	Q_{\scalebox{1.1}{$\scriptscriptstyle \underline{k}$}}(x_{\scalebox{1.1}{$\scriptscriptstyle 5$}}; u)
	\underline{x}^{\scalebox{1.1}{$\scriptscriptstyle \underline{k}$}}
	\end{split}
	\end{equation*} 
	where $P_{\scalebox{.95}{$\scriptscriptstyle l$}}(\underline{x}; u)$ and 
	$Q_{\scalebox{1.1}{$\scriptscriptstyle \underline{k}$}}(x_{\scalebox{1.1}{$\scriptscriptstyle 5$}}; u)$ are 
	polynomials in $x_{\scalebox{1.1}{$\scriptscriptstyle 1$}}, \ldots, x_{\scalebox{1.1}{$\scriptscriptstyle 4$}}, u$ 
	and $x_{\scalebox{1.1}{$\scriptscriptstyle 5$}}, u,$ respectively. 
	Here we set $|\underline{k}| = $\linebreak 
	$k_{\scalebox{1.1}{$\scriptscriptstyle 1$}} + \cdots + k_{\scalebox{1.1}{$\scriptscriptstyle 4$}}.$

	\item The power series obtained by expanding 
	\[
	\sum_{l \, \ge \, 0} P_{\scalebox{.95}{$\scriptscriptstyle l$}}(\underline{x}; u)
	x_{\scalebox{1.1}{$\scriptscriptstyle 5$}}^{\scalebox{1.1}{$\scriptscriptstyle l$}}
	\] 
	is absolutely convergent for arbitrary 
	$\underline{x}\in \mathbb{C}^{\scriptscriptstyle 4},$ provided 
	$|x_{\scalebox{1.1}{$\scriptscriptstyle 5$}}|$ is sufficiently small, and the power series obtained by expanding
	\[
	\sum_{|\underline{k}| \, \ge \, 0} 
	Q_{\scalebox{1.1}{$\scriptscriptstyle \underline{k}$}}(x_{\scalebox{1.1}{$\scriptscriptstyle 5$}}; u)
	\underline{x}^{\scalebox{1.1}{$\scriptscriptstyle \underline{k}$}}
	\] 
	is absolutely convergent for any $x_{\scalebox{1.1}{$\scriptscriptstyle 5$}} \in \mathbb{C},$ provided all 
	$|x_{\scalebox{1.1}{$\scriptscriptstyle 1$}}|, \ldots, |x_{\scalebox{1.1}{$\scriptscriptstyle 4$}}|$ are sufficiently small.

	\item We have 
	\[
	P_{\scriptscriptstyle 0}(\underline{x}; u) 
	= Q_{\scriptscriptstyle \underline{0}}(x_{\scalebox{1.1}{$\scriptscriptstyle 5$}}; u) \equiv 1
	\] 
	where $\underline{0} = (0, \ldots, 0).
	$

	\item The polynomials $P_{\scalebox{.95}{$\scriptscriptstyle l$}}(\underline{x}; u)$ are symmetric in 
	$\underline{x},$ and if $l$ is odd then $P_{\scalebox{.95}{$\scriptscriptstyle l$}}(\underline{x}; u)$ is even, i.e., 
	$
	P_{\scalebox{.95}{$\scriptscriptstyle l$}}(\underline{x}; u) = P_{\scalebox{.95}{$\scriptscriptstyle l$}}(-\underline{x}; u).
	$

	\item We have the functional equations 
	\begin{equation} \label{eq: poly-P-Q-func-eq}
	P_{\scalebox{.95}{$\scriptscriptstyle l$}}(x_{\scalebox{1.1}{$\scriptscriptstyle 1$}}, 
	x_{\scalebox{1.1}{$\scriptscriptstyle 2$}}, 
	x_{\scalebox{1.1}{$\scriptscriptstyle 3$}}, x_{\scalebox{1.1}{$\scriptscriptstyle 4$}}; u) 
	= x_{\scalebox{1.1}{$\scriptscriptstyle 1$}}^{\scriptscriptstyle l - \delta_{\scalebox{.75}{$\scriptscriptstyle l$}}} 
	P_{\scalebox{.95}{$\scriptscriptstyle l$}}\bigg(\frac{1}{x_{\scalebox{1.1}{$\scriptscriptstyle 1$}}}, 
	x_{\scalebox{1.1}{$\scriptscriptstyle 2$}}, 
	x_{\scalebox{1.1}{$\scriptscriptstyle 3$}},
	x_{\scalebox{1.1}{$\scriptscriptstyle 4$}}; u\bigg)
	\;\;\;\; \mathrm{and} \;\;\;\; 
	Q_{\scalebox{1.1}{$\scriptscriptstyle \underline{k}$}}(x_{\scalebox{1.1}{$\scriptscriptstyle 5$}}; u)
	= x_{\scalebox{1.1}{$\scriptscriptstyle 5$}}^{\scriptscriptstyle |\underline{k}| - \delta_{\scalebox{.75}
	{$\scriptscriptstyle |\underline{k}|$}}}
	Q_{\scalebox{1.1}{$\scriptscriptstyle \underline{k}$}}\bigg(\frac{1}{x_{\scalebox{1.1}{$\scriptscriptstyle 5$}}}; u\bigg)
	\end{equation} 
	with $\delta_{\scalebox{1.2}{$\scriptscriptstyle n$}} \! = 0$ or $1$ according as $n$ is even or odd. 
	
	\end{enumerate} 
      \end{lem}

      \begin{proof} For notational simplicity, we denote
        $\tZ(\xx):=\tZW(\xx;u)$ in this proof (the variable $u$ being
        fixed). The functional equation of $\tZ(\xx)$ can be
        broken into its even and odd parts, according to $\e_{\sss i}$, as:
        \[
\tZ^+_{\sss i}\!(\xx)=\frac{1-u\slash x_{\sss i}}{1-ux_{\sss i}}\tZ^+_{\sss i}\!(\ss_{\sss i}\xx), 
\quad \tZ^-_{\sss i}\!(\xx)=\frac{1}{x_{\sss i}} \tZ^-_{\sss i}\!(\ss_{\sss i}\xx)
\]
for $i=1,\ldots, 5$. Note that $\tZ^\pm_{\sss i}=\tZ^\pm_{\sss 1}$ for
$i=1,\ldots, 4$. The even functional equations can 
be expressed in terms of the functions
\[
  G_{\sss 1}(\xx)=\prod_{i=1}^{4} (1-ux_{\sss i})
  \tZ^+_{\sss 1}\!(\xx), \quad
  G_{\sss 5}(\xx)=(1-ux_{\sss 5}) 
\tZ^+_{\sss 5}\!(\xx) 
\]
as $G_{\sss 1}(\xx)=G_{\sss 1}(\ss_{\sss i}\xx)$ for $i=1, \ldots, 4$,
and $G_{\sss 5}(\xx)=G_{\sss 5}(\ss_{\sss 5}\xx)$.

By Theorem \ref{Divisibility-and-analytic-continuation},
the function $D\tZ$ is holomorphic in 
$
\Omega : =\{\mathbf{x} \in \mathbb{C}^{5} : | \mathbf{x}^{\scalebox{1.1}{$\scriptscriptstyle \delta$}}| < 1\},
$ 
and in this domain, it satisfies the functional equations 
\[
D\tZ^+_{\sss i}\!(\xx) = 
\frac{x_{\scalebox{1.2}{$\scriptscriptstyle i$}}
(1 + u x_{\scalebox{1.2}{$\scriptscriptstyle i$}})}
{u +  x_{\scalebox{1.2}{$\scriptscriptstyle i$}}} 
D\tZ^+_{\sss i}\!(\ss_{\sss i}\xx), \quad 
D\tilde{Z}^-_{\sss i}\!(\mathbf{x}) = - \, 
\frac{x_{\scalebox{1.2}{$\scriptscriptstyle i$}}
(1 - u^{\scalebox{1.2}{$\scriptscriptstyle 2$}}
x_{\scalebox{1.2}{$\scriptscriptstyle i$}}^{\scalebox{1.2}{$\scriptscriptstyle 2$}})}
{u^{\scalebox{1.2}{$\scriptscriptstyle 2$}} 
- x_{\scalebox{1.2}{$\scriptscriptstyle i$}}^{\scalebox{1.2}{$\scriptscriptstyle 2$}}}
D\tilde{Z}^-_{\sss i}\!(\ss_{\sss i}\xx )
\]
for $i=1,\ldots, 5$, which imply that
$
D\tilde{Z}^+_{\sss i}\!(\mathbf{x})
$ 
and 
$
D\tilde{Z}^-_{\sss i}\!(\mathbf{x})
$ 
are divisible by $1 + u x_{\scalebox{1.2}{$\scriptscriptstyle i$}}$ and 
$
1 - u^{\scalebox{1.2}{$\scriptscriptstyle 2$}}
x_{\scalebox{1.2}{$\scriptscriptstyle i$}}^{\scalebox{1.2}{$\scriptscriptstyle 2$}},
$
respectively. Thus the functions
\[ 
F_{\scalebox{1.1}{$\scriptscriptstyle 1$}}^+\!(\mathbf{x})\,  : = 
\, \prod_{\substack{\alpha \in  
\Phi_{\scalebox{.95}{$\scriptscriptstyle \mathrm{re}$}}^{\scalebox{.95}{$\scriptscriptstyle +$}} \\ 
\alpha \, \ne \, \alpha_{\scalebox{.7}{$\scriptscriptstyle 1$}}, \ldots,\alpha_{\scalebox{0.7}{$\scriptscriptstyle 4$}}}} 
\!\left(1 - u^{2}\mathbf{x}^{2\alpha} \right)
\cdot G_{\sss 1}(\mathbf{x}), 
\quad  
F_{\scalebox{1.1}{$\scriptscriptstyle 1$}}^-\!(\mathbf{x})\,  : = \,
\prod_{\substack{\alpha  \in   
\Phi_{\scalebox{.95}{$\scriptscriptstyle \mathrm{re}$}}^{\scalebox{.95}{$\scriptscriptstyle +$}} \\ 
\alpha \, \ne \, \alpha_{\scalebox{.7}{$\scriptscriptstyle 1$}}, \ldots, \alpha_{\scalebox{0.7}{$\scriptscriptstyle 4$}}}} 
\!\left(1 - u^{2}\mathbf{x}^{2\alpha} \right)
\cdot \tilde{Z}^-_{\sss 1}\!(\mathbf{x})
\]
and
\[
H_{\scalebox{1.1}{$\scriptscriptstyle 5$}}^+\!(\mathbf{x})\,  : = 
\, \prod_{\substack{\alpha  \in   
\Phi_{\scalebox{.95}{$\scriptscriptstyle \mathrm{re}$}}^{\scalebox{.95}{$\scriptscriptstyle +$}} \\ 
\alpha \, \ne \, \alpha_{\scalebox{.7}{$\scriptscriptstyle 5$}}}} 
\!\left(1 - u^{2}\mathbf{x}^{2\alpha} \right)
\cdot G_{\sss 5}(\mathbf{x}), 
\quad
H_{\scalebox{1.1}{$\scriptscriptstyle 5$}}^-\!(\mathbf{x})\,  : = 
\, \prod_{\substack{\alpha  \in   
\Phi_{\scalebox{.95}{$\scriptscriptstyle \mathrm{re}$}}^{\scalebox{.95}{$\scriptscriptstyle +$}} \\ 
\alpha \, \ne \, \alpha_{\scalebox{.7}{$\scriptscriptstyle 5$}}}} 
\!\left(1 - u^{2}\mathbf{x}^{2\alpha} \right)
\cdot \tilde{Z}^-_{\sss 5}\!(\mathbf{x}) 
\] 
are still holomorphic in $\Omega.$ Notice that, for $i = 1, \ldots, 4,$ 
$F_{\scalebox{1.1}{$\scriptscriptstyle 1$}}^+\!(\mathbf{x})$ is
$\sigma_{\scriptscriptstyle i}$-invariant,
$
F_{\scalebox{1.1}{$\scriptscriptstyle 1$}}^-\!(\mathbf{x})
= x_{\scalebox{1.1}{$\scriptscriptstyle i$}}^{\scalebox{1.1}{$\scriptscriptstyle -1$}}
F_{\scalebox{1.1}{$\scriptscriptstyle 1$}}^-\!(\sigma_{\scriptscriptstyle i}\mathbf{x}), 
$
and since $\sigma_{\scriptscriptstyle i}$ is just permuting the roots in 
$
\Phi_{\scalebox{1.2}{$\scriptscriptstyle \mathrm{re}$}}^{\scalebox{1.2}{$\scriptscriptstyle +$}}
\setminus
\{\alpha_{\scalebox{1.1}{$\scriptscriptstyle 1$}}, \ldots, \alpha_{\scalebox{1.1}{$\scriptscriptstyle 4$}}\}, 
$
the product in the definition of
$F_{\scalebox{1.1}{$\scriptscriptstyle 1$}}^\pm$ is also $\sigma_{\scriptscriptstyle i}$-invariant. By expanding the inverse of this product, and  
$F_{\scalebox{1.1}{$\scriptscriptstyle 1$}}^\pm\!(\mathbf{x})$ in
power series, we can write
\[   
G_{\sss 1}(\xx)
\; = \sum_{l - \mathrm{even}} P_{\scalebox{.95}{$\scriptscriptstyle l$}}(\underline{x}; u)
x_{\scalebox{1.1}{$\scriptscriptstyle 5$}}^{\scalebox{1.1}{$\scriptscriptstyle l$}}, \quad
\tilde{Z}_{\scriptscriptstyle 1}^{-}\!(\mathbf{x})
\; = \sum_{l - \mathrm{odd}} P_{\scalebox{.95}{$\scriptscriptstyle l$}}(\underline{x}; u)
x_{\scalebox{1.1}{$\scriptscriptstyle 5$}}^{\scalebox{1.1}{$\scriptscriptstyle l$}};
\]
these expansions hold as long as $|\mathbf{x}^{\alpha}| < |u|^{\scalebox{1.1}{$\scriptscriptstyle -1$}}$ for all 
$
\alpha \in \Phi_{\scalebox{1.2}{$\scriptscriptstyle \mathrm{re}$}}^{\scalebox{1.2}{$\scriptscriptstyle +$}}
\setminus
\{\alpha_{\scalebox{1.1}{$\scriptscriptstyle 1$}}, \ldots, \alpha_{\scalebox{1.1}{$\scriptscriptstyle 4$}}\} 
$ 
(e.g., $\underline{x} \in \mathbb{C}^{\scriptscriptstyle 4}$ is arbitrary, 
and $|x_{\scalebox{1.1}{$\scriptscriptstyle 5$}}|$ is sufficiently small), 
which justifies the first part (i.e., the $P$-part) of {\normalfont\itshape 2}. The functional 
equation~\eqref{eq: poly-P-Q-func-eq} of the coefficients $P_{\scalebox{.95}{$\scriptscriptstyle l$}}(\underline{x}; u)$ now follows from the $\sigma_{\scriptscriptstyle 1}$-invariance 
of~$G_{\sss 1}$ and the functional equation 
$
\tilde{Z}_{\sss 1}^{-}\!(\mathbf{x}) 
= x_{\scalebox{1.1}{$\scriptscriptstyle 1$}}^{\scalebox{1.1}{$\scriptscriptstyle -1$}}
\tilde{Z}_{\sss 1}^{-}\!(\sigma_{\scriptscriptstyle 1}\mathbf{x}).
$ 
It follows that 
$
P_{\scalebox{.95}{$\scriptscriptstyle l$}}(\underline{x}; u)
$ 
are polynomials in
$
x_{\scalebox{1.1}{$\scriptscriptstyle 1$}}, \ldots,
x_{\scalebox{1.1}{$\scriptscriptstyle 4$}}, u,
$
and our assertions {\normalfont\itshape 4} follow from the fact that
$\tZ_{\sss 1}^-\!(\xx)=\tZ_{\sss 1}^-\!(\e_{\sss 5}\xx)$ (as $\tZ^{\oo}=0$), 
and from the symmetry of 
$
\tilde{Z}_{\sss 1}^{+}\!(\mathbf{x})
$ 
and 
$
\tilde{Z}_{\sss 1}^{-}\!(\mathbf{x})
$ 
with respect to\linebreak
$
x_{\scalebox{1.1}{$\scriptscriptstyle 1$}}, \ldots, x_{\scalebox{1.1}{$\scriptscriptstyle 4$}}.
$ 

The $Q$-parts of {\normalfont\itshape1, 2} and {\normalfont\itshape 5} follow from the same argument, applied 
to $H_{\sss 5}^+$ and $H_{\sss 5}^-$.

Finally, notice that 
\[
\tilde{Z}(\underline{x}, 0) =    
\Delta(\underline{x}, 0)^{\scalebox{1.1}{$\scriptscriptstyle - 1$}}
\, \cdot \sum_{w \, \in \, \langle \sigma_{\scalebox{.75}{$\scriptscriptstyle i$}}
\rangle_{\scalebox{.75}{$\scriptscriptstyle 1\le i \le 4$}}} 1 \vert w(\underline{x}, 0) 
= \prod_{j = 1}^{4} (1 - ux_{\!\scalebox{1.1}{$\scriptscriptstyle j$}})^{\scalebox{1.1}{$\scriptscriptstyle - 1$}}
\] 
and 
\[
\tilde{Z}(\underline{0}, x_{\scalebox{1.1}{$\scriptscriptstyle 5$}} ) =    
\Delta(\underline{0}, x_{\scalebox{1.1}{$\scriptscriptstyle 5$}})^{\scalebox{1.1}{$\scriptscriptstyle - 1$}} 
(1 + 1 \vert \sigma_{\scriptscriptstyle 5}(x_{\scalebox{1.1}{$\scriptscriptstyle 5$}} ))
= (1 - ux_{\scalebox{1.1}{$\scriptscriptstyle 5$}})^{\scalebox{1.1}{$\scriptscriptstyle - 1$}}
\] 
which give {\normalfont\itshape 3}. This completes the proof. 
\end{proof}

The extra functional equation satisfied by $\tZZ(\xx;u)$ is 
\be\label{300}
\tZZ(\xx;u)=\tB(\xx;u) \tZZ\!\left(\xx;u\xx^\dd\right) 
\ee
where $\tB(\xx;u)=B(\xx;u)\slash \big(1-u^{\sss 2}x^\dd\big)^{\sss 2},$
with the matrix $B(\xx;u)$ from Theorem~\ref{Key-ingredient}. Using
this functional equation, we now show that $\tZZ(\xx;u)$ is
completely determined by $\tZZ(\xx;0)$ and $\tB(\xx;u).$
By contrast, the\linebreak functional equations~\eqref{eq:
  f-e-loc} alone determine $\tZZ(\xx;u)$ only up to a power
series in $u$ and $\xx^\dd$ (see \cite[Theo-\linebreak rem~3.7]{BD}).
\vskip5pt
\begin{lem}\label{L5.3} --- 
  Expand the vector function $\tZZ(\mathbf{x}; u)$ as
  \[
\tZZ(\mathbf{x}; u)
= \tZZ_{\scalebox{1.1}{$\scriptscriptstyle 0$}}(\mathbf{x}) 
+ u\tZZ_{\scalebox{1.1}{$\scriptscriptstyle 1$}}(\mathbf{x}) 
+ u^{\scalebox{1.1}{$\scriptscriptstyle 2$}}\tZZ_{\scalebox{1.1}{$\scriptscriptstyle 2$}}(\mathbf{x}) + \cdots 
\] 
and let 
$
\tB(\mathbf{x}; u) = 
\tB_{\scalebox{1.1}{$\scriptscriptstyle 0$}}(\mathbf{x}) 
+ u\tB_{\scalebox{1.1}{$\scriptscriptstyle 1$}}(\mathbf{x}) + \cdots
$. Then we have
\[
\tZZ_{\scalebox{1.1}{$\scriptscriptstyle 0$}}(\mathbf{x}) 
=\begin{pmatrix} 
1\\
0\\
0\\
\end{pmatrix},\quad
\tB_{\scalebox{1.1}{$\scriptscriptstyle 0$}}(\xx)
=\tB_{\scalebox{1.1}{$\scriptscriptstyle 0$}} \!:= 
\begin{pmatrix} 
1 & 0 & 0\\
0 & 0 & 0\\
0 & 0 & 0\\
\end{pmatrix}
\]
and the matrix $\tB(\mathbf{x}; u)$ determines $\tZZ(\mathbf{x}; u)$
recursively by
\be\label{eq: u-exp}
\tZZ_{n}(\xx) = (I-\xx^{n\dd} \tB_{\scalebox{1.1}{$\scriptscriptstyle 0$}})^{-1}
\cdot 
\sum_{i=0}^{n-1} \xx^{i\dd}\tB_{n-i}(\xx) \tZZ_{i}(\xx) 
\ee
for $n\ge 1$, where $I$ is the identity matrix. 
\end{lem}

\begin{proof} The expression for
  $
  \tZZ_{\scalebox{1.1}{$\scriptscriptstyle 0$}}(\mathbf{x})
  $
  follows at once from Lemma~\ref{Z-when-u-is-zero}, and the
  explicit formula of the matrix $B(\xx;u)$ yields
  $
  \tB_{\scalebox{1.1}{$\scriptscriptstyle 0$}}.
  $
  The recursion of $\tZZ_{n}(\xx)$ follows at once from~\eqref{300}.   
\end{proof}
The following corollary will be needed in the proof of the main
theorem in Section~\ref{comparison}. 
\vskip5pt
\begin{cor}\label{C5.4} --- We have that
  \[
    \tilde{Z}_{\scriptscriptstyle W}(\mathbf{x}; u)
    \!\!\!\pmod{u^{\scalebox{1.1}{$\scriptscriptstyle 2$}}}
= 1+ u(x_{\scalebox{1.1}{$\scriptscriptstyle 1$}} 
+ x_{\scalebox{1.1}{$\scriptscriptstyle 2$}} 
+ x_{\scalebox{1.1}{$\scriptscriptstyle 3$}} 
+ x_{\scalebox{1.1}{$\scriptscriptstyle 4$}} 
+ x_{\scalebox{1.1}{$\scriptscriptstyle 5$}}).
\] 
\end{cor}
\begin{proof}
  From the explicit formula of the matrix $\tB(\mathbf{x}; u)$, one
  finds that
  \[
    \tB_{\scalebox{1.1}{$\scriptscriptstyle 1$}}(\mathbf{x}) = 
\begin{pmatrix} 
0 & \mathbf{x}^{\delta}x_{\scalebox{1.1}{$\scriptscriptstyle 5$}}^{\scalebox{1.1}{$\scriptscriptstyle - 1$}} & 
\mathbf{x}^{\delta}
\!\left(x_{\scalebox{1.1}{$\scriptscriptstyle 1$}}^{\scalebox{1.1}{$\scriptscriptstyle - 1$}} 
+ x_{\scalebox{1.1}{$\scriptscriptstyle 2$}}^{\scalebox{1.1}{$\scriptscriptstyle - 1$}} 
+ x_{\scalebox{1.1}{$\scriptscriptstyle 3$}}^{\scalebox{1.1}{$\scriptscriptstyle - 1$}}  
+ x_{\scalebox{1.1}{$\scriptscriptstyle 4$}}^{\scalebox{1.1}{$\scriptscriptstyle - 1$}}\right)\\
x_{\scalebox{1.1}{$\scriptscriptstyle 5$}}^{} & 0 & 0\\
x_{\scalebox{1.1}{$\scriptscriptstyle 1$}}^{}  
+ x_{\scalebox{1.1}{$\scriptscriptstyle 2$}}^{}  
+ x_{\scalebox{1.1}{$\scriptscriptstyle 3$}}^{} 
+ x_{\scalebox{1.1}{$\scriptscriptstyle 4$}}^{} & 0 & 0\\
\end{pmatrix}.
\]
The formula now follows from~\eqref{eq: u-exp} with $n=1$,
after multiplying by the row vector $(1, 1, 1).$
\end{proof}

We shall also need a positivity result about the specializations
$P_{\scalebox{.95}{$\scriptscriptstyle l$}}(\u1;u)$ of the polynomials
in Lemma~\ref{initial-properties}, with $\u1=(1,1,1,1)$.
\vskip5pt
\begin{cor}\label{C5.5} --- The polynomial
  $P_{\scalebox{.95}{$\scriptscriptstyle l$}} (\u1; u)$ for $l$ odd,
  and the power series
  $
  P_{\scalebox{.95}{$\scriptscriptstyle l$}}(\u1; u )\slash (1-u)^{4}
  $
  for $l$ even have non-negative coefficients. 
\end{cor}
\begin{proof} By Proposition~\ref{pa.2}, the entries of the
  matrix $\tB(\u1,x_{\sss 5}; u)$, as power series in
  $x_{\sss 5}$ and $u$, have non-\linebreak negative
  coefficients. \!Thus the matrices $\tB_{n}(\u1, x_{\sss 5})$
  have entries with non-negative coefficients, and by
  in-\linebreak duction using~\eqref{eq:
    u-exp}, the same is true for the vectors $\tZZ_{n}(\u1, x_{\sss 5})$.
  \!Accordingly, the function
  $
  \tZW(\underline{1},x_{\sss 5};u)
  $
  has non-negative coefficients when expanded as a power series,
  and now Lemma~\ref{initial-properties} finishes the proof. 
\end{proof}

\section{Residues} %\label{res}
By Theorem \ref{Divisibility-and-analytic-continuation}, the
singularities of the function $\tilde{Z}_{\scriptscriptstyle W}(\mathbf{x}; u)$
can only occur at the zeros of the denominator $D(\mathbf{x}; u),$ and
thus $\tilde{Z}_{\scriptscriptstyle W}(\mathbf{x}; u)$ can only have
simple poles. Our goal is now to compute the residue of this function at each of
these poles.

The following theorem provides the key calculation:
\vskip5pt
\begin{thm} \label{Residue-tilde-Z-avg} --- The residue of 
	$\tilde{Z}_{\scriptscriptstyle W}(\mathbf{x}; u)$ at 
	$
	x_{\scalebox{1.1}{$\scriptscriptstyle 5$}} = u^{\scalebox{1.1}{$\scriptscriptstyle -1$}}
	$ 
	is given by the formula 
\begin{equation*}
\begin{split}
R(\underline{x};u) : & = \lim_{x_{\scalebox{.75}{$\scriptscriptstyle 5$}} \to \scalebox{1.}{$\scriptscriptstyle 1\slash $}u} (1 - u x_{\scalebox{1.1}{$\scriptscriptstyle 5$}})\tilde{Z}_{\scriptscriptstyle W}(\mathbf{x}; u) \\
&= \frac{1}{\left(P_{\phantom i}^{\scalebox{1.1}{$\scriptscriptstyle 2$}}; 
P_{\phantom i}^{\scalebox{1.1}{$\scriptscriptstyle 2$}}\right)_{\infty}
\!\left(u_{\phantom i}^{\scalebox{1.1}{$\scriptscriptstyle 2$}}
P_{\phantom i}^{\scalebox{1.1}{$\scriptscriptstyle 2$}}; 
P_{\phantom i}^{\scalebox{1.1}{$\scriptscriptstyle 2$}}\right)_{\infty}
\prod_{\scalebox{1.1}{$\scriptscriptstyle i = 1$}}^{\scalebox{1.1}{$\scriptscriptstyle 4$}}  
\left(x_{\scalebox{1.1}{$\scriptscriptstyle i$}}^{\scalebox{1.1}{$\scriptscriptstyle 2$}}; 
P_{\phantom i}^{\scalebox{1.1}{$\scriptscriptstyle 2$}}\right)_{\infty} 
\!\left(u_{\phantom i}^{\scalebox{1.1}{$\scriptscriptstyle 2$}}
x_{\scalebox{1.1}{$\scriptscriptstyle i$}}^{\scalebox{1.1}{$\scriptscriptstyle - 2$}}
P_{\phantom i}^{\scalebox{1.1}{$\scriptscriptstyle 2$}}; 
P_{\phantom i}^{\scalebox{1.1}{$\scriptscriptstyle 2$}}\right)_{\infty}
\cdot \;\,  \prod_{\scalebox{1.1}{$\scriptscriptstyle 1\, \le \, i \, < \, j \, \le \, 4$}}  
\left(x_{\scalebox{1.1}{$\scriptscriptstyle i$}}
x_{\scalebox{1.1}{$\scriptscriptstyle j$}}; P\right)_{\infty}}
\end{split}
\end{equation*} 
where 
$
P = (x_{\scalebox{1.1}{$\scriptscriptstyle 1$}}
x_{\scalebox{1.1}{$\scriptscriptstyle 2$}}
x_{\scalebox{1.1}{$\scriptscriptstyle 3$}}
x_{\scalebox{1.1}{$\scriptscriptstyle 4$}})\slash u^{\scalebox{1.1}{$\scriptscriptstyle 2$}}
$ 
and $(a; b)_{\infty} = \prod_{k\ge 0}\, (1 - ab^{\scalebox{1.1}{$\scriptscriptstyle k$}})$ is the 
$b$-Pochhammer symbol. 
\end{thm}
The proof of this theorem will be given in the rest of this section.
Using results in \cite{DIPP}, similar ideas can be\linebreak used to prove
explicit formulas for multiple residues of the average zeta
function of arbitrary simply laced reduced irreducible affine root
systems.

To compute the residue at any other pole of
$\tilde{Z}_{\scriptscriptstyle W}(\mathbf{x}; u),$ let $\alpha$ be an
affine positive real root, and write it as\linebreak
$
\alpha = \sum_{i} n_{\scriptscriptstyle i}\alpha_{\scriptscriptstyle
  i}
$
in terms of simple roots. Let
$
w_{\scalebox{1.1}{$\scriptscriptstyle \alpha$}}
$
be an element of the {W}eyl group sending $\alpha$ to
$\alpha_{\scriptscriptstyle 5},$ and if $\alpha \ne
\alpha_{\scriptscriptstyle i},$ hence $n_{\scriptscriptstyle 5} \ge
1,$ let $\zeta$ be a
$2n_{\scriptscriptstyle 5}$-th root of $1$ in $\mathbb{C}.$
Let
$
C_{\alpha,\,  \scalebox{1.1}{$\scriptscriptstyle \zeta$}}(\underline{x};u)
$
denote the limit
\[
C_{\alpha,\,  \scalebox{1.1}{$\scriptscriptstyle
    \zeta$}}(\underline{x};u) : = \hskip-44pt
\lim_{\hskip44pt x_{\scalebox{.75}{$\scriptscriptstyle 5$}}\to
  \zeta^{\scalebox{1.0}{$\scriptscriptstyle -1$}}(u
 \mathbf{x}^{\alpha \scalebox{.75}{$\scriptscriptstyle '$}})^{\scalebox{1.0}{$\scriptscriptstyle -1$}\slash n_{\scalebox{.75}{$\scriptscriptstyle 5$}}}}
\!\left(1 -
\zeta u^{1\slash n_{\scalebox{.75}{$\scriptscriptstyle 5$}}}
\mathbf{x}^{\alpha \slash n_{\scalebox{.75}{$\scriptscriptstyle 5$}}}\right)
\!\tilde{Z}_{\scriptscriptstyle W}(\mathbf{x}; u)
\]
where
  $
  \alpha' : = \alpha - n_{\scriptscriptstyle
    5}\alpha_{\scriptscriptstyle 5}
  $.
  Here, we chose the principal branch of the complex logarithm, and so
$
C_{\alpha,\,  \scalebox{1.1}{$\scriptscriptstyle \zeta$}}(\underline{x};u)
$
is well-defined and analytic at least
when $x_{\scalebox{1.1}{$\scriptscriptstyle 1$}}, \ldots,
x_{\scalebox{1.1}{$\scriptscriptstyle 4$}}$ and $u$ are away from the
non-positive real axis.
\vskip5pt
\begin{lem} \label{L6.2} --- With notation as above, we have
  \begin{equation} \label{eq: residue-general}
C_{\alpha,\,  \scalebox{1.1}{$\scriptscriptstyle \zeta$}}(\underline{x};u)
= \frac{1}{2n_{\scalebox{1.1}{$\scriptscriptstyle 5$}}}
\!\left.
\left\{R(w_{\scalebox{1.1}{$\scriptscriptstyle \alpha$}}\underline{x};u) \cdot
f_{w_{\scalebox{1.1}{$\scriptscriptstyle \alpha$}}}
(\xx;u)\right\}
\right\vert_{x_{\scalebox{.75}{$\scriptscriptstyle 5$}}\,=\, 
  \zeta^{\scalebox{1.0}{$\scriptscriptstyle -1$}}(u
 \mathbf{x}^{\alpha \scalebox{.75}{$\scriptscriptstyle
     '$}})^{\scalebox{1.0}{$\scriptscriptstyle -1$}\slash
   n_{\scalebox{.75}{$\scriptscriptstyle 5$}}}} 
\end{equation}

where $f_w=(1+ux_{\scalebox{1.1}{$\scriptscriptstyle 5$}})\| w$ and 
$
w_{\scalebox{1.1}{$\scriptscriptstyle \alpha$}}
\underline{x}
$
represents the first four components of
$
w_{\scalebox{1.1}{$\scriptscriptstyle \alpha$}}
\mathbf{x}.
$
\end{lem}

\begin{proof} By applying the functional equation~\eqref{eq:
    f-e-loc} of the vector function $\tZZ =\ov{\tZW}$ corresponding to
  $w=\ss_{\sss 5},$ one finds that
  \[
    \lim_{x_{\scalebox{.75}{$\scriptscriptstyle 5$}} \to e \scalebox{1.}{$\scriptscriptstyle /  $} u}
(1 -e u x_{\scalebox{1.1}{$\scriptscriptstyle 5$}}) \mathbf{\tilde{Z}}(\mathbf{x}; u)=
\frac 12 R(\underline{x};u)\cdot \, \!^{t}(1, e, 0) 
\]
where $e\in \{-1,1\}$. Applying~\eqref{eq: f-e-loc} again, for
$
w=w_{\scalebox{1.1}{$\scriptscriptstyle \alpha$}},
$
we obtain
\[
  \hskip-44pt \lim_{\hskip44pt x_{\scalebox{.75}{$\scriptscriptstyle 5$}}\to
  \zeta^{\scalebox{1.0}{$\scriptscriptstyle -1$}}(u
 \mathbf{x}^{\alpha \scalebox{.75}{$\scriptscriptstyle '$}})^{\scalebox{1.0}{$\scriptscriptstyle -1$}\slash n_{\scalebox{.75}{$\scriptscriptstyle 5$}}}}
\!\left(1 -
\zeta u^{1\slash n_{\scalebox{.75}{$\scriptscriptstyle 5$}}}
\mathbf{x}^{\alpha \slash n_{\scalebox{.75}{$\scriptscriptstyle 5$}}}\right)
\tZZ(\xx;u) = \frac{1}{2n_{\scalebox{1.1}{$\scriptscriptstyle 5$}}}
\!\left.
\left\{R(w_{\scalebox{1.1}{$\scriptscriptstyle \alpha$}}
\underline{x};u) \tL_{w_\a}\!(\xx)\right\}
\right\vert_{x_{\scalebox{.75}{$\scriptscriptstyle 5$}}\,=\, 
  \zeta^{\scalebox{1.0}{$\scriptscriptstyle -1$}}(u
 \mathbf{x}^{\alpha \scalebox{.75}{$\scriptscriptstyle
     '$}})^{\scalebox{1.0}{$\scriptscriptstyle -1$}\slash
   n_{\scalebox{.75}{$\scriptscriptstyle 5$}}}} \cdot  
   \,\!^{t}(1, \zeta^{n_{\scalebox{.75}{$\scriptscriptstyle 5$}}}\!, 0).
 \]
 The formula in the lemma now follows using~\eqref{eq:cocycle},
 after multiplying on the left by the vector $(1,1,1)$. 
\end{proof}

If $\alpha = \alpha_{\scriptscriptstyle i},$ for some $1\le i \le 4,$
the residue of $\tilde{Z}_{\scriptscriptstyle W}(\mathbf{x}; u)$ at
$x_{\scalebox{1.1}{$\scriptscriptstyle i$}} = 1\slash u$ can be
computed similarly from the functional equation \eqref{eq: f-e-loc}
with
$
w_{\scalebox{1.1}{$\scriptscriptstyle \alpha$}} =
\sigma_{\scriptscriptstyle i}\sigma_{\scriptscriptstyle 5},
$
which sends the simple root $\alpha_{\scriptscriptstyle i}$ to
$\alpha_{\scriptscriptstyle 5}.$

\subsection{Proof of Theorem~\ref{Residue-tilde-Z-avg}}\label{s6.1}
We start by recalling a result from~\cite{BD}, which identifies the
residue $R(\ux; u)$ up to a function depending only upon
$
P = (x_{\scalebox{1.1}{$\scriptscriptstyle 1$}}
x_{\scalebox{1.1}{$\scriptscriptstyle 2$}}
x_{\scalebox{1.1}{$\scriptscriptstyle 3$}}
x_{\scalebox{1.1}{$\scriptscriptstyle 4$}})\slash u^{\scalebox{1.1}{$\scriptscriptstyle 2$}}
$
and $u$. 
\vskip5pt
\begin{prop} \label{Initia-formula-residue} --- We have
\[
R(\ux; u) = f(P, u)\cdot
\frac{1}{\prod_{\scalebox{1.1}{$\scriptscriptstyle i = 1$}}^{\scalebox{1.1}{$\scriptscriptstyle 4$}}  
\left(x_{\scalebox{1.1}{$\scriptscriptstyle i$}}^{\scalebox{1.1}{$\scriptscriptstyle 2$}}; 
P_{\phantom i}^{\scalebox{1.1}{$\scriptscriptstyle 2$}}\right)_{\infty} 
\!\left(u_{\phantom i}^{\scalebox{1.1}{$\scriptscriptstyle 2$}}
x_{\scalebox{1.1}{$\scriptscriptstyle i$}}^{\scalebox{1.1}{$\scriptscriptstyle - 2$}}
P_{\phantom i}^{\scalebox{1.1}{$\scriptscriptstyle 2$}}; 
P_{\phantom i}^{\scalebox{1.1}{$\scriptscriptstyle 2$}}\right)_{\infty}
\cdot \;\,  \prod_{\scalebox{1.1}{$\scriptscriptstyle 1\, \le \, i \, < \, j \, \le \, 4$}}  
\left(x_{\scalebox{1.1}{$\scriptscriptstyle i$}}
x_{\scalebox{1.1}{$\scriptscriptstyle j$}}; P\right)_{\infty}}
\]
where the function $f(P, u)$ is meromorphic in the region $|P|<1$,
with possible poles only when $P^{n}=\pm u^{\sss -2}$
for\linebreak some $n\ge 1$. In addition, if we set
$
\mathrm{z} :=
x_{\scalebox{1.1}{$\scriptscriptstyle 1$}}
x_{\scalebox{1.1}{$\scriptscriptstyle 2$}}
x_{\scalebox{1.1}{$\scriptscriptstyle 3$}}
x_{\scalebox{1.1}{$\scriptscriptstyle 4$}}
$,
then for every $\epsilon > 0$, the function
$
f(u^{\scalebox{1.1}{$\scriptscriptstyle 2$}}\mathrm{z}, u)
$
is holomorphic in\linebreak the polydisc
$
|\mathrm{z}| < (1 + \epsilon)^{\sss -4},\, |u| < 1 + \epsilon.
$
\end{prop}
\begin{proof} For $|P|$ sufficiently small, the decomposition follows
  from~\cite[Lemma 4.30]{BD}. From
  Theorem~\ref{Divisibility-and-analytic-continuation}, it\linebreak
  follows that the function
$
\left.\left[D(\xx;u)\slash (1-u^{\sss 2}x_{\sss 5}^{\sss 2})\right]\right
\vert_{x_{\scalebox{.75}{$\scriptscriptstyle 5$}}=\scalebox{1.}{$\scriptscriptstyle 1\slash $}u} \cdot
  R(\ux;u)
  $
  is holomorphic in the region $|P|<1$, $u\ne 0$; when $|P|$ is small,
  we can express this function as
  \[
%\left. \frac{D(\xx;u)}{1-u^{\sss 2}x_{\sss 5}^{\sss 2}}\right|_{x_{\sss 5}=u^{\sss -1}} \cdot R(\ux;u)=
f(P,u) \left(P_{\phantom i}^{\scalebox{1.1}{$\scriptscriptstyle 2$}}; 
P_{\phantom i}^{\scalebox{1.1}{$\scriptscriptstyle 2$}}\right)_{\infty}
\!\left(u_{\phantom i}^{\scalebox{1.1}{$\scriptscriptstyle 4$}}
P_{\phantom i}^{\scalebox{1.1}{$\scriptscriptstyle 2$}}; 
P_{\phantom i}^{\scalebox{1.1}{$\scriptscriptstyle 2$}}\right)_{\infty}
\prod_{\scalebox{1.1}{$\scriptscriptstyle i = 1$}}^{\scalebox{1.1}{$\scriptscriptstyle 4$}}  
\left(u_{\phantom i}^{\scalebox{1.1}{$\scriptscriptstyle 2$}}
x_{\scalebox{1.1}{$\scriptscriptstyle i$}}^{\scalebox{1.1}{$\scriptscriptstyle 2$}}; 
P_{\phantom i}^{\scalebox{1.1}{$\scriptscriptstyle 2$}}\right)_{\infty} 
\!\left(u_{\phantom i}^{\scalebox{1.1}{$\scriptscriptstyle 4$}}
x_{\scalebox{1.1}{$\scriptscriptstyle i$}}^{\scalebox{1.1}{$\scriptscriptstyle - 2$}}
P_{\phantom i}^{\scalebox{1.1}{$\scriptscriptstyle 2$}}; 
P_{\phantom i}^{\scalebox{1.1}{$\scriptscriptstyle 2$}}\right)_{\infty}
\, \cdot \,  \prod_{\scalebox{1.1}{$\scriptscriptstyle 1\, \le \, i \, < \, j \, \le \, 4$}}  
\left(-x_{\scalebox{1.1}{$\scriptscriptstyle i$}}
x_{\scalebox{1.1}{$\scriptscriptstyle j$}}; P\right)_{\infty}.
\]
However, the product of the Pochhammer symbols is holomorphic when
$|P|<1$, and if in addition we take\linebreak $P^{n}\ne \pm u^{\sss -2}$ for all $n\ge
1$, we can choose $x_{\sss i}\in \CC$ $(i = 1, \ldots, 4)$
such that this product is non-zero. It follows that indeed
$f(P,u)$ is holomorphic for $|P|<1$ except for possible
poles when $P^{n}=\pm u^{\sss -2}$ for some $n\ge 1$,
and that the decomposition of $R(\ux; u)$ extends by
analytic continuation.

On the other hand, by Lemma \ref{initial-properties},
part {\normalfont\itshape 1}, and the functional equation
\eqref{eq: poly-P-Q-func-eq} of the polynomials
$
Q_{\scalebox{1.1}{$\scriptscriptstyle \underline{k}$}}(x_{\scalebox{1.1}{$\scriptscriptstyle 5$}}; u)
$,
we have
\begin{equation} \label{eq: holomorphicity-residue-u}
R(u\underline{x};u) \; =\sum_{|\underline{k}| \, \equiv \, 0 \!\!\!\!\pmod 2}
Q_{ \scalebox{1.1}{$\scriptscriptstyle \underline{k}$}}(u; u)
    \underline{x}^{\scalebox{1.1}{$\scriptscriptstyle \underline{k}$}}
  \end{equation}
  which implies easily that
$
f(u^{\scalebox{1.1}{$\scriptscriptstyle 2$}}\mathrm{z}, u)
$ 
(as a function of the variables
$
\mathrm{z} =
x_{\scalebox{1.1}{$\scriptscriptstyle 1$}}
x_{\scalebox{1.1}{$\scriptscriptstyle 2$}}
x_{\scalebox{1.1}{$\scriptscriptstyle 3$}}
x_{\scalebox{1.1}{$\scriptscriptstyle 4$}}
$
and $u$) is holomorphic in the polydisc
$
|\mathrm{z}|
< (1 + \epsilon)^{\sss -4},\, |u| < 1 + \epsilon.
$
This completes the proof.
\end{proof}
Thus it remains to show that
\be\label{22}
f(P,u)=\frac{1}{\left(P_{\phantom i}^{\scalebox{1.1}{$\scriptscriptstyle 2$}}; 
P_{\phantom i}^{\scalebox{1.1}{$\scriptscriptstyle 2$}}\right)_{\infty}
\!\left(u_{\phantom i}^{\scalebox{1.1}{$\scriptscriptstyle 2$}}
P_{\phantom i}^{\scalebox{1.1}{$\scriptscriptstyle 2$}}; 
P_{\phantom i}^{\scalebox{1.1}{$\scriptscriptstyle 2$}}\right)_{\infty}}.
\ee
The proof of this formula involves the following two steps:
\begin{enumerate}
 \item[(1)] Formula~\eqref{22} holds for $u=-1$. This is an immediate consequence of Macdonald's formula. 

 \item[(2)] The function
\[
g(P,u)=f(P,u)\left(P_{\phantom i}^{\scalebox{1.1}{$\scriptscriptstyle 2$}}; P_{\phantom i}^{\scalebox{1.1}{$\scriptscriptstyle 2$}}\right)_{\infty}
\!\left(u_{\phantom i}^{\scalebox{1.1}{$\scriptscriptstyle 2$}}
P_{\phantom i}^{\scalebox{1.1}{$\scriptscriptstyle 2$}}; 
P_{\phantom i}^{\scalebox{1.1}{$\scriptscriptstyle 2$}}\right)_{\infty}
\]
is invariant under the transformation $(P,u)\mapsto (P, u\slash P).$
(It is here where we crucially use the extra\linebreak functional equation
from~\ref{s4.2}.)
\end{enumerate}
The proofs of (1) and (2) will be given in the next two subsections,
but for now let us show how they imply\linebreak \eqref{22}.

By Proposition \ref{Initia-formula-residue}, the function
$g(u^{\scalebox{1.1}{$\scriptscriptstyle 2$}}\mathrm{z}, u)$
is holomorphic in the polydisc
$
|\mathrm{z}| < (1 + \epsilon)^{\sss -4},\, |u| < 1 + \epsilon
$,
hence\linebreak a normally convergent power series
in $\mathrm{z}$ and $u$. Substituting
$
\mathrm{z} \to P\slash u^{\scalebox{1.1}{$\scriptscriptstyle 2$}} 
$,
we can then write
\[
g(P,u) \; =\sum_{m, n\in \Z} a_{m,n} u^{m} P^{n}
\]
the series being normally convergent for, say
$
\epsilon \le |u| < 1 + \epsilon
$
and
$
|P| < |u|^{\scalebox{1.1}{$\scriptscriptstyle 2$}}(1 + \epsilon)^{\sss -4}
$.
Notice that $a_{m,n}=0$ if $n < 0$. Since
$g(P,u) = g(P,uP)$, it follows at once (after
taking $u$ and $P$ such that $\epsilon \le |uP| < 1 + \epsilon$)
that\linebreak $a_{m,n} = a_{m,n - m}$ for all $m,n \in \Z$.
Iterating, one finds that
$
a_{m,n}=a_{m,n+km}
$
for $k\in \Z$. Since $a_{m,n}=0$ for\linebreak $n < 0$, it
follows that $a_{m,n} = 0$ if $m \ne 0$. Therefore
$g(P,u)$ is independent of $u$, and by (1), $g\equiv 1$.

\subsection{The residue for $u=-1$}
From the relation between $\ZW(\mathbf{x};u)$ and $\tZW(\mathbf{x}; u)$, we can write
\[
R(\ux;u)=\frac{1}{\left(u_{\phantom i}^{\scalebox{1.1}{$\scriptscriptstyle 2$}}
P_{\phantom i}; 
P_{\phantom i}^{\scalebox{1.1}{$\scriptscriptstyle 2$}}\right)_{\infty}^{\scalebox{1.1}{$\scriptscriptstyle 2$}}\cdot 
\DD_{\sss 5}(\ux, 1\slash u)} \, \cdot 
\lim_{x_{\scalebox{.75}{$\scriptscriptstyle 5$}} \to
  \scalebox{1.}{$\scriptscriptstyle 1\slash $}u}
\frac{1-ux_{\sss 5}}{1-x_{\sss 5}^{\sss 2}}\ZW(\xx;u)
\]
where we put
$
\DD_{\sss 5}(\xx):=\DD(\xx)\slash (1-x_{\sss 5}^{\sss 2})
$.
When $u=-1$, we see by induction on $\ell(w)$ that 
$$
1\vert w(\xx;-1)=(- 1)^{\scalebox{1.1}{$\scriptscriptstyle \ell(w)$}}
\prod_{\beta \, \in \, \Phi(w)} \mathbf{x}^{\beta}
$$
and so $\ZW(\xx;-1)=F_{\sss MD}(\xx)$, where
$
F_{\sss MD}(\xx)=\sum_{w\in W} (- 1)^{\scalebox{1.1}{$\scriptscriptstyle \ell(w)$}}
\prod_{\beta \, \in \, \Phi(w)} \mathbf{x}^{\beta}
$
is the function studied by Macdonald~\cite{Mac}. Then
\[
  \lim_{u \to -1}
  \,\lim_{x_{\scalebox{.75}{$\scriptscriptstyle 5$}} \to
    \scalebox{1.}{$\scriptscriptstyle 1\slash $}u}
  \frac{1-ux_{\sss 5}}{1-x_{\sss 5}^{\sss 2}} \ZW(\xx;u)=
\frac 12 \ZW(\ux,-1;-1)=\frac{1}{2}F_{\sss MD} (\ux, -1)
\]
where in the first equality we interchanged the two limits;
this is justified since $\DD(\xx)$, and in particular
$1-x_{\sss 5}^{\sss 2}$, divides $\ZW(\xx;u)$
(see \cite[Section~4]{BD}), and so the function inside the
double limit is continuous at $(1,-1)$ as a function of
$z=ux_{\sss 5}$ and $u$. Consequently,
\be\label{39}
R(\ux;-1)=\frac{\frac{1}{2}F_{\sss MD} (\ux, -1)}{\left(P_{\phantom i}; 
P_{\phantom i}^{\scalebox{1.1}{$\scriptscriptstyle 2$}}\right)_{\infty}^{\scalebox{1.1}{$\scriptscriptstyle 2$}}\cdot
\DD_{\sss 5}(\ux, -1)}. 
\ee
Using the list of roots of $D_{\sss 4}$ at the beginning 
of Section~\ref{s4}, we have by Macdonald's formula~\cite{Mac} 
\[
\begin{aligned}
F_{\sss MD}(\xx)= 
&\prod_{i=1}^{4}\left(x_{\sss i}; \xx^{\dd}\right)_{\infty}
\left(x_{\sss i}^{\sss -1}\xx^{\dd}; \xx^{\dd}\right)_{\infty} 
\left(x_{\sss i}x_{\sss 5}; \xx^{\dd}\right)_{\infty}
\left(x_{\sss i}^{\sss -1} x_{\sss 5}^{\sss -1}\xx^{\dd}; \xx^{\dd}\right)_{\infty} 
\prod_{1\le i< j\le 4}\left(x_{\sss i}x_{\sss j}x_{\sss 5}; 
\xx^{\dd}\right)_{\infty} \\
&\cdot
\left(\xx^{\dd};\xx^{\dd}\right)^{\scalebox{1.1}{$\scriptscriptstyle 4$}}_{\infty}
\left(x_{\sss 5};\xx^{\dd}\right)_{\infty}
\left(x_{\sss 5}^{\sss -1}\xx^{\dd};\xx^{\dd}\right)_{\infty}.
\end{aligned}
\]
Also, $\DD(\xx)=F_{\sss MD} (x_{\sss 1}^{\sss 2}, \ldots, x_{\sss 5}^{\sss 2})$,
and one can easily check that~\eqref{39} matches the formula
in Theorem~\ref{Residue-tilde-Z-avg} for\linebreak $u=-1$.
This completes the proof of step (1) in~\ref{s6.1}.

\subsection{Functional equations of the residue function}
\label{s6.3}
From the formula of the residue function in Proposition
\ref{Initia-formula-residue}, it is clear that
$
R(\underline{x};u) = R(-\underline{x};u).
$
Then by Lemma \ref{initial-properties}, part
\!{\normalfont\itshape 1}, it follows at once that
%The functional equation~\eqref{eq: f-e-loc}, with $w=\ss_{\sss 5}$,
%implies that
\be\label{21}
\lim_{x_{\scalebox{.75}{$\scriptscriptstyle 5$}} \to
  \scalebox{1.}{$\scriptscriptstyle 1 \slash $}u}
(1-ux_{\sss 5})\tZZ(\xx; u) = \frac 12 R(\underline{x};u)\cdot v_{\sss 0}
\ee
where we set $v_{\sss 0}= \, \!^{t}(1, 1, 0)$.

We need the extension of the cocycle $\tL(\xx;u)$ to $W\oplus \Z$,
as in Section~\ref{s4.2}, by
$
\tL_{\tau^{-1}}(\xx;u)=\tB(\xx;u)
$,
with $\tB(\xx;u)$ defined as in the line following~\eqref{300}, and
$$
\tL_{\tau}(\xx;u)=\tA(\xx;u):=A(\xx;u)\cdot\left(1-u^{2}\slash \xx^{\dd}\right)^{2}
$$
with $A(\xx;u)$ from Theorem~\ref{T1}. \!Then the functional
equation~\eqref{200} holds for~$\tZZ(\xx;u)$,
with~$\Lam(\xx;u)$ re-\linebreak placed by~$\tL(\xx;u)$.
As in Section~\ref{s4.2}, we have an action of $\tau$ extending the action $\|$ 
such that 
\be \label{42}
\ov{f\|\tau}(\xx;u)=\tA(\xx;u) \cdot \bar{f}(\xx;u\slash \xx^\dd)
\ee 
for all $f\in\CC(\xx,u)_{\sss 0}$. 
\vskip5pt
\begin{prop}%\label{p5} 
--- Let $w\in W$ be such that
  $w \a_{\sss 5}=\a_{\sss 5}-\dd$. Then
\[
  R(\underline{x};u)=\lam_{w}(\ux;u)
\cdot \!\left. R(w\underline{x};u\slash\xx^{\dd})\right \vert_{x_{\scalebox{.75}{$\scriptscriptstyle 5$}}=\scalebox{1.15}{$\scriptscriptstyle 1\slash u$}}  
\]
where
$
\lambda_{w}(\ux;u)
=\left. [(1+ux_{\sss 5})\| \tau  w]^{\eee} (\xx;u)\right\vert_{x_{\scalebox{.75}{$\scriptscriptstyle 5$}}=\scalebox{1.15}{$\scriptscriptstyle 1\slash u$}} 
$.
\end{prop}
Notice from~\eqref{42} that
$
[(1+ux_{\sss 5})\| \tau](\xx;u)
=\tA_{\eee}(\xx;u)+(ux_{\sss 5}\slash \xx^{\dd})\tA_{\eo}(\xx;u)$,
where $\tA_{\eee}$ (resp. \!$\tA_{\eo}$) is the\linebreak sum of the entries
in the first (resp. \!second) column of the matrix $\tA(\xx;u)$.
\begin{proof} Since
  $
  w\xx=(w\ux, x_{\sss 5}\xx^{\dd})
  $,
  the analogue of the functional equation~\eqref{200} for
  $\tZZ(\xx;u),$ applied for $w\tau$, together with~\eqref{21} gives
  \[
 R(\ux;u)v_{\sss 0}= \!\left.\left\{R(w\underline{x};u\slash\xx^{\dd})
     \tL_{w\tau}(\xx;u)\right\}\right\vert_{x_{\scalebox{.75}{$\scriptscriptstyle 5$}}=\scalebox{1.15}{$\scriptscriptstyle 1\slash u$}} \cdot \, v_{\sss 0}.
\]
Note that
$
\tL_{w\tau}(\xx;u)=\tL_{\tau}(\xx;u)\tL_{w}(\xx;u\slash \xx^{\dd})
$
does not have a pole at $x_{\sss 5}=1\slash u;$ by Theorem~\ref{T1},
the ma-\linebreak trix $\tL_{\tau}(\xx;u)$ does not have a pole there,
and since $\a_{\sss 5}+\dd \notin \Phi(w)$, neither does
$\tL_{w}(\xx;u\slash \xx ^{\dd})$. We thus ob-\linebreak tain the functional
equation of $R(\ux;u)$, with $\lam_{w}(\ux;u)$ such that
$
\left. \tL_{w\tau} (\xx;u)
\right\vert_{x_{\scalebox{.75}{$\scriptscriptstyle 5$}}=\scalebox{1.15}{$\scriptscriptstyle 1\slash u$}} \cdot \, 
v_{\sss 0}  =  \lambda_{w}(\ux;u)
v_{\sss 0}.
$
Com-\linebreak bining~\eqref{42} and~\eqref{eq:cocycle}, we get the formula for
$\lambda_{w}(\ux;u)$, and completes the proof.
\end{proof}
By applying this proposition to the element
$
t=\ss_{\sss 1}\ss_{\sss 2}\ss_{\sss 3}\ss_{\sss 4}\ss_{\sss 5}$,
we obtain that 
\[
  R(\ux;u)=\lambda_{t}(\ux;u)
  R\left(u\slash x_{\sss 1},\ldots,u\slash x_{\sss 4}; u\slash P\right)
\]
with $\lambda_{t}(\ux;u)$ computed explicitly from the matrix
$A(\xx;u)$ in~\cite{eDPP} as
\[
  \lambda_{t}(\ux;u)=(1-u^{\sss 2})
  \prod_{i=1}^{4} (1-u^{\sss 2}\slash x_{\sss i}^{\sss 2}) 
\prod_{1\le i<j\le 4} (1-u^{\sss 2}\slash x_{\sss i}x_{\sss j}).
\]
One can check easily that this factor matches the one from the same
functional equation applied to the right-hand side of the formula in
Theorem~\ref{Residue-tilde-Z-avg}, and thus completing the proof
of step (2) in \ref{s6.1}.

\section{WMDS associated to moments of L-series}  \label{MDS} 
By Theorem \ref{Divisibility-and-analytic-continuation} and the definition \eqref{eq: denominator} of 
$D(\mathbf{x}; u),$ when $0 < u < 1,$ the function $\tilde{Z}_{\scriptscriptstyle W}(\mathbf{x}; u)$ is holomorphic for 
$|x_{\scalebox{1.1}{$\scriptscriptstyle i$}}| < 1$ ($i = 1, \ldots,
5$), and thus, in this polydisc,
\be\label{30}  
\tilde{Z}_{\scriptscriptstyle W}(\mathbf{x}; u) = 1 + \sum_{\mathrm{k}\, \ne \, \mathbf{0}} a(\mathrm{k}; u)\mathbf{x}^{\mathrm{k}}
\ee
the sum being over tuples 
$
\mathrm{k} = 
(k_{\scalebox{1.1}{$\scriptscriptstyle 1$}}, \ldots, k_{\scalebox{1.1}{$\scriptscriptstyle 5$}})\in 
\mathbb{N}^{5}\setminus \{\mathbf{0}\}.
$ 
The coefficients $a(\mathrm{k}; u)$ are polynomials in $u,$ and by Cauchy's inequalities, 
for every $\epsilon > 0,$ we have the estimate 
$
|a(\mathrm{k}; u)| \ll_{\epsilon} u^{- \epsilon |\mathrm{k}|}.
$

Before making the transition to multiple {D}irichlet series,
let us briefly recall some facts about quadratic {D}irichlet
$L$-functions in the rational function field setting.

Let $\mathbb{F}_{\! q}$ be a finite field of odd characteristic.
For
$
m \in \mathbb{F}_{\! q}[x],
$
$m \ne 0,$ set
$
|m| = q^{\deg m},
$
and for $d, m \in \mathbb{F}_{\! q}[x],$ $d \ne 0$ and
$m$ monic, let $\chi_{d}(m) = (d \slash m)$ denote
the quadratic symbol. We define $\chi_{d}(1) = 1,$ and if
$d \in \mathbb{F}_{\! q}^{\times},$\linebreak
we have $\chi_{d}(m) = \mathrm{sgn}(d)^{\deg m},$
for all non-constant $m \in \mathbb{F}_{\! q}[x],$
where $\mathrm{sgn}(d) = 1$ or $-1$ according as
$d \in (\mathbb{F}_{\! q}^{\times})^{\scalebox{1.1}{$\scriptscriptstyle 2$}}$
or not.

For $d$ square-free, the $L$-function associated
to the primitive character $\chi_{d}$ is defined by
\[
L(s, \chi_{d}) \;\; = \sum_{\substack{m \, \in \, \mathbb{F}_{\! q}[x] \\
m - \text{monic}}}
\chi_{d}(m) |m|^{-s}
  \;\;\, =  \prod_{p - \text{monic \& irreducible}}
  \big(1 - \chi_{d} (p)|p|^{-s}\big)^{\scalebox{1.1}{$\scriptscriptstyle -1$}}
\]
for complex $s,$ with $\Re(s) > 1.$ When $d$
is non-constant this $L$-function turns out to be a polynomial
in $q^{-s}$\linebreak
of degree $\deg d - 1,$ and if
$d \in \mathbb{F}_{\! q}^{\times},$
\[
  L(s, \chi_{d}) = \frac{1}
  {1 - \mathrm{sgn}(d) q^{{\scalebox{1.1}{$\scriptscriptstyle 1$}} - s}};
\]
when
$
d \in (\mathbb{F}_{\! q}^{\times})^{\scalebox{1.1}{$\scriptscriptstyle 2$}},
$
the $L$-function is just $\zeta(s)$--the zeta-function. In addition,
if one defines
$
\gamma_{\scalebox{1.1}{$\scriptscriptstyle q$}}(s, d)
$
by\hfill\break
\[
\gamma_{\scalebox{1.1}{$\scriptscriptstyle q$}}(s, d) = 
q^{\frac{1}{2}{\scalebox{1.05}{$\scriptscriptstyle \left(3 \, + \, (-1)^{\deg d} \right)$}} \left(s - \frac{1}{2}\right)} 
\big(1  -  \mathrm{sgn}(d) q^{- s} \big)^{\frac{1}{2}
\scalebox{1.1}{$\scriptscriptstyle \left(1 \, + \, (-1)^{\deg d}\right)$}} 
\big(1 - \mathrm{sgn}(d) q^{s - \scalebox{1.1}{$\scriptscriptstyle 1$}} \big)^{- 
  \frac{1}{2} \scalebox{1.1}{$\scriptscriptstyle \left(1 \, + \, (-1)^{\deg d}\right)$}}
\]

then $L(s, \chi_{d})$ satisfies the functional equation
\[
L(s, \chi_{d}) =\gamma_{\scalebox{1.1}{$\scriptscriptstyle q$}}(s, d)
|d|^{\frac{1}{2} - s} L(1 - s, \chi_{d}).
\]

We now use the coefficients $a(\mathrm{k}; u)$ to build the relevant
family of multiple {D}irichlet series.

Let $\mathbb{F}_{\! q}$ be a finite field of odd characteristic, and fix an element 
$
\theta_{\scalebox{1.1}{$\scriptscriptstyle 0$}} \in \mathbb{F}_{\! q}^{\times} 
\setminus (\mathbb{F}_{\! q}^{\times})^{\scalebox{1.1}{$\scriptscriptstyle 2$}}.
$ 
For 
$
\mathbf{s} = (s_{\scriptscriptstyle 1}, \ldots, s_{\scriptscriptstyle 5})
$ 
and 
$
a_{\scriptscriptstyle 1},  a_{\scriptscriptstyle 2} 
\in \{1, \theta_{\scalebox{1.1}{$\scriptscriptstyle 0$}}\},
$ 
we define 
$
\mathscr{W}\!(\mathbf{s};  a_{\scriptscriptstyle 2}, a_{\scriptscriptstyle 1}) 
$ 
by a series (summed over monics in $\mathbb{F}_{\! q}[x]$) of the form 
\begin{equation} \label{eq: MDS-vers0}
\mathscr{W}\!(\mathbf{s}; a_{\scriptscriptstyle 2}, a_{\scriptscriptstyle 1}) 
\;\;\;  = \sum_{\substack{m_{\scalebox{.75}{$\scriptscriptstyle 1$}}, \ldots, m_{\scalebox{.75}{$\scriptscriptstyle 4$}}, 
\, d - \mathrm{monic} \\ d = d_{\scalebox{.75}{$\scriptscriptstyle 0$}}^{} 
d_{\scalebox{.75}{$\scriptscriptstyle 1$}}^{\scalebox{.8}{$\scriptscriptstyle 2$}}, 
\; d_{\scalebox{.75}{$\scriptscriptstyle 0$}}^{} \; \mathrm{square \; free}}}  
\frac{\chi_{a_{\scalebox{.8}{$\scriptscriptstyle 1$}} d_{\scalebox{.8}{$\scriptscriptstyle 0$}}}
(\widehat{m}_{\scriptscriptstyle 1}
\widehat{m}_{\scriptscriptstyle 2}
\widehat{m}_{\scriptscriptstyle 3}
\widehat{m}_{\scriptscriptstyle 4})
\chi_{a_{\scalebox{.8}{$\scriptscriptstyle 2$}}}(d_{\scriptscriptstyle 0})
A(m_{\scriptscriptstyle 1}, m_{\scriptscriptstyle 2}, m_{\scriptscriptstyle 3}, m_{\scriptscriptstyle 4}, d)}
{|m_{\scriptscriptstyle 1}|^{s_{\scalebox{.75}{$\scriptscriptstyle 1$}}} 
|m_{\scriptscriptstyle 2}|^{s_{\scalebox{.75}{$\scriptscriptstyle 2$}}} 
|m_{\scriptscriptstyle 3}|^{s_{\scalebox{.75}{$\scriptscriptstyle 3$}}}
|m_{\scriptscriptstyle 4}|^{s_{\scalebox{.75}{$\scriptscriptstyle 4$}}} 
|d|^{s_{\scalebox{.75}{$\scriptscriptstyle 5$}}}}
\end{equation} 
where $\widehat{m}_{\scriptscriptstyle i}$ ($i = 1, \ldots, 4$) is the part of 
$m_{\scriptscriptstyle i}$ coprime to $d_{\scriptscriptstyle 0}.$ 
The coefficients 
$
A(m_{\scalebox{1.1}{$\scriptscriptstyle 1$}}, 
m_{\scalebox{1.1}{$\scriptscriptstyle 2$}}, 
m_{\scalebox{1.1}{$\scriptscriptstyle 3$}}, 
m_{\scalebox{1.1}{$\scriptscriptstyle 4$}}, d)
$ 
are com-\linebreak pletely determined by the following two conditions:
\begin{enumerate}[label=(\roman*)]
	\item If $p$ is monic irreducible, then 
	\[ 
	A\big(p^{k_{\scalebox{.75}{$\scriptscriptstyle 1$}}}\!,\, 
	p^{k_{\scalebox{.75}{$\scriptscriptstyle 2$}}}\!,\, 
	p^{k_{\scalebox{.75}{$\scriptscriptstyle 3$}}}\!,\, 
	p^{k_{\scalebox{.75}{$\scriptscriptstyle 4$}}}\!, \, p^{l} \big)
	% = |p|^{k_{\scriptscriptstyle 1} + \cdots + k_{r} + l}\, 
	%\lambda \big(k_{\scriptscriptstyle 1}, \ldots, k_{r}, l; q^{- \deg \, p}\big)
	= a\big(k_{\scalebox{1.1}{$\scriptscriptstyle 1$}}, 
	k_{\scalebox{1.1}{$\scriptscriptstyle 2$}}, 
	k_{\scalebox{1.1}{$\scriptscriptstyle 3$}}, 
	k_{\scalebox{1.1}{$\scriptscriptstyle 4$}}, l; 
	q^{- (\deg p)\slash 2}\big).
	\] 
	
	\item For monic 
	$
	m_{\scalebox{1.1}{$\scriptscriptstyle 1$}}, 
	m_{\scalebox{1.1}{$\scriptscriptstyle 2$}}, 
	m_{\scalebox{1.1}{$\scriptscriptstyle 3$}}, 
	m_{\scalebox{1.1}{$\scriptscriptstyle 4$}}, d,
	$ 
	we have 
	\[  
	A(m_{\scalebox{1.1}{$\scriptscriptstyle 1$}}, 
	m_{\scalebox{1.1}{$\scriptscriptstyle 2$}}, 
	m_{\scalebox{1.1}{$\scriptscriptstyle 3$}}, 
	m_{\scalebox{1.1}{$\scriptscriptstyle 4$}}, d)
	\; = \prod_{\substack{p^{k_{\scalebox{.75}{$\scriptscriptstyle i$}}} \parallel \, 
m_{\scalebox{.75}{$\scriptscriptstyle i$}}\\ p^{l} \parallel \, d}} 
	A\big(p^{k_{\scalebox{.75}{$\scriptscriptstyle 1$}}}\!,\, 
	p^{k_{\scalebox{.75}{$\scriptscriptstyle 2$}}}\!,\, 
	p^{k_{\scalebox{.75}{$\scriptscriptstyle 3$}}}\!,\, 
	p^{k_{\scalebox{.75}{$\scriptscriptstyle 4$}}}\!, \, p^{l} \big)
	\] 
	the product being taken over monic irreducibles.
      \end{enumerate}

\vskip5pt
\begin{lem} \label{Initial-abs-conv-scriptW}
--- The series defining
  $\mathscr{W}\!(\mathbf{s}; a_{\scriptscriptstyle 2},
  a_{\scriptscriptstyle 1})$ is absolutely convergent for
  $\Re(s_{\scriptscriptstyle i}) >  1,$ $i = 1, \ldots, 5.$
\end{lem}

\begin{proof} The proof is similar to that of \cite[Lemma
  6.5]{Dia}. We first show that the series \eqref{eq: MDS-vers0} is
  absolutely con-\linebreak
  vergent when $\Re(s_{\scriptscriptstyle i})$ is sufficiently
  large. To see this,
  by condition (\romannumeral 1) and the estimate of the coefficients
  $a(\mathrm{k}; u),$ we have
  \[
  \big|A\big(p^{k_{\scalebox{.75}{$\scriptscriptstyle 1$}}}\!,\, 
	p^{k_{\scalebox{.75}{$\scriptscriptstyle 2$}}}\!,\, 
	p^{k_{\scalebox{.75}{$\scriptscriptstyle 3$}}}\!,\, 
	p^{k_{\scalebox{.75}{$\scriptscriptstyle 4$}}}\!, \, 
	p^{k_{\scalebox{.75}{$\scriptscriptstyle 5$}}} \big)\big|
\ll_{\epsilon, q} |p|^{\epsilon |\mathrm{k}|}
\]
for every monic irreducible $p,$ the implied constant depending upon
$q$ and $\epsilon,$ but it is independent of the degree of $p.$
Choosing $n \ge 1$ such that this constant is smaller than $q^{n},$
it follows from (\romannumeral 2) that
\[
|A(m_{\scalebox{1.1}{$\scriptscriptstyle 1$}}, 
	m_{\scalebox{1.1}{$\scriptscriptstyle 2$}}, 
	m_{\scalebox{1.1}{$\scriptscriptstyle 3$}}, 
	m_{\scalebox{1.1}{$\scriptscriptstyle 4$}}, d)|
        < |m_{\scalebox{1.1}{$\scriptscriptstyle 1$}} 
	m_{\scalebox{1.1}{$\scriptscriptstyle 2$}}
	m_{\scalebox{1.1}{$\scriptscriptstyle 3$}} 
	m_{\scalebox{1.1}{$\scriptscriptstyle 4$}}d|^{n + \epsilon}
      \]
      for all monic polynomials
$
m_{\scalebox{1.1}{$\scriptscriptstyle 1$}}, 
m_{\scalebox{1.1}{$\scriptscriptstyle 2$}}, 
m_{\scalebox{1.1}{$\scriptscriptstyle 3$}}, 
m_{\scalebox{1.1}{$\scriptscriptstyle 4$}}, d.
$
Since
\[
|\mathscr{W}\!(\mathbf{s}; a_{\scriptscriptstyle 2}, a_{\scriptscriptstyle 1})| 
\;\; \le \sum_{m_{\scalebox{.75}{$\scriptscriptstyle 1$}}, \ldots,
  m_{\scalebox{.75}{$\scriptscriptstyle 4$}},  \, d} \, 
\frac{|A(m_{\scriptscriptstyle 1}, m_{\scriptscriptstyle 2}, m_{\scriptscriptstyle 3}, m_{\scriptscriptstyle 4}, d)|}
{|m_{\scriptscriptstyle 1}|^{\scalebox{1.1}{$\scriptscriptstyle
      \Re(s_{\scalebox{.75}{$\scriptscriptstyle 1$}})$}} 
|m_{\scriptscriptstyle 2}|^{\scalebox{1.1}{$\scriptscriptstyle
      \Re(s_{\scalebox{.75}{$\scriptscriptstyle 2$}})$}}
|m_{\scriptscriptstyle 3}|^{\scalebox{1.1}{$\scriptscriptstyle
      \Re(s_{\scalebox{.75}{$\scriptscriptstyle 3$}})$}}
|m_{\scriptscriptstyle 4}|^{\scalebox{1.1}{$\scriptscriptstyle
      \Re(s_{\scalebox{.75}{$\scriptscriptstyle 4$}})$}}
|d|^{\scalebox{1.1}{$\scriptscriptstyle
      \Re(s_{\scalebox{.75}{$\scriptscriptstyle 5$}})$}}}
\]
we have absolute convergence as long as
$\Re(s_{\scriptscriptstyle i}) >  n + 1 + \epsilon.$ On the other hand,
the right-hand side of the last inequality decomposes as
\[
  \prod_{p} \left(1 + \sum_{\mathrm{k} \ne \mathbf{0}}
  \big|a\big(\mathrm{k}; |p|^{\scalebox{1.1}{$\scriptscriptstyle - 1\slash 2$}}\big)\big|
    |p|^{\scalebox{1.1}{$\scriptscriptstyle -\langle \mathrm{k}, \Re({\mathbf{s}})\rangle$}} \right)
\] 
where
$
\Re({\mathbf{s}}) : = (\Re(s_{\scriptscriptstyle 1}), \ldots,
\Re(s_{\scriptscriptstyle 5})),
$
and $\langle \cdot, \cdot\rangle$ is the scalar product on
$\mathbb{R}^{5}.$ Taking the logarithm, our asser-\linebreak tion now
follows by a simple comparison with the series
\[ 
  \sum_{p} \sum_{\mathrm{k} \ne \mathbf{0}}
  |p|^{\scalebox{1.1}{$\scriptscriptstyle - \lambda|\mathrm{k}|$}}
\]
which is easily seen to be convergent when $\lambda > 1.$ 
\end{proof}

For simplicity, we shall assume from now on that $q\equiv 1 \!\pmod 4.$ Following \cite[Section 3]{Dia}, we can 
write\linebreak 
\begin{equation}  \label{eq: MDS-vers1} 
\mathscr{W}\!(\mathbf{s};  a_{\scriptscriptstyle 2}, a_{\scriptscriptstyle 1}) 
\; = \sum_{d = d_{\scalebox{.75}{$\scriptscriptstyle 0$}}^{} 
d_{\scalebox{.75}{$\scriptscriptstyle 1$}}^{\scalebox{.8}{$\scriptscriptstyle 2$}}} 
\frac{\prod_{i = 1}^{4} L\!\left(s_{\scriptscriptstyle i} + \tfrac{1}{2}, 
\chi_{a_{\scalebox{.8}{$\scriptscriptstyle 1$}} d_{\scalebox{.8}{$\scriptscriptstyle 0$}}}\right) 
\cdot \chi_{a_{\scalebox{.8}{$\scriptscriptstyle 2$}}}\!(d_{\scriptscriptstyle 0}) 
P_{\scalebox{1.1}{$\scriptscriptstyle d$}}(\mathbf{s}\scalebox{1.}{$\scriptscriptstyle '$}; 
\chi_{a_{\scalebox{.8}{$\scriptscriptstyle 1$}} 
d_{\scalebox{.8}{$\scriptscriptstyle 0$}}})}{|d|^{s_{\scalebox{.75}{$\scriptscriptstyle 5$}}}}
\end{equation} 
where 
$  
 P_{\scalebox{1.1}{$\scriptscriptstyle d$}}(\mathbf{s}\scalebox{1.}{$\scriptscriptstyle '$}; 
 \chi_{a_{\scalebox{.8}{$\scriptscriptstyle 1$}} d_{\scalebox{.8}{$\scriptscriptstyle 0$}}})
$ 
($
\mathbf{s}\scalebox{1.}{$\scriptscriptstyle '$} : = (s_{\scriptscriptstyle 1}, \ldots, s_{\scriptscriptstyle 4})
$) 
is the {D}irichlet polynomial defined by
\begin{equation} \label{eq: polyPd}
\begin{split} 
P_{\scalebox{1.1}{$\scriptscriptstyle d$}}(
s_{\scriptscriptstyle 1}, \ldots, 
s_{\scriptscriptstyle 4}; 
\chi_{a_{\scalebox{.8}{$\scriptscriptstyle 1$}} 
d_{\scalebox{.8}{$\scriptscriptstyle 0$}}}) \;\; = \!
 & \prod_{\substack{p^{l} \parallel \, d \\ l \, \equiv \, 1 \!\!\!\!\!\pmod 2}} \,
 P_{\scalebox{.95}{$\scriptscriptstyle l$}}\!\left(
|p|^{ - s_{\scalebox{.75}{$\scriptscriptstyle 1$}}}, \ldots, 
|p|^{ - s_{\scalebox{.75}{$\scriptscriptstyle 4$}}}; q^{- (\deg p)\slash 2}\right) \\
 & \cdot \prod_{\substack{p \, \mid \, d_{\scalebox{.75}{$\scriptscriptstyle 1$}} 
 \\ p^{l} \parallel \, d \\ l \, \equiv \, 0 \!\!\!\!\pmod 2}} \,
 P_{\scalebox{.95}{$\scriptscriptstyle l$}}\!\left(
\chi_{a_{\scalebox{.8}{$\scriptscriptstyle 1$}} 
d_{\scalebox{.8}{$\scriptscriptstyle 0$}}}(p) |p|^{ - s_{\scalebox{.75}{$\scriptscriptstyle 1$}}}, \ldots,
\chi_{a_{\scalebox{.8}{$\scriptscriptstyle 1$}} 
d_{\scalebox{.8}{$\scriptscriptstyle 0$}}}(p)|p|^{ - s_{\scalebox{.75}{$\scriptscriptstyle 4$}}}; 
q^{- (\deg p)\slash 2}\right).
\end{split}
\end{equation}
Here $P_{\scalebox{.95}{$\scriptscriptstyle l$}}(\underline{x}; u)$ are the polynomials in Lemma \ref{initial-properties}. The functions 
$
P_{\scalebox{1.1}{$\scriptscriptstyle d$}}(\mathbf{s}\scalebox{1.}{$\scriptscriptstyle '$}; 
\chi_{a_{\scalebox{.8}{$\scriptscriptstyle 1$}} d_{\scalebox{.8}{$\scriptscriptstyle 0$}}})
$ 
are symmetric in all vari-\linebreak 
ables, and by \eqref{eq: poly-P-Q-func-eq}, they satisfy a functional equation as $s_{\scriptscriptstyle i} \to  - s_{\scriptscriptstyle i}$ for each $1\le i \le 4.$ It is clear that \eqref{eq: MDS-vers1} 
converges absolutely for arbitrary 
$
s_{\scriptscriptstyle 1},\ldots, s_{\scriptscriptstyle 4} \in \mathbb{C} \setminus \{1\slash 2\}
$ 
as long as $s_{\scriptscriptstyle 5}$ has sufficiently large real
part.\linebreak

\vskip-13pt
We can also write 
 \begin{equation} \label{eq: MDS-vers2} 
 \mathscr{W}\!(\mathbf{s}; a_{\scriptscriptstyle 2}, a_{\scriptscriptstyle 1}) \;\; = \sum_{
 m_{\scalebox{.75}{$\scriptscriptstyle 1$}} 
 m_{\scalebox{.75}{$\scriptscriptstyle 2$}}
 m_{\scalebox{.75}{$\scriptscriptstyle 3$}}	
 m_{\scalebox{.75}{$\scriptscriptstyle 4$}}
 \, = \, n_{\scalebox{.75}{$\scriptscriptstyle 0$}}^{} 
 n_{\scalebox{.75}{$\scriptscriptstyle 1$}}^{\scalebox{.8}{$\scriptscriptstyle 2$}}} 
\frac{L\!\left(s_{\scalebox{1.}{$\scriptscriptstyle 5$}} + \tfrac{1}{2}, 
\chi_{a_{\scalebox{.8}{$\scriptscriptstyle 2$}}n_{\scalebox{.8}{$\scriptscriptstyle 0$}}}\right) 
\!\chi_{a_{\scalebox{.8}{$\scriptscriptstyle 1$}}}\!(n_{\scriptscriptstyle 0})
Q_{\underline{m}}(s_{\scalebox{1.}{$\scriptscriptstyle 5$}}; 
 \chi_{a_{\scalebox{.8}{$\scriptscriptstyle 2$}} n_{\scalebox{.8}{$\scriptscriptstyle 0$}}})}
  {|m_{\scriptscriptstyle 1}|^{s_{\scalebox{.75}{$\scriptscriptstyle 1$}}} 
|m_{\scriptscriptstyle 2}|^{s_{\scalebox{.75}{$\scriptscriptstyle 2$}}} 
|m_{\scriptscriptstyle 3}|^{s_{\scalebox{.75}{$\scriptscriptstyle 3$}}}
|m_{\scriptscriptstyle 4}|^{s_{\scalebox{.75}{$\scriptscriptstyle 4$}}}}
  \end{equation} 
  where, for 
  $
  \underline{m} = (m_{\scriptscriptstyle 1}, \ldots, m_{\scriptscriptstyle 4}),
  $ 
  the {D}irichlet polynomial 
  $  
  Q_{\underline{m}}(s_{\scalebox{1.}{$\scriptscriptstyle 5$}}; 
  \chi_{a_{\scalebox{.8}{$\scriptscriptstyle 2$}} n_{\scalebox{.8}{$\scriptscriptstyle 0$}}})
  $ 
  is given by 
  \begin{equation*} \label{eq: polyQm}
  Q_{\underline{m}}(s_{\scalebox{1.}{$\scriptscriptstyle 5$}}; 
  \chi_{a_{\scalebox{.8}{$\scriptscriptstyle 2$}}  n_{\scalebox{.8}{$\scriptscriptstyle 0$}}}) \;\;\, = \!
  \prod_{\substack{p^{k_{\scalebox{.75}{$\scriptscriptstyle i$}}} \parallel \, m_{\scalebox{.75}{$\scriptscriptstyle i$}}
  \\ |\underline{k}| \, \equiv \, 1 \!\!\!\!\!\pmod 2}} 
  Q_{\underline{k}}\!\left(|p|^{ - s_{\scalebox{.75}{$\scriptscriptstyle 5$}}}; 
  q^{-(\deg p)\slash 2}\right) \;\;\; \cdot
  \prod_{\substack{p \, \mid \, n_{\scalebox{.75}{$\scriptscriptstyle 1$}} 
  \\ p^{k_{\scalebox{.75}{$\scriptscriptstyle i$}}} \parallel \, m_{\scalebox{.75}{$\scriptscriptstyle i$}} 
  \\ |\underline{k}| \, \equiv \, 0 \!\!\!\!\pmod 2}}  
  Q_{\underline{k}}\!\left(\chi_{a_{\scalebox{.8}{$\scriptscriptstyle 2$}}  
  n_{\scalebox{.8}{$\scriptscriptstyle 0$}}}(p)
  |p|^{ - s_{\scalebox{.75}{$\scriptscriptstyle 5$}}};  q^{-(\deg p)\slash 2}\right).
  \end{equation*} 
  Again, by \eqref{eq: poly-P-Q-func-eq}, the polynomials 
  $
  Q_{\underline{m}}(s_{\scriptscriptstyle 5}; 
  \chi_{a_{\scalebox{.8}{$\scriptscriptstyle 2$}} 
  n_{\scalebox{.8}{$\scriptscriptstyle 0$}}})
  $ 
  satisfy a functional equation as $s_{\scriptscriptstyle 5} \to  - s_{\scriptscriptstyle 5},$ and for every 
  $s_{\scriptscriptstyle 5} \in \mathbb{C} \setminus \{1\slash 2\},$
  the series \eqref{eq: MDS-vers2} converges absolutely as long as all
  $s_{\scriptscriptstyle 1}, \ldots, s_{\scriptscriptstyle 4}$ have
  sufficiently large real parts.\linebreak
  \begin{prop} \label{functional-equations-MDS}
  --- Let $\sZ\!(\xx)$
  denote the function $\sW\!(\mathbf{s}; 1, 1)$ after substituting 
	$
	x_{\scalebox{1.1}{$\scriptscriptstyle i$}} = q^{ - s_{\scalebox{.75}{$\scriptscriptstyle i$}}},
        $ 
	$i = 1, \ldots, 5$. Then the functions $\sZ\!(\xx)$ and
        $\tilde{Z}_{\scriptscriptstyle W}(\mathbf{x}; \sqrt{q})$
        satisfy the same functional equation~\eqref{fun-eq-tZ}.    
      \end{prop}
      \begin{proof} This follows at once from \eqref{eq:
          poly-P-Q-func-eq}
        and the functional equation of $L(s, \chi_{d}).$
      \end{proof}
\begin{rem}
    This functional equation was stated first in \cite{BD}, in terms
    of the vector function 
  \[
  \pmb{\mathscr{W}}\!(\mathbf{s}) = 
  \frac{1}{2}\begin{pmatrix} 
  \mathscr{W}\!(\mathbf{s};  1, 1) + 
  \mathscr{W}\!(\mathbf{s}; \theta_{\scalebox{1.1}{$\scriptscriptstyle 0$}}, 1) \\
   \mathscr{W}\!(\mathbf{s};  1, 1) - 
  \mathscr{W}\!(\mathbf{s}; \theta_{\scalebox{1.1}{$\scriptscriptstyle 0$}}, 1)\\
   \mathscr{W}\!(\mathbf{s};  1, \theta_{\scalebox{1.1}{$\scriptscriptstyle 0$}}) + 
  \mathscr{W}\!(\mathbf{s}; \theta_{\scalebox{1.1}{$\scriptscriptstyle 0$}}, 
  \theta_{\scalebox{1.1}{$\scriptscriptstyle 0$}}) \\
\end{pmatrix}
\]
and predated the Chinta-Gunnells action. To make the
comparison with early work easier, and because it\linebreak generalizes
more readily to number fields (for which the Chinta-Gunnells action
is not available), we also state the vector form
of~\eqref{fun-eq-tZ}. The vector function obtained by substituting
$
x_{\scalebox{1.1}{$\scriptscriptstyle i$}} = q^{ - s_{\scalebox{.75}{$\scriptscriptstyle i$}}}
$
($i = 1, \ldots, 5$) in $\pmb{\mathscr{W}}\!(\mathbf{s})$
satisfies the same functional equation as the vector
\[
\mathbf{\tilde{Z}}_{\sss 1}\!(\mathbf{x}; u) : = 
\begin{pmatrix}
(\tilde{Z}_{\sss W})^{+}_{\sss 1}(\xx;u) \\%^{\scriptscriptstyle (\mathrm{ev})}(\xx; u) \\
(\tilde{Z}_{\sss W})^{-}_{\sss 1}(\xx;u)\\
(\tilde{Z}_{\sss W})^{+}_{\sss 1}(\e_{\sss 5}\xx; u) \\
\end{pmatrix}=V\ov{\tZ_{\sss W}}(\xx;u), \;\; \text{where } 
V= \begin{pmatrix} 1 & 0 & 1 \\
0 & 1 & 0\\
1& 0 & \!\!\!-1
\end{pmatrix} 
\]
namely $\mathbf{\tilde{Z}}_{\sss 1}\!(\mathbf{x}; u) =
\mathrm{M}_{w}(\mathbf{x}; u)
\mathbf{\tilde{Z}}_{\sss 1}\!(w\mathbf{x}; u)$ for all $w\in W,$ 
for the $1$-cocycle $\mathrm{M}_{w}(\mathbf{x}; u)$ defined on
generators by
$
\mathrm{M}_{\ss_{\scalebox{0.75}{$\scriptscriptstyle i$}}}\!(\xx;u)
=-\frac{1}{x_{\scalebox{0.85}{$\scriptscriptstyle i$}}^{\scalebox{0.85}{$\scriptscriptstyle 2$}}}
V \Lam_{\ss_{\scalebox{0.75}{$\scriptscriptstyle i$}}}\!(\xx;u) V^{\sss -1}$. 
\end{rem}

\begin{rem}\label{general-commentx} \!Expressions similar to \eqref{eq: MDS-vers1}, \eqref{eq: MDS-vers2}, and the conclusions of Lemma \ref{Initial-abs-conv-scriptW}, and Proposition \ref{functional-equations-MDS} all still hold if instead of
		$
		\mathscr{W}\!(\mathbf{s}; a_{\scriptscriptstyle 2}, a_{\scriptscriptstyle 1})
		$, 
		one takes the multiple {D}irichlet series constructed in the same way using the coefficients of 
		$Z_{\scriptscriptstyle W}^{\scriptscriptstyle \mathrm{CG}}(\mathbf{x}; u)$. \!More generally, one can 
		start with any power series $F(z, u)$, and take the multiple {D}irichlet series whose $p$-parts are 
		$
		F\big(|p|^{\scalebox{1.1}{$\scriptscriptstyle - \delta(\mathbf{s})$}}, |p|^{\scalebox{1.1}{$\scriptscriptstyle - 1\slash 2$}}\big)
		\tilde{Z}_{\scriptscriptstyle W}\big(|p|^{- s_{\scalebox{.85}{$\scriptscriptstyle 1$}}}, \ldots, |p|^{- s_{\scalebox{.85}{$\scriptscriptstyle 5$}}}; |p|^{\scalebox{1.1}{$\scriptscriptstyle - 1\slash 2$}}\big)
		$, where 
		$\delta(\mathbf{s}) = s_{\scriptscriptstyle 1} + s_{\scriptscriptstyle 2} + s_{\scriptscriptstyle 3}
		+ s_{\scriptscriptstyle 4} + 2s_{\scriptscriptstyle 5}
		$. 
		It is easy to see that this multiple {D}irichlet series is, in fact, 
		\begin{equation}  \label{eq: WMDS-var-p-parts}
			\prod_{p} F\big(|p|^{\scalebox{1.1}{$\scriptscriptstyle - \delta(\mathbf{s})$}}, |p|^{\scalebox{1.1}{$\scriptscriptstyle - 1\slash 2$}}\big) 
			\cdot \mathscr{W}\!(\mathbf{s}; a_{\scriptscriptstyle 2}, a_{\scriptscriptstyle 1}).
		\end{equation}
		Choosing $F(z, u)$ so that the product converges absolutely when $\Re(\delta(\mathbf{s})) >  6$, one sees that the multiple {D}irichlet series can be expressed as in \eqref{eq: MDS-vers1} and \eqref{eq: MDS-vers2} for $\Re(s_{\scriptscriptstyle i}) >  1$, $i = 1, \ldots, 5$, and that it satisfies the required functional equations. Thus to determine a canonical normalization of $\tilde{Z}_{\scriptscriptstyle W}(\mathbf{x}; u) $ (or equivalently $Z_{\scriptscriptstyle W}^{\scriptscriptstyle \mathrm{CG}}(\mathbf{x}; u)$), which eventually allows us to establish the analytic properties of the corresponding multiple {D}irichlet series, some additional conditions must be imposed.
	\end{rem}

There is yet another multiple {D}irichlet series
\[
 Z(\mathbf{x}; q) =
\sum_{\mathrm{k} \in \mathbb{N}^{r \scalebox{.93}{$\scriptscriptstyle + 1$}}}
b(\mathrm{k}; q) \mathbf{x}^{\mathrm{k}}
\]
associated with moments of $L$-functions that is worth considering; as
before, $x_{\scalebox{1.1}{$\scriptscriptstyle i$}}$ stands for
$q^{ - s_{\scalebox{.75}{$\scriptscriptstyle i$}}},$ $i =
1, \ldots,$ $r+1,$ and the coefficients $b(\mathrm{k}; q)$ are finite sums
$ 
\sum_{\lambda}
c_{\scalebox{1.}{$\scriptscriptstyle \lambda$}}\lambda
$
over $q$-Weil algebraic integers of weights
$
\nu_{\scalebox{1.05}{$\scriptscriptstyle \lambda$}} \in \mathbb{N},
$
subject to the following three conditions:
\begin{enumerate}[label=(\alph*)]
\item Each $q$-Weil integer $\lambda$ occurs in the sum together with
  all its complex conjugates. 

\item The coefficients
$
c_{\scalebox{1.}{$\scriptscriptstyle \lambda$}}
$
are rational numbers, and if $\lambda, \lambda'$ are conjugates over
$\mathbb{Q},$ then
$
c_{\scalebox{1.}{$\scriptscriptstyle \lambda$}}
\!= c_{\scalebox{1.}{$\scriptscriptstyle \lambda \scalebox{.75}{$\scriptscriptstyle '$} $}}.
$ 

\item For $|\mathrm{k}| = k_{\scalebox{1.1}{$\scriptscriptstyle 1$}} 
+ \cdots + k_{r \scalebox{1.1}{$\scriptscriptstyle +1$}} > 1,$ we have
that 
$
|\mathrm{k}| + 2
\le  \nu_{\scalebox{1.05}{$\scriptscriptstyle \lambda$}}
\le 2|\mathrm{k}|
$ 
for all $\lambda$ occurring in the sum.   
\end{enumerate} 
The lower bound of $\nu_{\scalebox{1.05}{$\scriptscriptstyle
    \lambda$}}$ in the last condition will be referred to as
{\normalfont\itshape dominance.} Let $\overline{\mathbb{F}}_{\! q}$
be a fixed algebraic closure of $\mathbb{F}_{\! q}.$ We require that
the multiple {D}irichlet series to be of the form \eqref{eq:
  MDS-vers0}, with multiplicative coeffi-\linebreak cients $B(m_{\scriptscriptstyle
  1}, \ldots, m_{\scalebox{1.21}{$\scriptscriptstyle r$}}, d),$ such that the following
conditions are also satisfied:
\begin{enumerate}[label=(A\arabic*)]
	\item The sub-series
	\[
	\sum_{\underline{k}  \in  \mathbb{N}^{r}}
	b(\underline{k}, 0; q)\underline{x}^{\scalebox{1.1}{$\scriptscriptstyle \underline{k}$}} 
	= \prod_{i = 1}^{r} \frac{1}{1 - qx_{\scalebox{1.1}{$\scriptscriptstyle i$}}}
      \]
      where, as before,
$
\underline{x} : = (x_{\scalebox{1.1}{$\scriptscriptstyle 1$}}, \ldots, x_{\scalebox{1.3}{$\scriptscriptstyle r$}})
$
and
$
\underline{k} : = (k_{\scalebox{1.1}{$\scriptscriptstyle 1$}}, \ldots, k_{\scalebox{1.3}{$\scriptscriptstyle r$}}).
$
In addition,
	\[
	\sum_{\underline{k}  \in  \mathbb{N}^{r}}
	b(\underline{k}, 1;
        q)\underline{x}^{\scalebox{1.1}{$\scriptscriptstyle
            \underline{k}$}} = q\;\;\; 
        \mathrm{and}\;\;\, 
	\sum_{l \, \ge \, 0} b(\underline{0}, l; q) x_{r
          \scalebox{1.1}{$\scriptscriptstyle +1$}}^{l}
	\!= \frac{1}{1 - qx_{r  \scalebox{1.1}{$\scriptscriptstyle
              +1$}}}.
      \]
      In particular, $B(1, \ldots, 1) = b(0, \ldots, 0; q) = 1.$

	\item The coefficients $b(\mathrm{k}; q^{n})$ corresponding to
          $Z(\mathbf{x}; q^{n})$ over any finite field extension $\mathbb{F}_{\! q^{n}}  \subset
        \overline{\mathbb{F}}_{\! q}$ of $\mathbb{F}_{\! q}$ are\linebreak given by 
	\[
	b(\mathrm{k}; q^{n}) 
	= \sum_{\lambda} c_{\scalebox{1.}{$\scriptscriptstyle
            \lambda$}} \lambda^{\! n}. 
	\]

       \item For every monic irreducible $p \in \mathbb{F}_{\! q}[x]$
          of degree $e \ge 1,$ the coefficients
          $
          B\big(p^{k_{\scalebox{.75}{$\scriptscriptstyle 1$}}}\!,
        \ldots, p^{k_{\scalebox{.88}{$\scriptscriptstyle r$}\scalebox{.75}{$\scriptscriptstyle + 1$}}} \big)
        $
        are given by\linebreak
        \[
B\big(p^{k_{\scalebox{.75}{$\scriptscriptstyle 1$}}}\!,
        \ldots, p^{k_{\scalebox{.88}{$\scriptscriptstyle r$}\scalebox{.75}{$\scriptscriptstyle + 1$}}} \big)  
	= q^{e |\mathrm{k}|}\!\sum_{\lambda} c_{\scalebox{1.}{$\scriptscriptstyle \lambda$}}\lambda^{\! - e}.
      \]

    \end{enumerate}
    In \cite{DV}, the first two authors established the existence of a
    unique multiple {D}irichlet series satisfying all these
    conditions. This axiomatic construction was generalized by Whitehead~\cite{White2} to simply laced affine root systems, and by Sawin~\cite{Saw}, using geometric methods, in a more general setting.

    One should notice that it is not a priori clear that this multiple
    {D}irichlet series satisfies a group of func-\linebreak tional equations.
    To see that this is indeed the case, one can adapt the proof of
    \cite[Proposition~2.2.1]{White2}. However, in the next section,
    this will be clarified when $r = 4.$

\section{Comparison} \label{comparison} Let notations
be as above. Our goal in this section is to compare the functions
$\mathscr{Z}\!(\mathbf{x})$ and $\tilde{Z}_{\scriptscriptstyle
  W}(\mathbf{x}; \sqrt{q}),$ and\linebreak
deduce from this comparison (combined
with Theorem \ref{Divisibility-and-analytic-continuation}) the
meromorphic continuation of $\mathscr{Z}\!(\mathbf{x})$ to
the\linebreak maximal possible region
$|\mathbf{x}^{\scalebox{1.1}{$\scriptscriptstyle \delta$}}| < 1.$

One checks that the functions $\sZ\!(\mathbf{x})$ and 
$\tZW(\mathbf{x}; \sqrt{q})$ satisfy the following conditions:
\begin{enumerate}[label=(\roman*)] 
 \item Both $\sZ\!(\mathbf{x})$ and $\tZW(\mathbf{x}; \sqrt{q})$ are
   holomorphic for $|x_{\scalebox{1.1}{$\scriptscriptstyle i$}}|
   < 1\slash q$ ($i = 1, \ldots, 5$).

\item Both
   $
   D(\mathbf{x}; \sqrt{q})\sZ\!(\mathbf{x})
   $
   and
   $
   D(\mathbf{x}; \sqrt{q})\tZW(\mathbf{x}; \sqrt{q})$ are power series
   that converge absolutely if either\linebreak
   $\underline{x} \in \mathbb{C}^{\scriptscriptstyle 4}$ and
   $|x_{\scalebox{1.1}{$\scriptscriptstyle 5$}}|$ is sufficiently small, or
   $x_{\scalebox{1.1}{$\scriptscriptstyle 5$}}\in \mathbb{C}$ and
   $|x_{\scalebox{1.1}{$\scriptscriptstyle 1$}}|, \ldots,
   |x_{\scalebox{1.1}{$\scriptscriptstyle 4$}}|$
   are sufficiently small.

\item Both $\sZ\!(\mathbf{x})$ and $\tZW(\mathbf{x}; \sqrt{q})$ 
are symmetric in
  $
  x_{\scalebox{1.1}{$\scriptscriptstyle 1$}}, \ldots,
  x_{\scalebox{1.1}{$\scriptscriptstyle 4$}},
  $
  and satisfy the same functional equation\linebreak \eqref{fun-eq-tZ}.
\end{enumerate}
It follows from \cite[Theorem~3.7]{BD} that
\begin{equation} \label{eq: uniqueness-thm}
\sZ\!(\mathbf{x}) 
= C(\mathbf{x}^{\scalebox{1.1}{$\scriptscriptstyle \delta$}})
\tZW(\mathbf{x}; \sqrt{q})
\end{equation}
for some function $C$ of one complex variable. 
\vskip5pt
\begin{thm} \label{main-thm} --- We have the equalities 
	\[
          \mathscr{Z}\!(\mathbf{x}) =
          \tilde{Z}_{\scriptscriptstyle W}(\mathbf{x}; \sqrt{q})
          = Z(q^{\scalebox{.95}{$\scriptscriptstyle -1\slash 2$}}\mathbf{x}; q).
	\]
\end{thm} 

\begin{proof} We first show that $\mathscr{Z}\!(\mathbf{x})$ and
  $\tilde{Z}_{\scriptscriptstyle W}(\mathbf{x}; \sqrt{q})$ have the same
  residue at $x_{\scalebox{1.1}{$\scriptscriptstyle 5$}} = 1\slash
  \sqrt{q}$ (hence
  $
  C(\mathbf{x}^{\scalebox{1.1}{$\scriptscriptstyle \delta$}})
  $
  in \eqref{eq: uniqueness-thm}\linebreak is identically $1$).

  From the expression \eqref{eq: MDS-vers2} of 
$
\mathscr{W}\!(\mathbf{s}; a_{\scriptscriptstyle 2}, a_{\scriptscriptstyle 1}),
$ 
we see that this function has a (simple) pole at
$s_{\scriptscriptstyle 5} = \frac{1}{2}$ only if 
$a_{\scriptscriptstyle 2} = 1,$ and the part contributing to this pole is 
\[
\zeta\left(s_{\scalebox{1.}{$\scriptscriptstyle 5$}} + \tfrac{1}{2}\right) 
\;\;\cdot\sum_{
m_{\scalebox{.75}{$\scriptscriptstyle 1$}} 
m_{\scalebox{.75}{$\scriptscriptstyle 2$}}
m_{\scalebox{.75}{$\scriptscriptstyle 3$}}	
m_{\scalebox{.75}{$\scriptscriptstyle 4$}}
\, = \, \square} \;
\frac{Q_{\underline{m}}(s_{\scalebox{1.}{$\scriptscriptstyle 5$}}; 1)}
{|m_{\scriptscriptstyle 1}|^{s_{\scalebox{.75}{$\scriptscriptstyle 1$}}} 
	|m_{\scriptscriptstyle 2}|^{s_{\scalebox{.75}{$\scriptscriptstyle 2$}}} 
	|m_{\scriptscriptstyle 3}|^{s_{\scalebox{.75}{$\scriptscriptstyle 3$}}}
	|m_{\scriptscriptstyle 4}|^{s_{\scalebox{.75}{$\scriptscriptstyle 4$}}}} 
\,= \, \zeta\left(s_{\scalebox{1.}{$\scriptscriptstyle 5$}} + \tfrac{1}{2}\right) 
\, \cdot\, \prod_{p} \bigg(\sum_{|\underline{k}| \, \equiv \, 0 \!\!\!\pmod 2} 
\, \frac{Q_{\underline{k}}
  \!\left(|p|^{- s_{\scalebox{.75}{$\scriptscriptstyle 5$}}};
    |p|^{\scriptscriptstyle -1 \slash 2}\right)}
{|p|^{k_{\scalebox{.75}{$\scriptscriptstyle 1$}}
    s_{\scalebox{.75}{$\scriptscriptstyle 1$}}
    +k_{\scalebox{.75}{$\scriptscriptstyle 2$}}
    s_{\scalebox{.75}{$\scriptscriptstyle 2$}}
    +k_{\scalebox{.75}{$\scriptscriptstyle 3$}}
    s_{\scalebox{.75}{$\scriptscriptstyle 3$}}
    +k_{\scalebox{.75}{$\scriptscriptstyle 4$}}
    s_{\scalebox{.75}{$\scriptscriptstyle 4$}}}}\bigg).
\]
Note that the local factor of the product in the right-hand side is 
\[
\sum_{|\underline{k}| \, \equiv \, 0 \!\!\!\!\pmod 2}
Q_{\scalebox{1.1}{$\scriptscriptstyle \underline{k}$}}
(x_{\scalebox{1.1}{$\scriptscriptstyle 5$}}; u)
\underline{x}^{\scalebox{1.1}{$\scriptscriptstyle \underline{k}$}}
\]
where, for a monic irreducible $p,$ we set 
$
x_{\scalebox{1.1}{$\scriptscriptstyle i$}} \! =  |p|^{ - s_{\scalebox{.75}{$\scriptscriptstyle i$}}}
$ 
($i=1, \ldots, 5$), and $u = |p|^{\scriptscriptstyle -1 \slash 2}.$
By \eqref{eq: holomorphicity-residue-u}, this local
factor evaluated at
$
x_{\scalebox{1.1}{$\scriptscriptstyle 5$}} = u
$
is just $R(u\underline{x};u),$ and thus
\[
\lim_{s_{\scalebox{.75}{$\scriptscriptstyle 5$}} \to \frac{1}{2}}
	\!\big(1 -  q^{\scalebox{.9}{$\scriptscriptstyle \frac{1}{2}$} - s_{\scalebox{.75}{$\scriptscriptstyle 5$}}}\big)\mathscr{W}\!(\mathbf{s};  1, a_{\scriptscriptstyle 1})
	= \prod_{p} R\big(
	|p|^{\scalebox{.9}{$\scriptscriptstyle - \frac{1}{2}$} - s_{\scalebox{.75}{$\scriptscriptstyle 1$}}}\!,
	|p|^{\scalebox{.9}{$\scriptscriptstyle - \frac{1}{2}$} - s_{\scalebox{.75}{$\scriptscriptstyle 2$}}}\!,
	|p|^{\scalebox{.9}{$\scriptscriptstyle - \frac{1}{2}$} - s_{\scalebox{.75}{$\scriptscriptstyle 3$}}}\!,
	|p|^{\scalebox{.9}{$\scriptscriptstyle - \frac{1}{2}$} -
          s_{\scalebox{.75}{$\scriptscriptstyle 4$}}};
        |p|^{\scalebox{.9}{$\scriptscriptstyle - \frac{1}{2}$}}\big).
\] 
By using the formula of $R(\underline{x};u)$ in Theorem \ref{Residue-tilde-Z-avg},
the product equals
$
 R(q^{ - s_{\scalebox{.75}{$\scriptscriptstyle 1$}}}\!,
q^{ - s_{\scalebox{.75}{$\scriptscriptstyle 2$}}}\!,
q^{ - s_{\scalebox{.75}{$\scriptscriptstyle 3$}}}\!,
q^{ - s_{\scalebox{.75}{$\scriptscriptstyle 4$}}};\sqrt{q}),
$ 
that is, $\mathscr{Z}\!(\mathbf{x})$ and
$
\tilde{Z}_{\scriptscriptstyle W}(\mathbf{x}; \sqrt{q})
$
have the same residue at $x_{\scalebox{1.1}{$\scriptscriptstyle 5$}} = 1\slash
\sqrt{q}.$ This shows that
\[
\mathscr{Z}\!(\mathbf{x}) =
\tilde{Z}_{\scriptscriptstyle W}(\mathbf{x}; \sqrt{q}).
\]

To prove the second equality, set $u=\sqrt{q}$. It is clear that
$\tilde{Z}_{\scriptscriptstyle  W}(u\mathbf{x};u)$ satisfies 
the conditions (a) and (b) in\linebreak the
previous section. The upper bound in condition (c) follows
easily from \eqref{eq: 1-acted-by-w} and the fact that the
coeffi-\linebreak cients of the power series expansion of
$J(x\slash u,u,\varepsilon)$ are polynomials in
$u^{\scalebox{1.1}{$\scriptscriptstyle -1$}}$.
% $q^{\scalebox{1.1}{$\scriptscriptstyle -1\slash 2$}}.$

To show that the coefficients of
$
\tilde{Z}_{\scriptscriptstyle W}(u\mathbf{x};u)
$
satisfy the dominance condition, we have to show that
the coef-\linebreak ficients $a(\mathrm{k};u)$
in~\eqref{30} are divisible by
$u^{\scalebox{1.1}{$\scriptscriptstyle 2$}}$ if $|\mathrm{k}|>1$.
This follows from Corollary~\ref{C5.4}.

The first and third conditions in (A1) can be verified using
Lemma \ref{initial-properties}; the generating function of
$b(\underline{k}, 1; q)$\linebreak can be computed from the term
$1\vert \ss_{\sss 5}(\xx;u)$ in $\ZW(\xx;u)$.
Conditions (A2) and (A3) can be easily verified directly.
This completes the proof.
% using the relation 
%$B(\underline{m},d)=
%|m_{\sss 1}m_{\sss 2} m_{\sss 3}m_{\sss 4}d|^{\sss 1\slash 2} A(\underline{m},d)$, 
%with $A(\underline{m},d)$ as in~\eqref{eq: MDS-vers0}, 
%$\underline{m}=(m_{\sss 1},m_{\sss 2},m_{\sss 3},m_{\sss 4})$. 
\end{proof} 

\begin{rem}\label{general-commentxx} \!The fact that
		$ 
		\mathscr{Z}\!(\mathbf{x})
		$ 
		coincides with 
		$
		Z(q^{\scalebox{.95}{$\scriptscriptstyle -1\slash 2$}}\mathbf{x}; q)
		$ 
		shows that our choice of the $p$-part~$\wZ_{\sss W}$ for the {W}eyl group multiple {D}irichlet series associated with the $4$-th moment of quadratic $L$-functions is canonical; this comparison was the sole reason for introducing the axiomatic multiple {D}irichlet series at the end of the previous section.
		The uniqueness of $Z(\mathbf{x}; q)$ implies, for example, that the multiple {D}irichlet series with all the $p$-parts equal to 
		$
		Z_{\scriptscriptstyle W}^{\scriptscriptstyle \mathrm{CG}}\big(|p|^{- s_{\scalebox{.85}{$\scriptscriptstyle 1$}}}, \ldots, |p|^{- s_{\scalebox{.85}{$\scriptscriptstyle 5$}}}; |p|^{\scalebox{1.1}{$\scriptscriptstyle - 1\slash 2$}}\big) 
		$ 
		cannot satisfy the full set of conditions that $\mathscr{Z}\!(\mathbf{x})$ satisfies. In fact, the reader can check directly using 
		\eqref{eq: WMDS-var-p-parts} in Remark \ref{general-commentx} that condition (A3)
		(i.e., the local-to-global principle)
		is {\normalfont\itshape not} satisfied for this choice of the $p$-parts.
	\end{rem}

\section{Applications} %\label{applications}
In this final section of the paper, we begin with two straightforward
consequences of Theorem \ref{main-thm}. First, by taking
coefficients of the function
$
\tilde{Z}_{\scriptscriptstyle W}(\mathbf{x}; \sqrt{q})
$
with respect to the variable
$x_{\scalebox{1.1}{$\scriptscriptstyle 5$}}$, we obtain an exact
for-\linebreak mula for the 4-th moments of quadratic {D}irichlet
$L$-functions over rational function fields, \!weighted by\linebreak the
polynomials $P_{\scalebox{1.1}{$\scriptscriptstyle d$}}$ defined
in Section~\ref{MDS} and Lemma~\ref{initial-properties}. \!Then,
from the analytic properties of
$
\tilde{Z}_{\scriptscriptstyle W}(\mathbf{x}; \sqrt{q})
$,\linebreak
we also deduce an asymptotic formula for these moments,
\!completely analogous to that conjectured in
\cite{DT}\linebreak for arbitrary moments. For the remaining
of the section, we study in some detail the secondary terms in
the asymptotic formula. In particular, our analysis will show
that all these terms are non-zero.

\vskip5pt
\begin{thm} %\label{exact-formula} 
--- For $D \ge 1,$ we have the exact formula:
 \[ 
\sum_{\deg d \, = D} L\!\left(\tfrac{1}{2}, \chi_{d_{\scalebox{.8}{$\scriptscriptstyle 0$}}}\right)^{4} 
\!P_{\scalebox{1.1}{$\scriptscriptstyle d$}}(\chi_{d_{\scalebox{.8}{$\scriptscriptstyle 0$}}})
= \mathrm{Coeff}_{\xi^{\scalebox{1.}{$\scriptscriptstyle D$}}}
\tilde{Z}_{\scriptscriptstyle W} (\underline{1}, \xi; \sqrt{q})
\]
where
$
P_{\scalebox{1.1}{$\scriptscriptstyle d$}}(\chi_{d_{\scalebox{.8}{$\scriptscriptstyle 0$}}})
= P_{\scalebox{1.1}{$\scriptscriptstyle d$}}(0, \ldots, 0; \chi_{d_{\scalebox{.8}{$\scriptscriptstyle 0$}}}),$
and $(\underline{1},\xi): = (1, 1, 1, 1,\xi).$
\end{thm}

\begin{proof} If we put
  $
  \xi=q^{ - s_{\scalebox{.75}{$\scriptscriptstyle 5$}}}
  $,
  then by \eqref{eq: MDS-vers1} we can write
  \[
\mathscr{W}\!(\mathbf{s}; 1, 1\!) 
\, =\,
\sum_{D \, \ge \, 0} \left\{\sum_{\deg d \, = D} 
  \; \prod_{i = 1}^{4} L\!\left(s_{\scriptscriptstyle i} + \tfrac{1}{2},
    \chi_{d_{\scalebox{.8}{$\scriptscriptstyle 0$}}}\right) 
P_{\scalebox{1.1}{$\scriptscriptstyle d$}}(\mathbf{s}\scalebox{1.}{$\scriptscriptstyle '$};
\chi_{d_{\scalebox{.8}{$\scriptscriptstyle 0$}}}) \right\}\xi^{D}.
\]
Setting $s_{\scriptscriptstyle i} = 0,$ the asserted formula follows
at once from the equality 
$
\mathscr{Z}\!(\underline{1}, \xi)
= \tilde{Z}_{\scriptscriptstyle W} (\underline{1}, \xi; \sqrt{q}).
$
\end{proof}

However, one cannot expect to have an analogue of this result for
general function fields, or number fields.\linebreak
For this reason, we shall also give the asymptotic formula for the
fourth moment sums that is, indeed, ex-\linebreak pected to generalize to any
global field. This asymptotic formula has also the advantage of
separating the\linebreak contributions corresponding to the
singularities of the function $\mathscr{Z}\!(\underline{1}, \xi).$

For $n \ge 1,$ let, as in \cite[Section 6]{DT},
\[
\Phi_{n}
= \left\{\sum_{i = 1}^{5}
n_{\scriptscriptstyle i}\alpha_{\scriptscriptstyle i} \in
\Phi_{\scalebox{1.2}{$\scriptscriptstyle \mathrm{re}$}}^{\scalebox{1.2}{$\scriptscriptstyle +$}}:
n_{\scriptscriptstyle 5} = n 
\right\}
\]
and, for $D \ge 1,$ define $Q_{n}(D, q)$ by
\[
Q_{n}(D, q) 
= \lim_{\underline{x} \, \to \, \underline{1}}\bigg(\sum_{\alpha \, \in \, \Phi_{n}}
\,\sum_{\zeta^{2n} = 1}
R_{\alpha,\,  \scalebox{1.1}{$\scriptscriptstyle \zeta$}}(\underline{x}; \sqrt{q})
\zeta^{\scalebox{1.1}{$\scriptscriptstyle D$}}
\mathbf{x}^{\scalebox{1.1}{$\scriptscriptstyle D$}
  \alpha \scalebox{.9}{$\scriptscriptstyle '$}\slash n}\bigg)
\]
where $\alpha' = \alpha - n \alpha_{\scriptscriptstyle 5},$ and
$
R_{\alpha,\,  \scalebox{1.1}{$\scriptscriptstyle \zeta$}}(\underline{x}; \sqrt{q})
$
is given by \eqref{eq: residue-general}. We have the following:

\vskip5pt
\begin{thm} \label{asymptotic} --- For $D, N \ge 1$ and
  $
  (N + 1)^{\scalebox{1.1}{$\scriptscriptstyle -1$}} < \Theta
  < N^{\scalebox{1.1}{$\scriptscriptstyle -1$}}\!,
  $
  we have the asymptotic formula
  \[ 
   \sum_{\deg d \, = D}
   L\!\left(\tfrac{1}{2}, \chi_{d_{\scalebox{.8}{$\scriptscriptstyle 0$}}}\right)^{4} 
\!P_{\scalebox{1.1}{$\scriptscriptstyle d$}}(\chi_{d_{\scalebox{.8}{$\scriptscriptstyle 0$}}})
\, = \sum_{n \, \le \, N} Q_{n}(D, q) q^{\scalebox{.8}{$\scriptscriptstyle \frac{D}{2n}$}}
\, + \, O_{\scalebox{1.1}{$\scriptscriptstyle \Theta$}, \, q}
\left(q^{\scalebox{.8}{$\scriptscriptstyle \frac{D\Theta}{2}$}}\right).
\]
\end{thm}

\begin{proof} The asymptotic formula follows from Theorem
  \ref{main-thm} and a straightforward application of the residue
  theorem to the integral
  \[
  \frac{1}{2 \pi i}\oint_{\partial \mathscr{A}_{\scalebox{1.}{$\scriptscriptstyle \Theta$}}} 
  \!\frac{\mathscr{Z}\!(\underline{1},
    \xi)}{\xi^{\scalebox{1.}{$\scriptscriptstyle D$} \scalebox{1.1}{$\scriptscriptstyle + 1$}}}\, d\xi
\]
where
$
\mathscr{A}_{\scalebox{1.}{$\scriptscriptstyle \Theta$}} =
\{\xi \in \mathbb{C} : q^{\scalebox{1.1}{$\scriptscriptstyle - 2$}} \le |\xi| \le 
q^{\scalebox{1.1}{$\scriptscriptstyle - \Theta\slash 2$}}\}
$ --- see also \cite[Theorem 6.1]{DT}.
\end{proof}

\begin{rem} %\label{Normalization-Asymptotic} 
One might be puzzled by the discrepancy of a factor of
$
q^{\scalebox{1.}{$\scriptscriptstyle D\slash 2$}}
$
in this asymptotic formula. \!This is simply explained by our
normalization
$
\mathbf{x} \to q^{\scalebox{1.}{$\scriptscriptstyle -1\slash 2$}}\mathbf{x}
$
of the function
$
\tilde{Z}_{\scriptscriptstyle W}(\mathbf{x}; \sqrt{q})
$
which is causing the cor-\linebreak rection polynomials
$
P_{\scalebox{1.1}{$\scriptscriptstyle
    d$}}(\chi_{d_{\scalebox{.8}{$\scriptscriptstyle 0$}}})
$
to be off by a factor of $|d|^{\scalebox{1.}{$\scriptscriptstyle 1\slash 2$}}.$
For instance, when
$
d = d_{\scalebox{1.1}{$\scriptscriptstyle 0$}}
$
is square-free, then\linebreak
\[
P_{\scalebox{1.1}{$\scriptscriptstyle
    d$}}(\chi_{d_{\scalebox{.8}{$\scriptscriptstyle 0$}}})
=  |d|^{\scalebox{1.1}{$\scriptscriptstyle - 1\slash 2$}}
\]
instead of $1.$ 
\end{rem}

Notice also that Corollary~\ref{C5.5} and definition \eqref{eq: polyPd}
of the {D}irichlet polynomials
$
P_{\scalebox{1.1}{$\scriptscriptstyle d$}}(s_{\scriptscriptstyle 1}, \ldots, 
s_{\scriptscriptstyle 4}; \chi_{a_{\scalebox{.8}{$\scriptscriptstyle 1$}} 
  d_{\scalebox{.8}{$\scriptscriptstyle 0$}}})
$
imply that
$
P_{\scalebox{1.1}{$\scriptscriptstyle d$}}(\chi_{d_{\scalebox{.8}{$\scriptscriptstyle 0$}}})
$
is non-negative for all $d$. In particular, we have the inequality
\[
  \sum_{\substack{d - \mathrm{monic \; \& \; sq. \, free} \\ \deg d \,= D}}
L\!\left(\tfrac{1}{2},
  \chi_{d}\right)^{\scalebox{1.1}{$\scriptscriptstyle 4$}}
|d|^{\scalebox{1.1}{$\scriptscriptstyle - 1\slash 2$}}
\,\le \,
\sum_{\substack{d - \mathrm{monic}\\\deg d \, = D}}
L\!\left(\tfrac{1}{2}, \chi_{d_{\scalebox{.8}{$\scriptscriptstyle
        0$}}}\right)^{\scalebox{1.1}{$\scriptscriptstyle 4$}} 
\!P_{\scalebox{1.1}{$\scriptscriptstyle d$}}(\chi_{d_{\scalebox{.8}{$\scriptscriptstyle 0$}}}).
\]
Moreover, \cite[Conjecture 1.2]{DT} predicts an asymptotic formula
for the (traditional) moment sum in the left-\linebreak hand side
of the above inequality, similar to that in Theorem \ref{asymptotic}.
(This similarity between the two asymp-\linebreak totics should still persist
when considering the analogues of these moment sums over any global field.) 
For these reasons, the presence of the correction factors
$
P_{\scalebox{1.1}{$\scriptscriptstyle d$}}(\chi_{d_{\scalebox{.8}{$\scriptscriptstyle 0$}}})
$
in our fourth moment sum is harmless for all practical purposes.
% of course we are thinking here of the analogue of the asymptotic formula
% in in the\linebreak number field setting.
That is, if an analogue of the asymptotic formula in Theorem
\ref{asymptotic} is proved in the\linebreak number field setting, it would have
the same applications as the corresponding asymptotic formula for
the\linebreak traditional fourth moment sum.

\subsection{An explicit formula for $Q_{n}(D, q)$}
We now give explicit formulas for the terms $Q_n(D,q)$  
in Theorem~\ref{asymptotic}, following closely~\cite{DT}.
In \emph{loc. cit.},\linebreak
conjectural formulas were given for the first two 
terms $Q_{\sss 1}$ and $Q_{\sss 2}$ in the asymptotics of  
the $r$-th moment of\linebreak quadratic {D}irichlet $L$-functions,
summed over square-free monic polynomials. It is 
interesting to note that\linebreak we obtain the same formulas, except 
for the so-called ``arithmetic factor'' which is simpler when summing
over all monic polynomials. Here, we are able to treat all the terms
$Q_{n}(D, q)$ since the sets $\Phi_{n}$ 
can be ex-\linebreak plicitly described
for $D_{\scriptscriptstyle 4}^{\scriptscriptstyle (1)}$.

First, we rewrite the double sum in the formula of $Q_n(D,u^2)$
by grouping together the terms with $\pm\zeta$. De-\linebreak note by 
$f^{\ec}=f^+_{\sss 1}$ (resp. \!$f^{\oc}=f^-_{\sss 1}$) the even
(resp. \!odd) part of the function $f(\xx)$ with respect to the sign
func-\linebreak tion $\e_{\sss 1}$, and let $\mu_{\sss k}$ be the set
of $k$-th roots of unity in $\CC$. We have
\[
Q_{n}(D, u^{\sss 2})=\frac{1}{n}
\sum_{\zeta \in \mu_{\scalebox{.75}{$\scriptscriptstyle 2n$}}
\slash \{\pm 1\}}\zeta^{\scalebox{1.1}{$\scriptscriptstyle D$}} 
\!I_{n,\, \zeta}(D,u) 
\]
with
\be\label{eq40}
I_{n,\, \zeta}(D,u)
=\lim_{\underline{x} \, \to \, \underline{1}}
\sum_{\alpha \, \in \, \Phi_{n}}
\mathbf{x}^{\scalebox{1.1}{$\scriptscriptstyle D$}
\alpha \scalebox{.9}{$\scriptscriptstyle '$}\slash n}
\!\!\left.\left\{R(w_{\scalebox{1.2}{$\scriptscriptstyle \a$}}\xx;u)
    f_{w_{\scalebox{.9}{$\scriptscriptstyle \a$}}}^{a_{\scalebox{.70}{$\scriptscriptstyle D$}}}\!(\xx;u)\right\}
  \right\vert_{x_{\scalebox{.75}{$\scriptscriptstyle 5$}}\,=\, 
  \zeta^{\scalebox{1.0}{$\scriptscriptstyle -1$}}(u
 \mathbf{x}^{\alpha \scalebox{.75}{$\scriptscriptstyle'$}})^{\scalebox{1.0}{$\scriptscriptstyle -1$}\slash n}}
\ee
where $f_{w}=(1+ux_{\sss 5})\| w$ was defined in Lemma~\ref{L6.2}, and
$a_{\sss D}\in \{\mathcal{e},\mathcal{o}\}$ denotes the parity of~$D$.
Although the function $R$ does not depend on $x_{\sss 5}$, we write
$R(\xx;u)$ for $R(\ux;u)$, hence $R(w\ux;u)=R(w\xx;u)$.
Notice that $R(w\xx;u)$ is even with respect to both
$\e_{\sss 1}$ and $\e_{\sss 5}$ for all $w\in W$,
and that
$
I_{n,-\zeta}(D, u)=(-1)^{\scalebox{1.0}{$\scriptscriptstyle D$}}
I_{n,\, \zeta}(D,u)
$,
so the sum over $\zeta$ above is indeed well-defined.

Next, we give integral formulas for the sum in~\eqref{eq40},
from which it will be clear that the limit exists,
and it is\linebreak a polynomial in $D$ of degree 10 if $n$ is odd,
and of degree 7 if $n$ is even.

\subsubsection{The case $n$ odd}
Let $n=2k+1$ with $k\ge 0$.
For $I\subset S:=\{1,2,3,4\},$ possibly empty, put
$w_{\scalebox{.9}{$\scriptscriptstyle I$}}=\prod_{\sss i\in I} \ss_{\sss i},$ and
$\a_{\scalebox{.9}{$\scriptscriptstyle I$}}=\sum_{i\in I} \a_{\sss i},$
with the understanding that the empty product is the identity,
and the empty sum is 0. Also, let
$
t=\ss_{\sss 1}\cdots\,\ss_{\sss 4}\ss_{\sss 5}
$,
for which $t\a_{\sss 5} = \a_{\sss 5}-\dd$. Then
$$
\Phi_{n}=\{\a_{\sss 5}+\a_{\scalebox{.9}{$\scriptscriptstyle I$}}+k\dd : I \subset S\}
\;\;\;\, \text{and} \;\;\;\,
t^{k}w_{\scalebox{.9}{$\scriptscriptstyle I$}}(\a_{\sss 5}
+\a_{\scalebox{.9}{$\scriptscriptstyle I$}}+k \dd )=\a_{\sss 5}.
$$
Notice that $t^2$ is a translation, with
$t^2\a_{\sss i}=\a_{\sss i}+\dd$ ($i=1,\ldots, 4$),
and $t^2\a_{\sss 5}=\a_{\sss 5}-2\dd$.
\vskip5pt
\begin{lem}\label{lem_odd} ---
  Let $n=2k+1$ with $k\ge 0,$ and let $\a=\a_{\sss 5}+k\dd$
  and $w_{\scalebox{1.2}{$\scriptscriptstyle \a$}}=t^{k}$. \!Then we
  have a decomposition 
  \[
    R(w_{\scalebox{1.2}{$\scriptscriptstyle \a$}}\xx;u)
    \vert_{x_{\scalebox{.75}{$\scriptscriptstyle 5$}}\,=\, 
  \zeta^{\scalebox{1.0}{$\scriptscriptstyle -1$}}(u
 \mathbf{x}^{\alpha\scalebox{.75}{$\scriptscriptstyle'$}})^{\scalebox{1.0}{$\scriptscriptstyle -1$}\slash n}} =
R_{n,\, \zeta}(\ux;u) \, \cdot \prod_{1\le i\le j\le 4}
\frac{1}{1-x_{\sss i} x_{\sss j}} 
\]
with
$
R_{n,\, \zeta}(\underline{1};u)=R_{n}(u^{-2\slash n}\slash \zeta^{2})
$
for an explicit function $R_{n}(\varrho)$. Moreover, the function
  $R_{n}(\vr)$ is given by an\linebreak absolutely convergent power
  series for $|\vr|<1;$ when $n=1,$ we have 
\[
R_{\sss 1}(\vr)=(\vr;\vr)_{\infty}^{\sss -11}. 
\]
\end{lem}
Due to Remark~\ref{r9.3} below, we omit the formula
for~$R_{n}(\vr)$ when $n>1.$
\begin{proof}
The assertion follows at once from the explicit formula of
  $R(\ux;u)$ in Theorem~\ref{Residue-tilde-Z-avg}. 
\end{proof}
Recall that $f_{w}=(1+ux_{\sss 5})\|w$. For $b\in \{\mathcal{e}, \mathcal{o}\}$, let 
$
f_{n,\, \zeta}^{b}(\underline{x};u)
=f_{w_{\scalebox{.9}{$\scriptscriptstyle \a$}}}^{b}\!(\xx;u)\vert_{x_{\scalebox{.75}{$\scriptscriptstyle 5$}}\,=\, 
  \zeta^{\scalebox{1.0}{$\scriptscriptstyle -1$}}(u
 \mathbf{x}^{\alpha\scalebox{.75}{$\scriptscriptstyle'$}})^{\scalebox{1.0}{$\scriptscriptstyle -1$}\slash n}}
$
for $\a$ and $w_{\scalebox{1.2}{$\scriptscriptstyle \a$}}$
as\linebreak in the previous lemma. Since $\a$ and $w_{\a}$
are fixed, we omit them from the notation
$R_{n,\,\zeta}$ and $f_{n,\,\zeta}^{b}$.
\vskip5pt
\begin{prop}\label{p9.4} --- The limit~\eqref{eq40} is given by the
  integral
  \[
I_{n,\, \zeta}(D,u)=\frac{1}{4!}\frac{1}{(2\pi i)^{^{\! 4}}}\!\oint \cdots \oint
	h_{n,\, \zeta,\, D}(\underline{z};u)
	\frac{\prod_{1\le i < j \le 4}(z_{\scriptscriptstyle j}^{} -
          z_{\scriptscriptstyle i}^{})^{^{\! 2}}
		\!(1 - z_{\scriptscriptstyle i}^{} z_{\scriptscriptstyle j}^{})}
	{\prod_{i=1}^{4} (1 - z_{\scriptscriptstyle i})^{^{\! 8}}}
\frac{d z_{\scriptscriptstyle 1}}{z_{\scriptscriptstyle 1}^{\scriptscriptstyle 4}} \cdots 
	\frac{d z_{\scriptscriptstyle 4}}{z_{\scriptscriptstyle 4}^{\scriptscriptstyle 4}}
      \]
      where
      \[
          h_{n,\, \zeta, \, D}(\underline{z};u)=
\frac{R_{n,\, \zeta}(\underline{z};u)}{\prod_{\sss i=1}^{\sss 4} z_{\sss i}^{\sss D\slash (2n)}} \cdot 
\begin{cases}
    f_{n,\, \zeta}^{\ec}(\underline{z};u)
    \prod_{i=1}^{4}\dfrac{1-u}{1-u\slash z_{\sss i}} & \text{if $D$ even} \\ 
    f_{n,\, \zeta}^{\oc}(\underline{z};u)\prod_{i=1}^{4}
    z_{\sss i}^{\sss 1\slash 2} & \text{if $D$ odd}.
                 \end{cases}
               \]
               Here each path of integration encloses the point
               $z_{\sss j}=1$, but not the points $z_{\sss j} = 0, u$.
\end{prop}

\begin{proof} Let $\a\in\Phi_{n}$ and
  $
  w_{\scalebox{1.2}{$\scriptscriptstyle \a$}}\in W
  $
  be as in the previous lemma.
  \!By taking $\b=\a+\a_{\scalebox{.9}{$\scriptscriptstyle I$}}\in\Phi_{n}$ in~\eqref{eq40}, with
  $w_{\scalebox{1.2}{$\scriptscriptstyle \b$}}
  =w_{\scalebox{1.2}{$\scriptscriptstyle
      \a$}}w_{\scalebox{.9}{$\scriptscriptstyle I$}}$, we have
  \[
    f_{w_{\scalebox{.9}{$\scriptscriptstyle \a$}}
      w_{\scalebox{.75}{$\scriptscriptstyle I$}}}(\xx)
    =f_{w_{\scalebox{.9}{$\scriptscriptstyle \a$}}}\|w_{\scalebox{.9}{$\scriptscriptstyle I$}} (\xx)=
    \prod_{i\, \in \, I}\frac{1-u\slash x_{\sss i}}{1-ux_{\sss i}}
    f_{w_{\scalebox{.9}{$\scriptscriptstyle \a$}}}^{\ec}\!(w_{\scalebox{.9}{$\scriptscriptstyle I$}}\xx)+\frac{1}{\xx^{\a_{\scalebox{.75}{$\scriptscriptstyle I$}}}} 
f_{w_{\scalebox{.9}{$\scriptscriptstyle \a$}}}^{\oc}\!(w_{\scalebox{.9}{$\scriptscriptstyle I$}}\xx).
\]
Assuming $D$ is even (the case $D$ odd being similar), we can then
write the sum in~\eqref{eq40} as
\[
 \prod_{i=1}^{4} (1-u\slash x_{\sss i})x_{\sss i}^{\sss D\slash 2}
 \sum_{I\subset S} 
  \left.h (w_{\scalebox{.9}{$\scriptscriptstyle I$}}\xx;u)
  \right\vert_{x_{\scalebox{.75}{$\scriptscriptstyle 5$}}\,=\, 
  \zeta^{\scalebox{1.0}{$\scriptscriptstyle -1$}}(u
 \mathbf{x}^{\a_{\scalebox{0.55}{$\sss I$}}+k\alpha_{\scalebox{0.55}{$\sss S$}} })^{\scalebox{1.0}{$\scriptscriptstyle -1$}\slash n}}
\]
where
$
h(\xx;u)= \frac{R(w_{\scalebox{.9}{$\scriptscriptstyle \a$}}\xx;u)
  f_{w_{\scalebox{.9}{$\scriptscriptstyle \a$}}}^{\ec}\!(\xx;u)}{\prod_{\sss i=1}^{\sss 4} 
  (1-u\slash x_{\scalebox{.9}{$\scriptscriptstyle i$}}) x_{\scalebox{.9}{$\scriptscriptstyle i$}}^{\sss D\slash (2n)}}$. Now we use
the following:
\vskip5pt
 \begin{fact}\label{L9.5} --- For any function $g(\xx)$ and a fixed index $1\le j\le 4$, the transformation 
 $x_{\sss j}\mapsto 1\slash x_{\sss j} $ takes 
 \[
   g(\xx)\vert_{x_{\scalebox{.75}{$\scriptscriptstyle
         5$}}\,=\,c\slash x_{\scalebox{.75}{$\scriptscriptstyle j$}}^{\scalebox{.95}{$\scriptscriptstyle a$}}} 
  \xrightarrow{x_{\scalebox{.75}{$\scriptscriptstyle j$}}\mapsto 1\slash x_{\scalebox{.75}{$\scriptscriptstyle j$}}} 
   g(\ss_{\sss j} \xx)\vert_{x_{\scalebox{.75}{$\scriptscriptstyle 5$}}
     \,=\,c\slash x_{\scalebox{.75}{$\scriptscriptstyle j$}}^{\scalebox{.95}{$\scriptscriptstyle b$}}}
 \]
 where $a, b\in\Q$ with $a+b=1$, and $c$ is any function not
 depending on $x_{\sss j}$ and $x_{\sss 5}$. 
 \end{fact} 
Taking for $g(\xx)$ the function $h(\xx;u)$ and applying 
\cite[Lemma 7.1]{DT} (the case $m=0$ of Lemma~\ref{integral} below) 
yields an integral representation for the sum above. 
Now one can take $\ux\to \underline{1}$ inside the integral, 
giving the\linebreak above expression for $I_{n,\, \zeta}(D,u)$. 
\end{proof} 
From the integral representation, a standard argument (e.g., 
\cite[Prop. 7.7]{DT}) leads to the following:
\vskip5pt
\begin{cor} --- When $n$ is odd, $I_{n,\, \zeta}(D,u)$ is a 
  polynomial in $D$ of degree 10, with leading term given by
  \[
D^{10} R_{n,\, \zeta}(\underline{1};u)
f_{n,\,\zeta}^{a_{\scalebox{.70}{$\scriptscriptstyle D$}}}(\underline{1};u)  
\frac{1}{4!(2n)^{^{\! 10}}}\frac{1}{(2\pi i)^{^{\! 4}}}
\!\oint_{|t_{\scalebox{.75}{$\scriptscriptstyle 4$}}|=1} 
\cdots \oint_{|t_{\scalebox{.75}{$\scriptscriptstyle 1$}}|=1} 
\prod_{i=1}^{4}e^{-t_{\scalebox{.75}{$\scriptscriptstyle i$}}} \cdot
\frac{\prod_{1\le i < j \le 4} (t_{\sss i}-t_{\sss j})^{^{\! 2}}
  \!(t_{\sss i}+t_{\sss j} ) }{\prod_{i=1}^4 t_{\sss i}^{\sss 8}} 
dt_{\sss 1}\cdots \, dt_{\sss 4}. 
\]
\end{cor}
The integral (with the factor $(2\pi i)^{-4}$ included) equals
$8\slash 1575$. When $n=1$, following~\cite[Prop. 2.1]{GHRR},
we can express the leading term of $Q_{\sss 1}(D,q)$ as
\[
  D^{10} R_{1}(q^{\sss -1}) 
  \frac{1}{2^{r}} \prod_{j=0}^{r-1} \frac{(2j)!}{(r+j)!}
  \;\,\quad\;\, (r=4).
\]
It agrees with the analogous main term in the conjectural
asymptotic formula of the fourth moment over the rationals
(see \cite{DGH}, \cite{CFKRS}, or \cite[Theorem 1.1]{GHRR}),
apart from the ``arithmetic factor''\!, which in our case
is $R_{1}(\vr) =\qpoc{\vr}{\vr}^{\sss -11}$.

\subsubsection{The case $n$ even}
Let $n=2k$ with $k\ge 1$, and let
$
t=\ss_{\sss 1}\cdots\, \ss_{\sss 5}
$
be as before, with $t\a_{\sss 5}=\a_{\sss 5}-\dd$. Then
\[
\Phi_{n}=\{\pm \a_{\sss i} +k\dd : i\in S\}\;\;\;\, \text{and} \;\;\;\, 
t^{k-1}\ss_{\sss i}t(-\a_{\sss i}+k \dd)=
t^{k-1}\ss_{\sss i} t \ss_{\sss i}(\a_{\sss i}+k \dd )= \a_{\sss 5}.
\]
\vskip5pt
\begin{lem}\label{lem_even} --- Let $n=2k$ with $k\ge 1,$ and let
  $\a=-\a_{\sss 4}+k\dd\in\Phi_{n}$ and
  $
  w_{\scalebox{1.2}{$\scriptscriptstyle \a$}}
  =t^{k-1}\ss_{\sss 4} t$. Then we have a decom-\linebreak position 
\[
    R(w_{\scalebox{1.2}{$\scriptscriptstyle
        \a$}}\xx;u)\vert_{x_{\scalebox{.75}{$\scriptscriptstyle
          5$}}\,=\, \zeta^{\scalebox{1.0}{$\scriptscriptstyle -1$}}(u
 \mathbf{x}^{\alpha\scalebox{.9}{$\scriptscriptstyle '$}
 })^{\scalebox{1.0}{$\scriptscriptstyle -1$}\slash n}} = R_{n,\, \zeta}(\ux;u)
\,\cdot \, \frac{1}{1-x_{\sss 4}^{\sss 2}} 
\prod_{j=1}^{3} \frac{1}{(1-x_{\sss 4}x_{\sss j})(1-x_{\sss 4}\slash x_{\sss j})}
\]
with
$
R_{n,\, \zeta}(\underline{1};u)=
R_{n}(u^{-2\slash n}\slash \zeta^{2})
$
for an explicit function $R_{n}(\varrho)$. In addition, the
function $R_{n}(\vr)$ is given by an absolutely convergent power
series for $|\vr|<1$; when $n=2,$ we have
\[
    R_{2}(\vr)=\qpoc{\vr}{\vr}^{\sss -8}\qpoc{\vr}{\vr^{\sss 2}}^{\sss -6}(1-1\slash \vr)^{\sss -7}. 
  \]
\end{lem}
\begin{proof} The assertion follows at once from the explicit formula
  of $R(\ux;u)$ in Theorem~\ref{Residue-tilde-Z-avg}. 
\end{proof}
For 
$
b\in \{\mathcal{e}, \mathcal{o}\}
$, 
let 
$
f_{n,\, \zeta}^{b}(\underline{x};u)
= f_{w_{\scalebox{.9}{$\scriptscriptstyle \a$}}}^{b}\!(\xx;u)
\vert_{x_{\scalebox{.75}{$\scriptscriptstyle 5$}}\,=\, 
  \zeta^{\scalebox{1.0}{$\scriptscriptstyle -1$}}(u
 \mathbf{x}^{\alpha\scalebox{.75}{$\scriptscriptstyle'$}})^{\scalebox{1.0}{$\scriptscriptstyle
     -1$}\slash n}}
$ 
for $\a$ and $w_{\scalebox{1.2}{$\scriptscriptstyle \a$}}$ as in the
previous lemma. 
\vskip5pt
\begin{prop} \label{p9.8} --- The limit 
$I_{n,\, \zeta}(D,u)$ is given by the integral 
\[
\frac{1}{2^{^{3}}3!}\frac{1}{(2\pi i)^{^{\! 4}}}\!\oint \cdots \oint	
h_{n,\, \zeta,\, D}(\underline{z};u)
\frac{\prod_{1\le i< j \le 4} (z_{\scriptscriptstyle i}^{}- z_{\scriptscriptstyle j}^{})^{\mathrm{e}_{\scalebox{.65}{$\scriptscriptstyle ij$}}}
(1 - z_{\scriptscriptstyle i}^{} z_{\scriptscriptstyle j}^{})
\prod_{1\le k\le l \le 3}(1 - z_{\scriptscriptstyle k}^{}
z_{\scriptscriptstyle l}^{})
\prod_{i=1}^{3} \!z_{\scriptscriptstyle i}^{}}
 {\prod_{i=1}^{4} (1 - z_{\scriptscriptstyle i})^{^{\! 8}}}
\frac{d z_{\scriptscriptstyle 1}}{z_{\scriptscriptstyle 1}^{\scriptscriptstyle 4}} \cdots 
\frac{d z_{\scriptscriptstyle 4}}{z_{\scriptscriptstyle 4}^{\scriptscriptstyle 4}}
\]
where
$
\mathrm{e}_{\scalebox{.85}{$\scriptscriptstyle ij$}} = 1
$
or $2$ according as $j = 4$ or not, and
\[
  h_{n,\,\zeta,\, D}(\underline{z};u)
  =\frac{R_{n,\, \zeta}(\underline{z};u)}{z_{\sss 4}^{\sss D\slash n}}
 \cdot 
\begin{cases}
    f_{n,\,\zeta}^{\ec}(\underline{z};u)
    \prod_{i=1}^{4}\dfrac{1-u}{1-u\slash z_{\sss i}} & \text{if $D$ even} \\ 
        f_{n,\, \zeta}^{\oc}(\underline{z};u)\prod_{i=1}^{4} z_{\sss i}^{\sss 1\slash 2} & \text{if $D$ odd}.
                 \end{cases}
               \]
               The paths of integration are as in Proposition~\ref{p9.4}.
\end{prop}
\begin{proof} Let $\a\in\Phi_{n}$ and 
$w_{\scalebox{1.2}{$\scriptscriptstyle \a$}}\in W$ be as in the
previous lemma, and set
$\a^{+} \!= \ss_{\sss 4}\a\in \Phi_{n}$,
  and
  $
  w_{\scalebox{1.2}{$\scriptscriptstyle \a^{\scalebox{.6}{$\scriptscriptstyle +$}}$}}
  \!=w_{\scalebox{1.2}{$\scriptscriptstyle \a$}}\ss_{\sss 4}$.
  As before, assume that $D$ is even, as the case $D$ odd
  is entirely similar. We have
  $
  f_{w_{\scalebox{.9}{$\scriptscriptstyle \a^{\scalebox{.6}{$\scriptscriptstyle +$}}$}}}
  \! =f_{w_{\scalebox{.9}{$\scriptscriptstyle \a$}}}\|\ss_{\sss 4}
  $,
  and the two terms of the sum in~\eqref{eq40}
  corresponding to $\a$ and $\a^+$ can be written as
\[
  \prod_{i=1}^{4} 
  (1-u\slash x_{\sss i})x_{\sss i}^{\sss D\slash 2}\cdot \left(
\left. h(\xx;u)\right\eval+\left.h(\ss_{\sss 4}\xx;u)\right\vert_{x_{\scalebox{.75}{$\scriptscriptstyle 5$}}\,=\, 
  \zeta^{\scalebox{1.0}{$\scriptscriptstyle -1$}}(u
 \mathbf{x}^{\alpha^{\scalebox{.9}{$\scriptscriptstyle +$}\scalebox{.9}{$\scriptscriptstyle '$}}})^{\scalebox{1.0}{$\scriptscriptstyle -1$}\slash n}}
\right)
\]
where
$
h(\xx;u)
=\frac{R(w_{\scalebox{.9}{$\scriptscriptstyle \a$}}\xx;u)
f_{w_{\scalebox{.9}{$\scriptscriptstyle \a$}}}^{\ec}\!(\xx;u)}
{x_{\sss 4}^{\sss D\slash n}\prod_{\sss i=1}^{\sss 4} (1-u\slash
  x_{\scalebox{.9}{$\scriptscriptstyle i$}})}
$
and $\alpha^{+ '}$ \!denotes the sum of the first four
  components of $\a^{+}$. 
By Fact~\ref{L9.5}, the two terms inside the parenthesis
are interchanged by the transformation
$x_{\sss 4}\mapsto 1\slash x_{\sss 4}$. Moreover,
the function $h(\xx;u)$ is clearly symmetric
in the first three variables, and
we claim that the first term in the parenthesis is also invariant
under the transformations
$
x_{\sss j}\mapsto 1\slash x_{\sss j},
$
for $j=1,2,3$. Indeed, fixing one such\linebreak $j$, by applying again
Fact~\ref{L9.5}, we have that
\[
  \left.h(\xx;u)\right\eval
  \xrightarrow{x_{\scalebox{.75}{$\scriptscriptstyle j$}}\mapsto
    1\slash x_{\scalebox{.75}{$\scriptscriptstyle j$}}} 
\left.h(\ss_{\sss j}\xx;u) \right\eval.
\]
One checks that
$
h(\ss_{\sss j}\xx;u)=\frac{R(w_{\scalebox{.9}{$\scriptscriptstyle \a$}}
  \ss_{\scalebox{.8}{$\scriptscriptstyle j$}}\xx;u)
  f_{w_{\scalebox{.9}{$\scriptscriptstyle \a$}}
    \ss_{\!\scalebox{.55}{$\scriptscriptstyle j$}}}^{\ec}\!(\xx;u)}
{x_{\sss 4}^{\sss D\slash n}\prod_{\sss i=1}^{\sss 4} (1-u\slash
  x_{\scalebox{.9}{$\scriptscriptstyle i$}})}
$.
Since
$
w_{\scalebox{1.2}{$\scriptscriptstyle \a$}}
\ss_{\sss j}\a
=w_{\scalebox{1.2}{$\scriptscriptstyle \a$}}\a
= \a_{\sss 5}
$, 
our claim follows now by Lemma~\ref{L6.2}. 
%$(h(\ss_{\sss j}\xx;u)-h(\xx;u))\eval=0$
%$h(\ss_{\sss j}\xx;u)$ and $h(\xx;u)$
%agree when evaluated at 
%${x_{\scalebox{.9}{$\scriptscriptstyle 5$}}\,=\, \zeta^{\scalebox{1.0}{$\scriptscriptstyle -1$}}(u
% \mathbf{x}^{\alpha\scalebox{.9}{$\scriptscriptstyle '$} })^{\scalebox{1.0}{$\scriptscriptstyle -1$}\slash n}}$, 
%proving the claim.

Using these symmetries and the lemma below with $m=3$,
we can express the sum in~\eqref{eq40} as an integral. Our
assertion follows now by taking $\ux\to 1$ inside the integral. 
\end{proof} 
\vskip5pt 
\begin{lem}\label{integral} --- Let 
	$
	a_{\scalebox{1.1}{$\scriptscriptstyle 1$}}, \ldots, a_{r}
	$ 
	be distinct non-zero complex numbers such that $a_{\scriptscriptstyle i}a_{\scriptscriptstyle j} \ne 1$ for all 
	$1\le i, j \le r.$ Suppose $h$ is a function of $r$ complex variables, holomorphic on a domain containing 
	$
	\Big(a_{\scalebox{1.1}{$\scriptscriptstyle 1$}}^{\delta_{\scalebox{.9}{$\scriptscriptstyle 1$}}}\!, \ldots, a_{r}^{\delta_{r}}\Big)
	$ 
	for all 
	$
	(\delta_{\scalebox{1.1}{$\scriptscriptstyle 1$}}, \ldots, \delta_{r}) \in \{\pm 1 \}^{r}.
	$ 
	For $0 \le m < r,$ define $K_{m}(\mathrm{z})$ by 
	\[
	K_{m}(\mathrm{z}) = \frac{h(\mathrm{z})}
	{\prod_{k = 1}^{m}\prod_{l = m + 1}^{r} (1 - z_{\scriptscriptstyle k}^{} 
		z_{\scriptscriptstyle l}^{})   
		(1 - z_{\scriptscriptstyle k}^{\scriptscriptstyle -1} 
		z_{\scriptscriptstyle l}^{})   
		\prod_{m + 1 \, \le \, k  \, \le \, l \, \le \, r} 
		(1 - z_{\scriptscriptstyle k}^{} 
		z_{\scriptscriptstyle l}^{})}, \quad 
		\mathrm{z}: = (z_{\scalebox{1.1}{$\scriptscriptstyle 1$}}, \ldots, z_{r}).
	\] 
	Then we have 
\begin{equation*}
\begin{split}
&\sum_{\sigma \, \in \, \mathbb{S}_{r}} \; \sum_{\delta_{\scalebox{.85}{$\scriptscriptstyle \sigma(i)$}}  \, = \, \pm 1} 
K_{m}\Big(a_{\scalebox{.95}{$\scriptscriptstyle \sigma(1)$}}^{\delta_{\scalebox{.85}{$\scriptscriptstyle \sigma(1)$}}}\!, \ldots, 
a_{\scalebox{.95}{$\scriptscriptstyle \sigma(r)$}}^{\delta_{\scalebox{.85}{$\scriptscriptstyle \sigma(r)$}}} \!\Big)\\ 
& = \frac{(-1)^{r(r  + 1)\slash 2}}{\left(2\pi \sqrt{-1}\right)^{r}}\!
	\oint \!\cdots \!\oint 
h(\mathrm{z})
\cdot \frac{\prod_{1 \, \le \, i  \, < \, j \, \le \, r}\,  
(z_{\scriptscriptstyle i} - z_{\scriptscriptstyle j})^{\mathrm{e}_{\scalebox{.65}{$\scriptscriptstyle ij$}}}
(1 - z_{\scriptscriptstyle i} z_{\scriptscriptstyle j}) \, \cdot \, \prod_{1 \, \le \, k  \, \le \, l \, \le \, m}
(1 - z_{\scriptscriptstyle k} z_{\scriptscriptstyle l}) 
\prod_{i = 1}^{m} z_{\scriptscriptstyle i}^{\scriptscriptstyle r - m}} 
{\prod_{i,\, j = 1}^{r} \big(1 - z_{\scriptscriptstyle i}a_{\scriptscriptstyle j}\big)
\big(1 - z_{\scriptscriptstyle i}a_{\scriptscriptstyle j}^{\scriptscriptstyle -1}\big)} 
\frac{d \mathrm{z}}{\mathrm{z}^{\scriptscriptstyle r}}	
\end{split}
\end{equation*} 
where $\mathrm{e}_{\scalebox{.85}{$\scriptscriptstyle ij$}} = 1$ or $2$ according as $i\le m$ and $j\ge m+1$ or not, 
$\mathrm{z}^{r} : = z_{\scalebox{.98}{$\scriptscriptstyle 1$}}^{r} \cdots\,  z_{r}^{r},$ and where 
each path of integration encloses the $a_{\scriptscriptstyle j}^{ \scriptscriptstyle \pm 1},$ but not 
$
z_{\scriptscriptstyle j} = 0.
$ 
\end{lem}
\begin{proof} The same idea as in the proof of \cite[Lemma 7.5]{DT}.
\end{proof}
The argument in \cite[Prop. 7.7]{DT} gives the following:
\vskip5pt
\begin{cor} --- When $n$ is even, $I_{n, \, \zeta}(D,u)$ is a polynomial
  of degree 7 in $D$ with leading term
  \[
    \frac{D^{\sss 7}}{7!n^{\sss 7}}R_{n}(\zeta^{-2}u^{-2\slash n})
    f_{n,\, \zeta}^{a_{\scalebox{.70}{$\scriptscriptstyle D$}}}(\underline{1};u).  
\]
\end{cor}

\begin{rem}\label{r9.3}
  We omitted providing formulas for the functions 
  $R_{n}(\vr)$ in Lemmas~\ref{lem_odd} and~\ref{lem_even},
  when $n \ge 3.$\linebreak 
Instead, we will show in the next subsection that the leading term of 
$Q_{n}(D,u^{\sss 2})$ can be solely expressed in\linebreak
terms of $R_{\sss 1}(\vr)$ or $R_{\sss 2}(\vr)$ (according as $n$ is odd or
even), and of the matrix $B(\xx; u)$ in Theorem~\ref{Key-ingredient}.
\end{rem}

\subsubsection{The leading term of $Q_{n}(D, q)$}
The leading term of $Q_{n}(D,q)$, as a polynomial in $D$, is
\[
 \frac{D^{\sss 10}}{2^{\sss 4} n^{\sss 10}} \prod_{j=0}^{3}
 \frac{(2j)!}{(4+j)!} S_n(D,\sqrt q)
\;\;\; \text{or}\;\;\;
\frac{D^{\sss 7}}{7!n^{\sss 7}} S_n(D,\sqrt q)
\]
according as $n$ is odd or even. Here we set
\[
S_{n}(D,u):=\frac{1}{n}\sum_{\zeta\in
  \mu_{\scalebox{.75}{$\scriptscriptstyle 2n$}}\slash\{\pm 1\}}
\zeta^{\scalebox{1.1}{$\scriptscriptstyle D$}} 
R_{n}(\zeta^{-2}u^{-2\slash n}) 
f_{w_{\scalebox{.9}{$\scriptscriptstyle \a$}}}^{a_{\scalebox{.70}{$\scriptscriptstyle D$}}}\!
(\underline{1},\zeta^{ -1}u^{ -1\slash n};u)
\]
where $\a\in \Phi_{n}$, $w_{\a}\in W$, and the
functions $R_{n}(\varrho)$ are those occurring
in Lemmas~\ref{lem_odd} and~\ref{lem_even}. Recall that
$a_{\scalebox{.95}{$\scriptscriptstyle D$}}\in\{\mathcal{e},
\mathcal{o}\}$ stands for the parity of $D$,
and that $f^{\mathcal{e}}=f^{\eee}+f^{\oe}$, $f^{\mathcal{o}}=f^{\eo}$
for $f\in\CC(\xx,u)_{\sss 0}$.

In what follows, we show that the functions
  $S_{n}(D, \sqrt q)$ are non-zero for all $n\ge 1$
  and all $D$, so that all $Q_{n}(D,q)$ are present in the
  asymptotic formula in Theorem~\ref{asymptotic}.
  In fact, we prove the stronger result that
  $S_{n}(D, \sqrt q)$ is given by a power series in $q^{-1\slash 2n}$,
  whose coefficients are either all positive or all negative,
  depending on the residue of $n$ modulo $4$.

  We start by giving an expression for $S_{n}(D,u)$
  involving only the term for $\zeta=1$ in the summation.
  \!To do so, we need a property of the function $f_{w}=(1+ux_{\sss 5})\| w$.
\vskip5pt
\begin{lem} \label{L9.11} --- For all $w\in W$, the functions
  $f_{w}^{\eee}$, $f_{w}^{\eo}$ and $f_{w}^{\oe}$ are even, odd and
  odd, respectively, as func-\linebreak tions of~$u$.
\end{lem}
\begin{proof}
  We have
  \be\label{eq41}
  \ov{f_{w}}(\xx; u)=\tL_{w}(\xx; u)
  \cdot \, \!^t(1, u\xx^{w^{\scalebox{.75}{$\scriptscriptstyle -1$}}
    \a_{\scalebox{.75}{$\scriptscriptstyle 5$}}}\!,0).
  \ee
  As functions of $u$, the entries of $\tL_{w}(\xx;u)$ have the same parities 
  as the corresponding entries of the matrix
  $
  \left(\begin{smallmatrix} 1& u& u\\
   u& 1& 1\\
    u& 1& 1
   \end{smallmatrix}\right)\!. 
$
This follows easily by induction on the length $\ell(w)$ of $w$; the
case $\ell(w)=1$ is clear from the formulas of $\Lam_{\sss 1}$ and
$\Lam_{\sss 2}$ in Section~\ref{s4.1}, and the product of such matrices
is again of the same type, so the co-\linebreak cycle relation gives the induction
step. Our assertions now follow from~\eqref{eq41}. 
\end{proof}
For $c\in\Z$, let $\UU_{n,\, c}$ be the operator acting on {L}aurent
series $f(\vr)=\sum_{k \, \in\, \Z} n_{\sss k} \vr^{k} $ with $n_{\sss k} \in \CC$, 
%in the variable $\vr=u^{\sss -2\slash n}$ 
defined by\linebreak
\[
\UU_{n,\,c}(f)\;
:=\; \sum_{\substack{k \, \equiv \, c\!\!\!\!\pmod{n}}} n_{\sss k} \vr^{k}
=\;\frac{1}{n}\sum_{\zeta\in \mu_{\scalebox{.75}{$\scriptscriptstyle n$}}}\zeta^{c} f(\zeta^{\sss -1} \vr).
\]
We will apply this operator to Laurent series $f$ which
converge absolutely for $0<|\vr|<1$, when it is well-defined and the equality above clearly holds. 
Note that $\UU_{n,\,c}$ depends only on~$c$ modulo $n$. 
%We omit the 
%dependence on~$n$ from the notation since $n$ is fixed throughout.
\vskip5pt
\begin{prop}\label{p9.12} ---
  With $\a$, $w_{\a}$ and $R_{n}(\vr)$ from Lemma~\ref{lem_odd}
  (resp. Lemma~\ref{lem_even}) for $n$ odd (resp. $n$ even),
  we\linebreak have:
  \[
S_{n}(D, u) = \UU_{n,\,{\sss D\slash 2}}\!\left[R_{n}(u^{ -2\slash n}) \cdot
f^{\eee}_{w_{\a}}(\u1, u^{ -1\slash n} ;u)\right]+
u\UU_{n,\, {\sss D\slash 2}}\!\left[R_{n}(u^{ -2\slash n} ) \cdot 
u^{\sss -1}f^{\oe}_{w_{\a}}(\u1, u^{ -1\slash n} ;u) \right]
\]
if $D$ is even, and
\[
S_{n}(D, u) = u^{ 1-1\slash n}\,
\UU_{n,\, {\sss \lfloor D \slash 2\rfloor}}\!\left[R_{n }(u^{ -2\slash n})\cdot
  u^{ -1+1\slash n}f^{\eo}_{w_{\a}}
  (\u1,u^{ -1\slash n} ;u) \right]
\]
if $D$ is odd. Here we view the functions between brackets
as Laurent series in $\vr=u^{ -2\slash n}$, by the previous lemma.
\end{prop}
\begin{proof} By Lemma~\ref{L9.11}, we have that
  \[
    f^{\eee}_{w_{\scalebox{.9}{$\scriptscriptstyle \a$}}}(\underline{1}, x_{\sss 5}; u)
    = g_{\sss 1}(x_{\sss 5}^{\sss 2} , u^{\sss 2}),\;\;\;\, 
    f^{\eo}_{w_{\scalebox{.9}{$\scriptscriptstyle \a$}}}(\underline{1}, x_{\sss 5}; u)
    = ux_{\sss 5} g_{\sss 2}(x_{\sss 5}^{\sss 2} , u^{\sss 2}) \;\;\;\, \text{and}\;\;\;\, 
    f^{\oe}_{w_{\scalebox{.9}{$\scriptscriptstyle \a$}}}(\underline{1}, x_{\sss 5}; u) =
    u g_{\sss 3}(x_{\sss 5}^{\sss 2} , u^{\sss 2})
  \]
  with $g_{\sss i}\in \Q(x,u)$. Thus, assuming $D$ to be even
  (the case $D$ odd being similar), we have the following
  ex-\linebreak pression for $S_{n}(D,u)$:
  \[
\begin{aligned}
S_{n}(D, u)= S_{n}(D, \vr^{-n\slash 2})&
  = \frac{1}{n}\sum_{\zeta\in\mu_{\scalebox{.75}{$\scriptscriptstyle 2n$}}
    \slash \{\pm 1\}}\zeta^{\scalebox{1.1}{$\sss D$}}
  R_{n}(\vr\slash \zeta^{2})\big(g_{\sss 1}(\vr\slash \zeta^{2}, \vr^{-n})
  + ug_{\sss 3}(\vr\slash \zeta^{2}, \vr^{-n})\big)\\
  & = \UU_{n,\, {\sss D\slash 2}}[R_{n}(\vr) g_{\sss 1}(\vr, \vr^{-n})]
  + u\UU_{n,\, {\sss D\slash 2}} [R_{n}(\vr) g_{\sss 3}(\vr, \vr^{-n})].
\end{aligned}
\]
Note that the functions between brackets converge absolutely as
Laurent series in $\vr$ for $0<|\vr|<1$,
by Lemmas~\ref{lem_odd} and~\ref{lem_even}, so the last
equality holds. Switching back to the variable $u$
completes the proof.
\end{proof}
This proposition yields explicit formulas for the leading term
of $Q_{n}(D,q)$, which, for small $n$, can be used\linebreak
to check that this term is given by a non-zero power series
with coefficients of the same sign. To prove this\linebreak
property for all~$n$, in Corollary~\ref{C9.14} below, we shall
give different expressions for the functions between brackets
in the proposition, using the extra functional equation
introduced in Section~\ref{s4.2}.

Recall the extension of the cocycle $\tL$ to
$W\oplus \Z$ from the beginning of Section~\ref{s6.3}.
In particular, we have\linebreak 
\[
  \tL_{\tau^{-1}}(\xx;u)=\tB(\xx;u) = B(\xx;u)\slash
  (1-u^{\sss 2}x^\dd)^{\scalebox{1.1}{$\sss 2$}}
\]
and the vector function
$
\mathbf{\tilde{Z}}=\ov{\tilde{Z}_{\sss W}}
$
satisfies the functional equation~\eqref{200}, with the cocycle
$\tL$ in place of~$\Lam$.\linebreak

\vskip-14.5pt
Consider the vector version of the residue in Lemma~\ref{L6.2}:
\[
\vC_{\alpha,\, \scalebox{1.1}{$\scriptscriptstyle \zeta$}}(\underline{x};u):=
 \hskip-44pt \lim_{\hskip44pt x_{\scalebox{.75}{$\scriptscriptstyle 5$}}\to
  \zeta^{\scalebox{1.0}{$\scriptscriptstyle -1$}}(u
 \mathbf{x}^{\alpha \scalebox{.75}{$\scriptscriptstyle '$}})^{\scalebox{1.0}{$\scriptscriptstyle -1$}\slash n}}
\!\left(1 -\zeta u^{1\slash n}
\mathbf{x}^{\alpha \slash n}\right)
\tZZ(\xx;u)
\]
for any
$
\a=n\a_{\sss 5}+\a\scalebox{1.1}{$\scriptscriptstyle '$}\in\Phi_{n}
$
and
$\zeta\in\mu_{\scalebox{1.1}{$\scriptscriptstyle 2n$}}$.
The residue vector
$
\vC_{\alpha,\,  \scalebox{1.1}{$\scriptscriptstyle \zeta$}}(\underline{x};u)
$
appears in the proof of Lemma~\ref{L6.2}, where it is shown that
\be\label{49}
\vC_{\alpha,\, \scalebox{1.1}{$\scriptscriptstyle \zeta$}}(\underline{x};u)=
\frac{1}{2n} \left. \!\left\{ R(w_\a \ux; u) \ov{f_{w_\a}}(\xx;u)
\right\}\right\vert_{x_{\scalebox{.75}{$\scriptscriptstyle 5$}}
        \,=\,\zeta^{\scalebox{1.0}{$\scriptscriptstyle -1$}} (u \mathbf{x}^{\alpha\scalebox{.9}{$\scriptscriptstyle '$}
        })^{\scalebox{1.0}{$\scriptscriptstyle -1$}\slash n}}
      \ee
      with $w_{\a}\in W$ such that $w_{\a} \a =\a_{\sss 5}$.

      Assuming now that~$\a$ is chosen as in
      Lemmas~\ref{lem_odd} and~\ref{lem_even}, we give a formula for
      $
      \vC_{\alpha,\,  \scalebox{1.1}{$\scriptscriptstyle
          \zeta$}}(\underline{x};u)
      $
      in terms of the above extended cocycle.
      \vskip5pt
      \begin{lem}\label{L9.13} --- a). If $n=2k+1$ and
        $\a=\a_{\sss 5}+k\dd\in\Phi_n$, then
\[
\vC_{\alpha,\, \scalebox{1.1}{$\scriptscriptstyle \zeta$}}(\underline{x};u)
          = \frac{1}{2n}
          \left. \!\left\{ R(\ux; u\xx^{k\dd})
              \tL_{\tau^{-k}}(\xx;u) \cdot v_{\sss 0}\right\}
        \right\vert_{x_{\scalebox{.75}{$\scriptscriptstyle 5$}}
        \,=\,\zeta^{\scalebox{1.0}{$\scriptscriptstyle -1$}} (u \mathbf{x}^{\alpha\scalebox{.9}{$\scriptscriptstyle '$}
        })^{\scalebox{1.0}{$\scriptscriptstyle -1$}\slash n}}
    \]
    where $v_{\sss 0}=\, \!^t(1,1,0)$.

    b). If $n=2k+2$ and $\a=-\a_{\sss 4}+(k+1)\dd \in\Phi_{n}$, then
\[
\vC_{\alpha,\, \scalebox{1.1}{$\scriptscriptstyle \zeta$}}(\underline{x};u)
= \frac{1}{2n}
\left.\!\left\{R(\ss_{\sss 4}t\ux; u\xx^{k\dd})
    \tL_{\tau^{-k}}(\xx;u) 
     \cdot \ov{f_{\ss_{\scalebox{.67}{$\scriptscriptstyle 4$}} t}}(\xx; u\xx^{k\dd}) \right\}
        \right\vert_{x_{\scalebox{.75}{$\scriptscriptstyle 5$}}
        \,=\, \zeta^{\scalebox{1.0}{$\scriptscriptstyle -1$}}
        (u \mathbf{x}^{\alpha\scalebox{.9}{$\scriptscriptstyle '$}
        })^{\scalebox{1.0}{$\scriptscriptstyle -1$}\slash n}}
\]
where $t=\ss_{\sss 1}\ss_{\sss 2}\ss_{\sss 3}\ss_{\sss 4} \ss_{\sss 5}$.
\end{lem}
\begin{proof}
  We use the functional equation~\eqref{200}, written for
  $\tZZ$ and the cocycle $\tL$, to compute the
  residue vector
  $
  \vC_{\alpha,\, \scalebox{1.1}{$\scriptscriptstyle
      \zeta$}}(\underline{x};u)
  $.
  In part a), we take the element $\tau^{-k}$ in~\eqref{200},
  while in part b) we take the element $\ss_{\sss 4}t\tau^{-k}$.
  Note that by Theorem~\ref{Key-ingredient} and the cocycle relation,
  the matrix $\tL_{\tau^{-k}}(\xx;u)$ has poles only if
  $u\xx^{\b}=\pm 1$ for some
  $
  \b\in\Phi_{\scalebox{.95}{$\scriptscriptstyle \mathrm{re}$}}^{\scalebox{.95}{$\scriptscriptstyle +$}}
  $
  with $\b<k\dd$. Since $\a>k\dd$, it can be evaluated at 
  $
  x_{\sss 5}\,=\, \zeta^{\sss -1}(u \mathbf{x}^{\alpha\scalebox{.9}{$\scriptscriptstyle '$}})^{\scalebox{1.0}{$\scriptscriptstyle -1$}\slash n}
  $,
  that is, when $u\xx^{\a}=\pm 1$. 
\end{proof}
The decompositions in Lemmas~\ref{lem_odd} and~\ref{lem_even}
have the following analogues: 
\vskip5pt
\begin{lem}\label{L9.14} --- a). For $n=2k+1\ge 1$
  and $\a=\a_{\sss 5}+k\dd\in\Phi_{n}$, we have 
  \[
    R(\ux; u\xx^{k\dd}) 
    \vert_{x_{\scalebox{.75}{$\scriptscriptstyle 5$}}
        \,=\, \zeta^{\scalebox{1.0}{$\scriptscriptstyle -1$}}(u
        \mathbf{x}^{\alpha\scalebox{.9}{$\scriptscriptstyle
            '$}})^{\scalebox{1.0}{$\scriptscriptstyle -1$}\slash n}}
      = R_{\sss 1,\,1}\!\left(\ux;u\xx^{k\dd}\eval\right)
        \, \cdot \, \prod_{1\le i\le j\le 4} \frac{1}{1-x_{\sss i} x_{\sss j}}
      \]
      where the first factor in the right-hand side becomes 
      $R_{\sss 1}(u^{-2\slash n}\slash \zeta^{2})$ when evaluated
      at $\ux=\u1$. Here $R_{\sss 1,\,1}(\ux;u)$ and
      $R_{\sss 1}(\vr)$ are the same functions as in Lemma~\ref{lem_odd}. 

      b). For $n=2k+2\ge 2$ and $\a=-\a_{\sss 4}+(k+1)\dd\in\Phi_{n}$,
      we have
      \[
        R(\ss_{\sss 4}t\ux; u\xx^{k\dd})
    \vert_{x_{\scalebox{.75}{$\scriptscriptstyle 5$}}
        \,=\,\zeta^{\scalebox{1.0}{$\scriptscriptstyle -1$}} (u
        \mathbf{x}^{\alpha\scalebox{.9}{$\scriptscriptstyle '$}})^{\scalebox{1.0}{$\scriptscriptstyle -1$}\slash n}}
      = R_{\sss 2,\,1}\!\left(\ux;u\xx^{k\dd}\eval\right)
        \, \cdot \, \frac{1}{1-x_{\sss 4}^{\sss 2}} 
\prod_{j=1}^{3} \frac{1}{(1-x_{\sss 4}x_{\sss j})(1-x_{\sss 4}\slash x_{\sss j})}
\]
where the first factor in the right-hand side becomes
$R_{\sss 2}(u^{-2\slash n}\slash \zeta^{2})$ when evaluated
at $\ux=\u1$. Here $R_{\sss 2,\,1}(\ux;u)$ and $R_{\sss 2}(\vr)$
are the same functions as in Lemma~\ref{lem_even}. 
\end{lem}
\begin{proof} The assertions follow at once from the cases
  $n=1$, $\zeta=1$ and $n=2$, $\zeta=1$ of the two
  aforementioned lemmas, after making the substitution
  \[
u\mapsto  u\xx^{k\dd}\eval=\begin{cases}
    \zeta^{{\scalebox{1.1}{$\scriptscriptstyle 1$}}-n}  (u \mathbf{x}^{\alpha\scalebox{.9}{$\scriptscriptstyle '$}})^{\scalebox{1.1}{$\scriptscriptstyle 1$}\slash n} & \text{if $n$ odd} \\
    \zeta^{{\scalebox{1.1}{$\scriptscriptstyle 2$}}-n}
    x_{\sss 4}^{{\scalebox{1.1}{$\scriptscriptstyle 1-2$}}\slash n}
    u^{{\scalebox{1.1}{$\scriptscriptstyle 2$}} \slash n}   & \text{if $n$ even}.
\end{cases}
\]
Note that $R(\ux;u)$ is even in $u$, and so the factors
$\zeta^{n}=\pm 1$ can be ignored in the above substitution. 
\end{proof}
We now give simpler formulas for the functions inside
brackets in Proposition~\ref{p9.12}, in which the dependence
upon $n$ appears only in the cocycle matrix
$
\tL_{\tau^{-\lfloor (n-1)\slash 2 \rfloor}}
$.
We express the arguments in terms of
$
\vr=
$
$u^{{\scalebox{1.1}{$\scriptscriptstyle -2 $}}\slash n}$,
choosing the square root
$
\vr^{\scalebox{1.1}{$\scriptscriptstyle  1\slash 2$}}
:=u^{{\scalebox{1.1}{$\scriptscriptstyle -1$}}\slash n}
$. 
\vskip5pt
\begin{cor} \label{C9.14} --- Let $\a\in \Phi_{n}$,
  $w_{\a}\in W$ and $R_{n}(\vr)$ be as in Lemma~\ref{lem_odd},
  for $n$ odd, or as in Lemma~\ref{lem_even}, for $n$\linebreak
  even. \!Then:

  a). If $n=2k+1$, we have
  \[
    R_{n}(\vr)\ov{f_{w_{\a}}}(\u1, \vr^{\sss 1\slash 2} ; \vr^{\sss -n\slash 2})=
R_{\sss 1}(\vr) \tL_{\tau^{-k}}(\u1, \vr^{\sss 1\slash 2} ; \vr^{\sss -n\slash 2 })\cdot v_{\sss 0}
\]
where we recall that $R_{\sss 1}(\vr)=\qpoc{\vr}{\vr}^{\sss -11}$.

b). If $n=2k+2$, we have
\[
R_{n}(\vr)\ov{f_{w_{\a}}}(\u1, \vr^{\sss 1\slash 2} ; \vr^{\sss -n\slash 2 })=
R_{\sss 2}(\vr) \tL_{\tau^{-k}}(\u1, \vr^{\sss 1\slash 2} ; \vr^{\sss -n\slash 2 })\cdot 
\ov{f_{\ss_{\scalebox{.67}{$\scriptscriptstyle 4$} } t}}(\u1, \vr^{\sss 1\slash 2} ; \vr^{\sss -1 })
\]
where we recall that
$
R_{\sss 2}(\vr)=\qpoc{\vr}{\vr}^{\sss -8}\qpoc{\vr}{\vr^{\sss
    2}}^{\sss -6}(1-1\slash \vr)^{\sss -7}
$.
We also have
\[
  \ov{f_{\ss_{\scalebox{.67}{$\scriptscriptstyle 4$} } t}}(\u1,
  \vr^{\sss 1\slash 2} ; \vr^{\sss -1 })= \, \!
^{^t}\!\!\left(\tfrac{\vr^{4} + 7 \vr^{3} + 13 \vr^{2} + 7 \vr + 1}{\vr^{4}},
 \tfrac{\vr^{3} + 7 \vr^{2} + 7 \vr + 1}{\vr^{7 \slash 2}},
\tfrac{3 \vr^{2} + 7 \vr + 3}{\vr^{3}}\right)\!.
\]
\end{cor}
\begin{proof} The identities follow by expressing
  $
  \vC_{\alpha,\, 1}(\underline{x};u)
  $
  in two ways, using~\eqref{49} and Lemma~\ref{L9.13}. After
  can-\linebreak celling the singular parts, one can then evaluate at
  $\ux=\u1$ by using Lemma~\ref{lem_odd}
  (resp. Lemma~\ref{lem_even}) for $n$ odd\linebreak (resp. \!$n$ even),
  and Lemma~\ref{L9.14}.
\end{proof}
We are now ready to show the following:
\vskip5pt
\begin{prop} \label{p9.14} --- For $n\ge 1$, put
  $
  \vr=q^{\scalebox{1.1}{$\scriptscriptstyle -1\slash$}n}<1
  $.
  Then the leading coefficient of $Q_{n}(D,q)$ is
  \[
    (-1)^{\lfloor n \slash 2 \rfloor } \vr^{3  \lfloor n\slash 2 \rfloor \cdot  \lfloor (n+1)\slash 2 \rfloor   }
    \cdot g_{\sss n,D}(\sqrt{\vr})
  \]
  where $g_{\sss n, D}\ne 0$ is a power series with non-negative
  coefficients depending only upon $\lfloor D\slash 2 \rfloor$
  modulo $n$.
\end{prop}
\begin{proof} \!As before, put
  $
  k=\lfloor (n-1) \slash 2 \rfloor
  $.
  When $k=0$ (i.e., $n = 1, 2$), our assertion can be verified
  directly from\linebreak Proposition~\ref{p9.12} and
  Corollary~\ref{C9.14}; the cocycle $\tL_{\tau^{-k}}$ is the
  identity matrix in this case. If $k>0$, we first observe that the
  entries of the matrix
  $
  \tL_{\tau^{-k}}(\u1, \vr^{\sss 1\slash 2} ; \vr^{\sss -n\slash 2 })
  $
  in Corollary~\ref{C9.14} are of the form
  $
  (-1)^{k}\vr^{3k(n-k)}
  $
  times non-zero power series in~$\vr^{\sss 1\slash 2} $
  with non-negative coefficients. \!Indeed, by the cocycle relation,
  we can write
%\begin{equation} \label{eq: cocyle-pos}
\[\tL_{\tau^{-k}}( \underline{1}, x_{\sss 5} ; 1\slash x_{\sss 5}^{\sss n})= 
\tB(\underline{1},  x_{\sss 5} ; 1\slash x_{\sss 5}^{\sss n}) 
\tB(\underline{1},  x_{\sss 5} ; 1\slash x_{\sss 5}^{\sss n-2})
\cdot \ldots\cdot \tB(\underline{1},  x_{\sss 5} ; 1\slash x_{\sss 5}^{\sss n-2k+2})
\]%\end{equation}
and note that $n-2k+2\ge 3$. Now, for $i\ge 3$, Lemma~\ref{LA.1}
gives the functional equation
\begin{equation} \label{eq: f-eq-positivity}
  \tB(\underline{1},  x_{\sss 5} ; 1\slash x_{\sss 5}^{\sss i})
  =-\, x_{\sss 5}^{\sss 6(i-1)}\cdot 
\, \!^t\!\tB(\underline{1},  x_{\sss 5} ; x_{\sss 5}^{\sss i-2}).
\end{equation}
By Theorem~\ref{Key-ingredient}, the function
$\tB(\u1, x_{\sss 5}; u)$ has poles only if
$ux_{\sss 5}^{j}=\pm 1$ ($j=0,1,2$), and so both sides of
\eqref{eq: f-eq-positivity}\linebreak are defined.
\!Furthermore, by Proposition~\ref{pa.2}, the entries of the matrix in the
right-hand side of \eqref{eq: f-eq-positivity} are non-zero power
series in $x_{\sss 5}$ with positive coefficients, and thus the
assertion about
$
\tL_{\tau^{-k}}(\u1, \vr^{\sss 1\slash 2} ; \vr^{\sss -n\slash 2 })
$
fol-\linebreak lows at once by taking $x_{\sss 5}=\vr^{\sss 1\slash 2}$ and
summing the exponents
% $3(n-1), 3(n-3), \ldots,3( n-2k+1)$
of $\vr$ coming from \eqref{eq: f-eq-positivity}.

The remaining factors in the right-hand sides of the formulas
in Corollary~\ref{C9.14} are also power series with non-negative
coefficients, except for the factor
\[
  (1-1\slash \vr)^{\sss -7} = - \vr^{\sss 7}
  \left(1 + \vr + \vr^{\sss 2} + \vr^{\sss 3} + \cdots \right)^{\sss 7}
\]
in the formula of $R_{\sss 2}(\vr)$. \!Thus, when $n$ is even,
the product
$
- \vr^{\sss 7}
\ov{f_{\ss_{\scalebox{.67}{$\scriptscriptstyle 4$} } t}}(\u1,
  \vr^{\sss 1\slash 2} ; \vr^{\sss -1 })
  $ 
gives an additional factor of $(-\vr)^{\sss 3}$. Finally, to see 
that the series obtained after applying Proposition~\ref{p9.12} 
is not identically zero, one\linebreak just needs to notice that the 
power series of $R_{\sss 1}(\vr)$ and $R_{\sss 2}(\vr)$ contain 
powers $\vr^{\sss j}$ with non-zero coefficients for all residue 
classes $j$ modulo $n$ (coming for example from the expansion 
of $(1-\vr)^{\sss -1}$). This completes the proof.
\end{proof}

\begin{rem} \!The estimate of the power of $\vr$ dividing the leading
  term in the previous proposition is essentially optimal.
  To see this, let us assume that $D$ is even and that
  $
  D\slash 2\equiv 3 \lfloor n\slash 2 \rfloor \cdot  \lfloor
  (n+1)\slash 2 \rfloor \pmod n
  $.
  Then the leading coefficient of the polynomial $Q_{n}(D,q)$,
  expanded as a power series in
  $
  \vr^{\sss 1\slash 2} = q^{\sss -1\slash 2n}
  $,
  is asymptotically
  \[
(-1)^{\lfloor n \slash 2 \rfloor } \vr^{3  \lfloor n\slash 2 \rfloor \cdot  \lfloor (n+1)\slash 2 \rfloor }
(1+O(\vr)) \cdot\begin{cases}
  \frac{1}{2^{\scalebox{.85}{$\scriptscriptstyle 4$}}
    n^{\scalebox{.85}{$\scriptscriptstyle 10$}}} \prod_{j=0}^{3}
 \frac{(2j)!}{(4+j)!} & \text{ if $n$ odd}\\ 
\frac{1}{7!n^{\scalebox{.85}{$\scriptscriptstyle 7$}}} & \text{ if $n$ even}.
\end{cases}
\]
Indeed, by Lemma~\ref{L5.3}, one has $\tB(\xx;0)=\tB_{\sss 0}$, and
the formula follows at once from the argument of the proof of the
previous proposition, combined with Proposition~\ref{p9.12} and the
formulas at the beginning of this subsection.  
\end{rem}

\appendix 
\section{Properties of the matrix $B(\xx; u)$} 
We collect here two properties of the matrix $B(\xx;u)$ in 
Theorem~\ref{Key-ingredient} that were needed in the proofs 
of Corollary~\ref{C5.5} and Proposition~\ref{p9.14}.

From the explicit formula of the matrix 
$
B(\xx;u)=A^{-1}(\xx; u\xx^{\dd})
$, 
with $A$ given in~\cite{eDPP}, one verifies the follow-\linebreak ing 
functional equation. 
\vskip5pt
\begin{lem}\label{LA.1} --- The matrix 
$\tB(\xx;u)=B(\xx;u)\slash (1-u^{\sss 2}x^\dd)^{\sss 2}$ satisfies
\[
^t\!\tB\big(\xx; 1\slash u\xx^\dd\big)=-u^{\sss 6}\xx^{{\sss 3}\dd} \tB(\xx;u). 
\] 
\end{lem} 

The main result of this appendix is the following:
\vskip5pt
\begin{prop}\label{pa.2} --- Each entry of 
$\tB(\underline{1},x_{\sss 5};u)$ is of the form 
\[
\sum_{\substack{a,b,c\ge 0 \\a+b+c=11}}\frac{p_{\sss a,b,c}(x_{\sss 5},u)}
{\left(1-u^{\sss 2}\right)^{\scalebox{1.1}{$\scriptscriptstyle a$}}
\!\left(1-u^{\sss 2}x_{\sss 5}^{\sss
    2}\right)^{\scalebox{1.1}{$\scriptscriptstyle b$}} 
\!\left(1-u^{\sss 2}x_{\sss 5}^{\sss
    4}\right)^{\scalebox{1.1}{$\scriptscriptstyle c$}}} \ne 0 
\] 
where $p_{\sss a,b,c}(x,y)$ are polynomials with non-negative coefficients. 
\end{prop} 
Let $\vr=x_{\sss 5}^{\sss 2}$, $q=u^{\sss 2}$, and set $r=\vr q$. 
We have: 
\[
\tB(\underline{1},x_{\sss 5};u)= \begin{pmatrix}
\frac{b_{\scalebox{.7}{$\scriptscriptstyle
      11$}}}{(q-1)^{\scalebox{.85}{$\scriptscriptstyle
      4$}}(r-1)^{\scalebox{.85}{$\scriptscriptstyle 11$}}
  (\vr r-1)^{\scalebox{.85}{$\scriptscriptstyle 4$}}}   &  \frac{\sqrt{r}\,
  b_{\scalebox{.7}{$\scriptscriptstyle
      12$}}}{(q-1)^{\scalebox{.85}{$\scriptscriptstyle
      4$}}(r-1)^{\scalebox{.85}{$\scriptscriptstyle 11$}} }   &  
\frac{\sqrt{q}\,b_{\scalebox{.7}{$\scriptscriptstyle
      13$}}}{(q-1)^{\scalebox{.85}{$\scriptscriptstyle
      4$}}(r-1)^{\scalebox{.85}{$\scriptscriptstyle 10$}} (\vr
  r-1)^{\scalebox{.85}{$\scriptscriptstyle 4$}}}    \\[8pt]    
\frac{\sqrt{r}\, b_{\scalebox{.7}{$\scriptscriptstyle 21$}}}{(\vr
  r-1)^{\scalebox{.85}{$\scriptscriptstyle
      4$}}(r-1)^{\scalebox{.85}{$\scriptscriptstyle 11$}} } &
\frac{b_{\scalebox{.7}{$\scriptscriptstyle
      22$}}}{(r-1)^{\scalebox{.85}{$\scriptscriptstyle 11$}}} &  
\frac{\sqrt{\vr}\, b_{\scalebox{.7}{$\scriptscriptstyle 23$}}}{(\vr
  r-1)^{\scalebox{.85}{$\scriptscriptstyle
      4$}}(r-1)^{\scalebox{.85}{$\scriptscriptstyle 10$}} } \\[8pt]
  \frac{\sqrt{q}\,b_{\scalebox{.7}{$\scriptscriptstyle
        31$}}}{(q-1)^{\scalebox{.85}{$\scriptscriptstyle
        4$}}(r-1)^{\scalebox{.85}{$\scriptscriptstyle 10$}} (\vr
    r-1)^{\scalebox{.85}{$\scriptscriptstyle 4$}}}  &\frac{\sqrt{\vr}\,
    b_{\scalebox{.7}{$\scriptscriptstyle
        32$}}}{(q-1)^{\scalebox{.85}{$\scriptscriptstyle
        4$}}(r-1)^{\scalebox{.85}{$\scriptscriptstyle 10$}} }  & 
  \frac{b_{\scalebox{.7}{$\scriptscriptstyle
        33$}}}{(q-1)^{\scalebox{.85}{$\scriptscriptstyle
        4$}}(r-1)^{\scalebox{.85}{$\scriptscriptstyle 9$}} (\vr
    r-1)^{\scalebox{.85}{$\scriptscriptstyle 4$}}}                                \end{pmatrix} 
\]
where the numerators $b_{\sss ij}=b_{\sss ij}(\vr,q)$ are explicit 
non-zero polynomials, given in~\cite{eDPP}. The proposition
fol-\linebreak lows immediately from the following decompositions
of $b_{\sss ij}$.
\vskip5pt
\begin{lem}\label{LA.2} --- We have
  $
  b_{\sss 22}(\vr,q)= p_{\sss
    22}(r)
  $,
  where the polynomial $p_{\sss 22}$ has negative coefficients,
  and the other numerators decompose as follows 
  \[
    b_{\sss 11}=\vec{a}_{\sss 8}\cdot \vec{p}_{\sss 11},\quad 
b_{\sss 12}=\vec{e}_{\sss 4} \cdot\vec{p}_{\sss 12},\quad 
b_{\sss 21}=\vec{f}_{\sss 4} \cdot\vec{p}_{\sss 21},
\]
\[
  b_{\sss 13}=(q-1)\vr \vec{a}_{\sss 6}\cdot\vec{p}_{\sss 13}+
            (r-1)^{\sss 3}\vec{a}_{\sss 4}\cdot\vec{q}_{\sss 13}, \quad 
 b_{\sss 31}=(\vr r-1)\vec{a}_{\sss 6}\cdot\vec{p}_{\sss 13}+
            (r-1)^{\sss 3}\vec{a}_{\sss 4}\cdot\vec{q}_{\sss 13},           
\]
\[
  b_{\sss 23}=\vec{f}_{\sss 3}\cdot\vec{p}_{\sss 23}, \quad
  b_{\sss 32}=\vec{e}_{\sss 3}\cdot\vec{p}_{\sss 32}, \quad
  b_{\sss 33}=\vec{a}_{\sss 6}\cdot \vec{p}_{\sss 33} 
\]
in terms of the vectors
$
\vec{a}_{\sss 2k}=\left[(r-1)^{\sss 2k-2i}(q-1)^{\sss i}(\vr r-1)^{\sss i} \right]_{i=0, \ldots, k },$
\[
\vec{e}_{\sss k}=\left[(r-1)^{\sss k-i}(q-1)^{\sss i}\right]_{i=0, \ldots, k },\quad
\vec{f}_{\sss k}=\left[(r-1)^{\sss k-i}(\vr r-1)^{\sss i}\right]_{i=0, \ldots, k }
\]
where $\vec{p}_{\sss ij}$ and $\vec{q}_{\sss 13}$ denote vectors 
of non-zero polynomials with negative coefficients
of the same size as the first vector in the scalar products.
\end{lem}
\begin{proof} The decompositions can be verified using the explicit
  formulas of the vectors $\vec{p}_{ij}$ and $\vec{q}_{\sss 13}$
  in~\cite{eDPP}. By Lemma~\ref{LA.1}, the polynomials
  $b_{ji}(\vr,q)$ can be expressed in terms of
  $b_{ij}(\vr,1/q\vr^{\sss 2})$, and so we only have to consider half
  of the off-diagonal entries of $\tB$. We illustrate how one 
  arrives at these decompositions by considering an example.

  For the entry $b_{\sss 33}$, we notice that
  $
  b_{\sss 33}(\vr,1/\vr^{n})=
  (\vr-1)^{\sss 6}g_{n}(\vr)
  $
  for $n=0,1,2,\ldots$, where $g_{n}$ are Laurent poly-\linebreak
  nomials with negative coefficients. Therefore we look for a
  decomposition involving the factors in the de-\linebreak
  nominator of the $(3,3)$ entry of the form:
\[
b_{\sss 33}(\vr,q)= \sum_{i=0}^3 (r-1)^{\sss 6-2i}(q-1)^{\sss i}
(\vr r-1)^{\sss i} h_{i}(r)
\]
with $h_{i}(r)$ polynomials in $r=\vr q$. (This assumption on
$h_{i}$ and the type of decomposition is suggested by the
symmetry of the diagonal entries under
$q\mapsto 1\slash q \vr^{\sss 2}$.) Under this assumption,
the polynomials $h_{i}$ can be determined recursively, 
by letting $h_{\sss 0}(r)=b_{\sss 33}(r,1)\slash (r-1)^{\sss 6}$,
and repeating the same procedure with $b_{\sss 33}(\vr,q)$ replaced
by
\[
  \frac{b_{\sss 33}(\vr,q)-b_{\sss 33}(r,1)}{(q-1)(\vr r-1)}
\]
etc. All the coefficients of the polynomials $h_{\sss 0}(r)$
and $h_{\sss 3}(r)$ thus found are already negative, and only
two\linebreak of the coefficients of $h_{\sss 1}(r)$ and $h_{\sss 2}(r)$
are positive. \!To eliminate the latter, we replace $h_{\sss 1}(r)$ and
$h_{\sss 2}(r)$ in the above decomposition by
\[
h_{\sss 2}'(r)=h_{\sss 2}(r)-(r-1)^{\sss 2}(r+r^{\sss 7}), \quad 
h_{\sss 1}'(\vr,q)=h_{\sss 1}(r)+(r+r^{\sss 7})(q-1)(\vr r-1)
\]
so that $h_{\sss 2}'(r)$ has now all coefficients negative.
\!In the same way, we replace $h_{\sss 1}'(\vr,q)$ and
$h_{\sss 0}(r)$ by $h_{\sss 1}''(\vr,q)$\linebreak and $h_{\sss
  0}'(\vr,q)$, respectively, so that $h_{\sss 1}''(\vr,q)$ has only negative
coefficients, and find that $h_{\sss 0}'(\vr,q)$ has only\linebreak
negative coefficients as well.

The same method can be used to decompose the entries
$b_{\sss 11}$, $b_{\sss 12}$ and $b_{\sss 32}$ in terms of the
entries of the corresponding vector $\vec{a}_{\sss 2k}$ or
$\vec{e}_{\sss k}$, as in the statement of the lemma. For
$b_{\sss 13}$, we decompose both
\[
  (b_{\sss 13}(\vr,q)\pm b_{\sss 31}(\vr,q))\slash (\vr\pm 1)  
\]
in terms of the entries of $\vec{a}_{\sss 6}$, with coefficients
that are again polynomials of $r$ alone. We recover
$b_{\sss 13}(\vr,q)$ as a linear combination of these decompositions.
\end{proof}

\Addresses

\end{document}